\documentclass[11pt]{article}
\usepackage{latexsym,amsfonts,amssymb,amsmath,amsthm}
\usepackage{graphicx}

\usepackage[usenames,dvipsnames]{color}
\usepackage{ulem}

\begin{document}
\setlength{\baselineskip}{16pt}

\parindent 0.5cm
\evensidemargin 0cm \oddsidemargin 0cm \topmargin 0cm \textheight 22.5cm \textwidth 16cm \footskip 2cm \headsep
0cm

\newtheorem{theorem}{Theorem}[section]
\newtheorem{lemma}{Lemma}[section]
\newtheorem{proposition}{Proposition}[section]
\newtheorem{definition}{Definition}[section]
\newtheorem{example}{Example}[section]
\newtheorem{corollary}{Corollary}[section]

\newtheorem{remark}{Remark}[section]

\numberwithin{equation}{section}

\def\p{\partial}
\def\I{\textit}
\def\R{\mathbb R}
\def\C{\mathbb C}
\def\u{\underline}
\def\l{\lambda}
\def\a{\alpha}
\def\O{\Omega}
\def\e{\epsilon}
\def\ls{\lambda^*}
\def\D{\displaystyle}
\def\wyx{ \frac{w(y,t)}{w(x,t)}}
\def\imp{\Rightarrow}
\def\tE{\tilde E}
\def\tX{\tilde X}
\def\tH{\tilde H}
\def\tu{\tilde u}
\def\d{\mathcal D}
\def\aa{\mathcal A}
\def\DH{\mathcal D(\tH)}
\def\bE{\bar E}
\def\bH{\bar H}
\def\M{\mathcal M}
\renewcommand{\labelenumi}{(\arabic{enumi})}

\def\disp{\displaystyle}
\def\undertex#1{$\underline{\hbox{#1}}$}
\def\card{\mathop{\hbox{card}}}
\def\sgn{\mathop{\hbox{sgn}}}
\def\exp{\mathop{\hbox{exp}}}
\def\OFP{(\Omega,{\cal F},\PP)}
\newcommand\JM{Mierczy\'nski}
\newcommand\RR{\ensuremath{\mathbb{R}}}
\newcommand\CC{\ensuremath{\mathbb{C}}}
\newcommand\QQ{\ensuremath{\mathbb{Q}}}
\newcommand\ZZ{\ensuremath{\mathbb{Z}}}
\newcommand\NN{\ensuremath{\mathbb{N}}}
\newcommand\PP{\ensuremath{\mathbb{P}}}
\newcommand\abs[1]{\ensuremath{\lvert#1\rvert}}

\newcommand\normf[1]{\ensuremath{\lVert#1\rVert_{f}}}
\newcommand\normfRb[1]{\ensuremath{\lVert#1\rVert_{f,R_b}}}
\newcommand\normfRbone[1]{\ensuremath{\lVert#1\rVert_{f, R_{b_1}}}}
\newcommand\normfRbtwo[1]{\ensuremath{\lVert#1\rVert_{f,R_{b_2}}}}
\newcommand\normtwo[1]{\ensuremath{\lVert#1\rVert_{2}}}
\newcommand\norminfty[1]{\ensuremath{\lVert#1\rVert_{\infty}}}
\newcommand{\ds}{\displaystyle}

\title{Dynamics in chemotaxis models of parabolic-elliptic type on bounded
domain with time and space dependent logistic sources}
\author{
Tahir  Bachar Issa and Wenxian Shen  \\
Department of Mathematics and Statistics\\
Auburn University\\
Auburn University, AL 36849\\
U.S.A. }

\date{}
\maketitle

\begin{abstract}
The current paper considers the dynamics of the following chemotaxis system of parabolic-elliptic type with local as well as nonlocal time and space dependent logistic source
$$
\begin{cases}
u_t=\Delta u-\chi\nabla (u\cdot \nabla v)+u\left(a_0(t,x)-a_1(t,x)u-a_2(t,x)\int_{\Omega}u\right),\quad x\in \Omega\cr
0=\Delta v+ u-v,\quad x\in \Omega \quad \cr
\frac{\p u}{\p n}=\frac{\p v}{\p n}=0,\quad x\in\p\Omega,
\end{cases}
$$
 where $\Omega \subset \mathbb{R}^n(n\ge1)$ is a bounded domain with  smooth boundary $\p\Omega$ and  $a_i(t,x)$ ($i=0,1,2$)  are locally H\"older continuous in $t\in\RR$ uniformly with respect to $x\in\bar\Omega$ and continuous in $x\in\bar\Omega$.  We first prove the local existence and uniqueness of  classical solutions $(u(x,t;t_0,u_0),v(x,t;t_0,u_0))$ with $u(x,t_0;t_0,u_0)=u_0(x)$ for various initial functions $u_0(x)$.  Next, under some conditions on the coefficients $a_1(t,x)$, $a_2(t,x)$,  and $\chi$ and the dimension $n$, we prove  the
global existence and boundedness of classical solutions $(u(x,t;t_0,u_0),v(x,t;t_0,u_0))$ with  given nonnegative initial function $u(x,t_0;t_0,u_0)=u_0(x)$.
  Then, under the same conditions for the global existence, we show  that the system has an  entire  positive classical solution $(u^*(x,t),v^*(x,t))$.
  Moreover,  if $a_i(t,x)$ $(i=0,1,2)$ are periodic in $t$ with period $T$ or are independent of $t$, then  the system has a time periodic positive solution $(u^*(x,t),v^*(x,t))$ with periodic $T$  or a steady state positive solution $(u^*(x),v^*(x))$. If $a_i(t,x)$ $(i=0,1,2)$ are independent of $x$ (i.e. are spatially homogeneous), then the system
  has a spatially homogeneous entire positive solution $(u^*(t),v^*(t))$. Finally, under some further assumptions, we prove that the system has a unique
  entire positive solution $(u^*(x,t),v^*(x,t))$ which is globally stable in the sense that for any given $t_0\in\RR$ and $u_0\in C^0(\bar\Omega)$ with $u_0(x)\ge 0$ and $u_0(x)\not\equiv 0$,
  $$
  \lim_{t\to\infty} \Big(\|u(\cdot,t;t_0,u_0)-u^*(\cdot,t)\|_{C^0(\bar\Omega)}+\|v(\cdot,t;t_0,u_0)-v^*(\cdot,t)\|_{C^0(\bar\Omega)}\Big)=0.
  $$
  Moreover, if $a_i(t,x)$ $(i=0,1,2)$ are periodic or almost periodic in $t$, then the unique entire positive solution
  $(u^*(x,t),v^*(x,t))$ is also periodic or almost periodic in $t$.
\end{abstract}

  \medskip


  \medskip


\section{Introduction and the statements of the main results}
\label{S:intro}

Chemotaxis refers to the movement of  living organisms in response to certain chemicals in their environments.  This type of movement  exists in many
 biological phenomena such as bacteria aggregation, immune system response or angiogenesis in the embryo formation and in tumour development. At the beginning of the 1970s,  Keller and Segel  (\cite{KS1970}, \cite{KS71})  introduced some mathematical models to describe the aggregation of certain types of bacteria. Since then,   a variety of mathematical models to describe chemotaxis have been proposed. Systems with chemotactic terms have been used to model not only the mentioned biological processes at the microscopic scale but also population dynamics at the macroscopic scale in the context of life sciences, 'gravitational collapse' in astrophysics, material sciences, etc.  A large amount of research has been carried out toward various central problems in chemotaxis models, including
 global existence of classical/weak solutions with given initial  data;  finite-time blow-up;  pattern formation;  existence, uniqueness, and stability of certain special solutions;  etc. (see,  for instance, \cite{JIDTN1995, FTJ02, GZ1998, H03, HW2001,  HW05, ST2005, ST2001, STW13,  T.Wang2016}, etc.).

Consider   chemotaxis systems for the time evolution of the densities of one species and one chemoattractant.
According to the nature of the equation  satisfied by the chemoattractant,  such  systems  can be classified as  parabolic-parabolic systems,  parabolic-elliptic systems,  or parabolic-ODE systems.  Many interesting dynamical scenarios are observed in such chemotaxis systems.
 For example, it is observed that chemotactic cross-diffusion may exert a strongly destabilizing action  in the sense that finite-time blow-up might occur in the following system,
\begin{equation}
\label{u-v-eq0-1}
\begin{cases}
u_t=\Delta u-\chi\nabla (u\cdot \nabla v),\quad x\in \Omega \cr
\tau v_t=\Delta v+ u-v,\quad x\in \Omega \cr
\frac{\p u}{\p n}=\frac{\p v}{\p n}=0, \quad x\in\p \Omega,
\end{cases}
\end{equation}
 (see \cite{HMVJ1996, HVJ1997b, JaW92, NT1995, NT2001} for the parabolic-elliptic case (i.e. $\tau=0$) and \cite{HVJ1997a}, \cite{win_JMAA_veryweak}, \cite{Win11} for the fully parabolic case (i.e. $\tau>0$)).
 It is observed that logistic-type sources may suppress such blow-up phenomena to some
extent. More precisely, consider
\begin{equation}
\label{u-v-eq0-2}
\begin{cases}
u_t=\Delta u-\chi\nabla (u\cdot \nabla v)+u(a-bu),\quad x\in \Omega \cr
\tau v_t=\Delta v+ u-v,\quad x\in \Omega \cr
\frac{\p u}{\p n}=\frac{\p v}{\p n}=0, \quad x\in\p \Omega.
\end{cases}
\end{equation}
It is shown that
 solutions of \eqref{u-v-eq0-2} with positive initial functions  exist globally in time
provided that $b$ is sufficiently large relative to $\chi$
 (see e.g. \cite{lankeit_exceed, NTa15, TW07, TW12, Vig, Win10}).
 Also, quite rich dynamical features, including spatial pattern formation and spatio-temporal chaos,
 may be observed in chemotaxis models, including \eqref{u-v-eq0-2}, at least numerically (see \cite{kuto_PHYSD, PaHi}).
  The reader is referred to \cite{ChTe, KhBa,  KuTe, lankeit_eventual, TaWi15, TJiCMAJ, T04, win_JNLS}, etc. for the analytic studies of chemotaxis models
   with logistic sources. However, many central problems for chemotaxis models with logistic sources, including the possible
   occurrence of finite time blow up when the logistic damping is not large relative to the chemotaxis sensitivity,
   are not well understood yet. In particular, there is little study on chemotaxis systems with time and space dependent logistic sources.

 In reality, the environments of many living organisms are spatially and temporally heterogeneous. It is of both biological and mathematical interests to study chemotaxis models with certain time and space dependence.
In the present paper,  we consider the following  chemotaxis system of parabolic-elliptic type  with both local and nonlocal heterogeneous logistic source,
\begin{equation}
\label{u-v-eq1}
\begin{cases}
u_t=\Delta u-\chi\nabla (u\cdot \nabla v)+u(a_0(t,x)-a_1(t,x)u-a_2(t,x)\int_{\Omega}u),\quad x\in \Omega \cr
0=\Delta v+ u-v,\quad x\in \Omega \cr
\frac{\p u}{\p n}=\frac{\p v}{\p n}=0, \quad x\in\p \Omega,
\end{cases}
\end{equation}
where $\Omega$ is a bounded subset of $\mathbb{R}^n$ with smooth
boundary,   $u(x,t)$ and $v(x,t)$ represent  the population densities   of  living
organisms and   some chemoattractant substance, respectively, $\chi \in \mathbb{R}$ is
the chemotactic sensitivity,   $a_0,a_1$ are nonnegative bounded functions and $a_2$ is a bounded real valued function.

System $\eqref{u-v-eq1}$ with constant
coefficients for both one and two species, was introduced recently in
\cite{NT13} by Negreanu and Tello. As mentioned in \cite{NT13},
the logistic growth describes the competition of the individuals of
the species for the resources of the environment and the cooperation
to survive. The coefficient $a_0$ induces an exponential growth for
low density populations and the term $a_1u$ describes a local competition of the species. At the time that the population grows, the
competitive effect of the local term $a_1u$  becomes more influential. The non-local term $a_2\int_{\Omega}u$ describes the influence of the total mass of the species in the growth of the population. If $a_2>0,$ we have a competitive term which limits such growth and when $a_2<0$ the individuals cooperate globally to survive. In the last case, the individuals compete locally but cooperate globally and the effects of $a_1u$ and $a_2\int_{\Omega}u$ balance the system.  Note that $(u,v)\equiv
(0,0)$ is always a solution of \eqref{u-v-eq1}, which will be called
the {\it  trivial solution} of \eqref{u-v-eq1}. Due to the
biological reason, we are only interested in nonnegative solutions,
in particular, nonnegative and nontrivial solutions, of
\eqref{u-v-eq1}.

In the case that the chemotaxis and nonlocal competition are absent
(i.e. $\chi=0$ and  $a_2\equiv 0$) in \eqref{u-v-eq1}, the population density $u(x,t)$ of the living organisms
satisfies
the following scalar reaction diffusion equation,
\begin{equation}
\label{fisher-eq}
\begin{cases}
u_t=\Delta u+u(a_0(t,x)-a_1(t,x)u),\quad x\in \Omega \cr
\frac{\p u}{\p n}=0, \quad x\in\p \Omega.
\end{cases}
\end{equation}
Equation \eqref{fisher-eq} is called Fisher or KPP type equation in literature because of the pioneering works by Fisher (\cite{Fis}) and Kolmogorov, Petrowsky, Piscunov
(\cite{KPP}) in the special case $a_0(t,x)=a_1(t,x)=1$, and has been extensively studied (see \cite{CaCo}, \cite{HeWe}, \cite{NMN00}, \cite{ShYi}, \cite{QXZ}, etc.).
The dynamics of \eqref{fisher-eq} (and then the dynamics of \eqref{u-v-eq1})
is quite well understood.
For example, if $a_0(t,x)\equiv a_0(t)$
and $a_1(t,x)\equiv a_1(t)$,  it is proved in \cite{NMN00} that
\eqref{u-v-eq1} has a unique bounded entire solution,  that is positive,  does not approach the zero-solution in the past and in the future and attracts  all positive solutions. If $a_0(t,x)$ and $a_1(t,x)$ are positive and almost periodic in $t$, it is proved in \cite{ShYi} that \eqref{fisher-eq} has a
unique globally stable  time almost periodic positive solution.

In the case of constant coefficients with $a_0>0$ and $a_1-|\Omega|(a_2)_->0$, it is clear that $(u,v)\equiv (\frac{a_0}{a_1+a_2|\Omega|},\frac{a_0}{a_1+a_2|\Omega|})$ is
the unique nontrivial spatially and temporally homogeneous steady state solution of \eqref{u-v-eq1}, where $|\Omega|$ is the Lebesgue measure of $\Omega$. It is proved in \cite{NT13}   that
  the condition $a_1>2\chi+\abs{a_2}$ ensures the global stability of the homogeneous steady state   (see \cite{TW07} when $a_2=0$) and  that, if furthermore $a_2=0$,  the assumption $a_1>\frac{n-2}{n}\chi$  ensures the global existence of a unique bounded classical solution  $(u(x,t;t_0,u_0),v(x,t;t_0,u_0))$ with given nonnegative  initial function
  $u_0\in C^{0,\alpha}(\bar\Omega)$ (i.e. $u(x,t_0;t_0,u_0)=u_0(x)\ge 0$) (see \cite{TW07}).
It should be pointed  that,  when $n\ge 3$ and $a_1\le \frac{n-2}{n}\chi$ ($\chi>0$ and $a_2=0$),  it  remains open whether for any
 given nonnegative initial function $u_0\in C^{0,\alpha}(\bar\Omega)$,  \eqref{u-v-eq1}
 possesses a global classical solution $(u(x,t;t_0,u_0),v(x,t;t_0,u_0))$ with
 $u(x,t_0;t_0,u_0)=u_0(x)$, or whether finite-time blow-up occurs for some initial data.  We mention the works \cite{lankeit_exceed},  \cite{win_JMAA_veryweak},  \cite{win_JNLS}  along this direction.
 It is shown in \cite{lankeit_exceed}, \cite{win_JNLS} that in presence of suitably weak logistic dampening
(that is, small $a_1$)  certain transient growth phenomena do occur for some
initial data. It is shown in \cite{win_JMAA_veryweak} that replacing $a_1 u$ by $a_1 u^\kappa$ with  suitable
$\kappa<1$ (for instance, $\kappa=1/2$) and replacing $u-v$ by $u-\frac{1}{|\Omega|}\int_\Omega u(x)dx$, then  finite-time blow-up is possible.

 However, as far as $\chi\not =0$ and $a_0(t,x),a_i(t,x),a_2(t,x)$ are not constants, there is little study of \eqref{u-v-eq1}.
 The objective of the present paper is to investigate the asymptotic dynamics of \eqref{u-v-eq1}.
 To this end, we first study the local and global existence of classical solutions of \eqref{u-v-eq1} with given nonnegative initial functions,
 next study the existence  of entire positive solutions, and then investigate the uniqueness and stability of entire positive solutions and
  the asymptotic behavior of  positive solutions
of \eqref{u-v-eq1}. Throughout this paper, we assume that $a_i$ $(i=0,1,2)$ satisfy the following standing assumption.

\medskip
\noindent {\bf (H1)} {\it $a_0(t,x)$, $a_1(t,x)$ and $a_2(t,x)$ are
\text{ H\"older}  continuous  in $t\in\RR $ with exponent $\nu>0$ uniformly with respect to $x\in\bar\Omega$, continuous in $x\in\bar\Omega$ uniformly
with respect to $t\in\RR$, and there are nonnegative constants $\alpha_i$, $A_i$ $(i=0,1,2)$ with $\alpha_1+\alpha_2>0$  such that
\begin{equation}
\begin{cases}
\label{conditions-on-a-i-eq}
0< \alpha_0\leq a_0(t,x) \leq A_0\cr
0\le \alpha_1\leq a_1(t,x) \leq A_1\cr
0\le \alpha_2\leq \abs{a_2(t,x)} \leq A_2.
\end{cases}
\end{equation}
}

 A vector valued function $(u(x,t),v(x,t))$ is called a {\it classical solution} of \eqref{u-v-eq1} on
$\Omega\times (t_1,t_2)$ ($-\infty\le t_1<t_2\le\infty$)  if $(u,v)\in C(\bar \Omega\times (t_1,t_2))\cap C^{2,1}(\bar \Omega\times (t_1,t_2))$ and satisfies \eqref{u-v-eq1} for $t\in (t_1,t_2)$ in the classical sense. A classical solution  $(u(x,t),v(x,t))$ of \eqref{u-v-eq1} on
$\Omega\times (t_1,t_2)$ is called {\it nonnegative} if $u(x,t)\ge 0$ and $v(x,t)\ge 0$ for $(x,t)\in\bar\Omega\times (t_1,t_2)$,
and is called {\it positive} if $\inf_{(x,t)\in\bar\Omega\times (t_1,t_2)} u(x,t)>0$ and $\inf_{(x,t)\in\bar\Omega\times (t_1,t_2)} v(x,t)>0$.
  $(u(x,t),v(x,t))$ is called an {\it entire classical solution} of \eqref{u-v-eq1}  if it is a classical solution of \eqref{u-v-eq1}
on $(-\infty,\infty)$.
For a given $t_0\in\RR$ and a given function $u_0(\cdot)$ on $\Omega$, it is said that  \eqref{u-v-eq1} has a {\it  classical solution with initial condition $u(x,t_0)=u_0(x)$}  if \eqref{u-v-eq1}  has a
classical solution, denoted by $(u(x,t;t_0,u_0),v(x,t;t_0,u_0))$,  on  $(t_0,T)$ for some $T>t_0$  satisfying that  $\lim_{t\to t_0+}u(\cdot,t;t_0,u_0)=u_0(\cdot)$ in certain sense.
A classical solution of \eqref{u-v-eq1} with initial condition $u(x,t_0)=u_0(x)$ exists globally if \eqref{u-v-eq1} has a classical solution $(u(x,t;t_0,u_0),v(x,t;t_0,u_0))$ with $u(x,t_0;t_0,u_0)=u_0(x)$ on $(t_0,\infty)$.

For given $1\le p<\infty$,  let $X=L^{p}(\Omega)$
  and $A=-\Delta+I$ with
  $$D(A)=\left\{ u \in W^{2,p}(\Omega) \, |  \, \frac{\p u}{\p n}=0 \quad \text{on } \, \p \Omega \right\}.
  $$
It is well known that  $A$ is a sectorial operator in $X$  (see, for example,  \cite[Example 1.6]{DH77}) and thus generates an analytic semigroup $\left(e^{-At}\right)_{t\geq 0}$ in $X$ (see, for example, \cite[Theorem 1.3.4]{DH77}). Moreover $0 \in \rho(A)$ and
$$
\|e^{-A t}u\|_{X}\le  e^{-t}\|u\|_{X}\quad {\rm for}\quad t\ge 0 \,{\rm and }\, u \in X,
 $$
Because $\Delta$ is  dissipative operator and $range(I-\Delta)=X,$ so it generates a strongly continuous semigroup of contraction on $X.$

 Let $X^{\alpha}=D(A^{\alpha})$ equipped with the graph norm  $\|u\|_{\alpha}:=\|u\|_{X^\alpha}=\|A^{\alpha}u\|_{p}$ (see, for example,
  \cite[Definition 1.4.7]{DH77}).

 Throughout this paper, $A$ and $X^\alpha$ are defined as in the above. For given $-\infty\le t_1<t_2\le \infty$ and $0\le \delta<1$, $C^\delta((t_1,t_2), X^\alpha)$ is the space of all locally H\"older continuous functions from $(t_1,t_2)$ to $X^\alpha$ with exponent $\delta$.
First of all, we have the following local existence theorem.

\begin{theorem}
\label{thm-001} Suppose that $p>1$ and $1/2 < \alpha<1$ are such that $X^{\alpha}\subset C^1(\bar{\Omega}).$
\begin{itemize}
\vspace{-0.05in}\item[(1)] For any
   $t_0 \in \mathbb{R}$  and $u_0 \in X^{\alpha}$ with $u_0 \geq 0$,  there exists $T_{\max} \in (0,\infty]$  such that $\eqref{u-v-eq1}$  has a unique non-negative classical solution $(u(x,t;t_0,u_0),v(x,t;t_0,u_0))$ on $(t_0,t_0+T_{\max})$   satisfying that $\lim_{t\to t_0}\|u(\cdot,t;t_0,u_0)-u_0(\cdot)\|_{X^\alpha}=0$, and
   \begin{equation}
   \label{local-1-eq1}
   u(\cdot,\cdot;t_0,u_0)\in C([t_0,t_0+T_{\max}),X^\alpha)\cap C^\delta((t_0,t_0+T_{\max}),X^\alpha)
   \end{equation}
   for some $0<\delta<1$.
Moreover if $T_{\max}< \infty,$ then
\vspace{-0.05in}\begin{equation}
\label{local-1-eq2}
\limsup_{t \nearrow T_{\max}}   \left\| u(\cdot,t+t_0;t_0,u_0) \right\|_{X^\alpha}=\limsup_{t \nearrow T_{\max}}   \left\| u(\cdot,t+t_0;t_0,u_0) \right\|_{C^0(\bar \Omega)}  =\infty.
\vspace{-0.05in}\end{equation}

\vspace{-0.1in}\item[(2)] For any given  $t_0 \in \mathbb{R}$  and $u_0 \in C^0(\bar{\Omega})$ with $u_0 \geq 0$,  there exists $T_{\max} \in (0,\infty]$  such that $\eqref{u-v-eq1}$  has a unique non-negative classical solution $(u(x,t;t_0,u_0),v(x,t;t_0,u_0))$ on $(t_0,t_0+T_{\max})$ satisfying that $\lim_{t\to t_0}\|u(\cdot,t;t_0,u_0)-u_0(\cdot)\|_{C^0(\bar\Omega)}=0$, and
\vspace{-0.05in}\begin{equation}
\label{local-2-eq1}
u(\cdot,\cdot;t_0,u_0)\in  C((t_0,t_0+T_{\max}),X^\alpha)\cap C^\delta((t_0,t_0+T_{\max}),X^\alpha)
\vspace{-0.05in}\end{equation}
for some $0<\delta<1$.
Moreover if $T_{\max}< \infty,$ then
\begin{equation}
\label{local-2-eq2}
\limsup_{t \nearrow T_{\max}} \left\| u(\cdot,t+t_0;t_0,u_0) \right\|_{C^0(\bar \Omega)}  =\infty.
\end{equation}
\end{itemize}
\end{theorem}

\begin{remark}
\begin{itemize}
\item[(1)]
Since  $X^{\alpha}\subset C^1(\bar\Omega)\subset C^0(\bar\Omega)$, the existence of a local classical solution in Theorem \ref{thm-001}(1)  is guaranteed by Theorem \ref{thm-001}(2). However $\lim_{t\to t_0}u(\cdot,\cdot; t_0,u_0)=u_0(\cdot)$ in the $X^\alpha$-norm in Theorem \ref{thm-001}(1)
  is  not included in Theorem \ref{thm-001} (2).

\vspace{-0.1in}\item[(2)] Theorem \ref{thm-001}(2) is consistent (one species version ) with \cite[Lemma 2.1]{STW13}.

\vspace{-0.1in}\item[(3)] Semigroup theory and fixed point theorems together with regularity and a prior estimates for elliptic and parabolic equations
are among basic tools used in  literature  to prove the local existence of classical solutions of chemotaxis models with various given initial functions. For the self-completeness, we will give a proof of Theorem \ref{thm-001}(1) by using semigroup theory and give a proof of Theorem \ref{thm-001}(2) based on the combination of fixed point theorems and semigroup theory.
\end{itemize}
\end{remark}

 We next consider the global existence of classical solutions of \eqref{u-v-eq1} with given initial functions. Throughout the paper, we put
\vspace{-0.05in} \begin{equation}
 \label{a-i-sup-inf-eq1}
a_{i,\inf}=\inf _{ t \in\RR,x \in\bar{\Omega}}a_i(t,x),\quad a_{i,\sup}=\sup _{ t \in\RR,x \in\bar{\Omega}}a_i(t,x),
 \vspace{-0.05in}\end{equation}
 \begin{equation}
 \label{a-i-sup-inf-eq2}
a_{i,\inf}(t)=\inf _{x \in\bar{\Omega}}a_i(t,x),\quad a_{i,\sup}(t)=\sup _{x \in\bar{\Omega}}a_i(t,x),
\vspace{-0.05in} \end{equation}
 unless specified otherwise. For convenience, we introduce the following assumptions.

 \smallskip
\noindent {\bf (H2)} {\it $a_1(t,x)$, $a_2(t,x)$, and $\chi$ satisfy
\begin{equation}
\label{global-existence-cond-eq}
\inf_{t \in \mathbb{R}} \Big\{ a_{1,\inf}(t)-|\Omega|\Big( a_{2,\inf}(t)\Big)_- \Big\}>(\chi)_+,
 \end{equation}
  where $|\Omega|$ is the Lebesgue measure of $\Omega$ and $(\chi)_+=\max\{0,\chi\}$.
}

\medskip
\noindent {\bf (H2)$^{'}$} {\it $a_1(t,x)$, $a_2(t,x)$, and $\chi$ satisfy
 $\inf_{t \in \mathbb{R}}\big \{ a_{1,\inf}(t)-|\Omega|\big(a_{2,\inf}(t)\big)_-\big \}>0$ and  if $n\geq 3,$ $a_{1,\inf}>\frac{\chi(n-2)}{n}$.
}

\medskip

The following is our  main result on  the global existence of positive classical solutions to system  \eqref{u-v-eq1}.

\begin{theorem}
\label{thm-002}
\begin{itemize}
\item[(1)] Assume that {\bf (H2) } holds.
Then for any  $t_0\in\RR$ and  $u_0 \in C^0(\bar{\Omega})$ with $u_0\ge 0$,
$\eqref{u-v-eq1}$ has a unique   global  classical solution $(u(x,t;t_0,u_0),v(x,t;t_0,u_0))$  which satisfies that  $\lim_{t\to t_0}\|u(\cdot,t;t_0,u_0)-u_0(\cdot)\|_{C^0(\bar\Omega)}=0$ and  \eqref{local-2-eq1}, \eqref{local-2-eq2}.  Moreover,  we have
\vspace{-0.05in}\begin{align}
\label{global-1-eq1}
0\le v(x,t;t_0,u_0) &\leq \max_{x \in \bar \Omega}u(x,t;t_0,u_0) \nonumber\\
 &\leq   \max\Big\{\sup u_0(x), \frac{a_{0,\sup}}{\inf_{t \ge t_0} \big\{ a_{1,\inf}(t)-|\Omega|\Big( a_{2,\inf}(t)\Big)_--(\chi)_+ \big\}} \Big\}
 \vspace{-0.05in}\end{align}
 for all $(x,t) \in \bar{\Omega} \times [t_0,\infty)$.

\vspace{-0.05in}\item[(2)] Assume that {\bf (H2)$^{'}$} holds. Then for any $t_0\in\RR$ and $u_0 \in C^0(\bar{\Omega})$  with $u_0\ge 0$,  system  $\eqref{u-v-eq1}$ has a unique  global  classical solution $(u(x,t;t_0,u_0),v(x,t;t_0,u_0))$ which satisfies that
$\lim_{t\to t_0}\|u(\cdot,t;t_0,u_0)-u_0(\cdot)\|_{C^0(\bar\Omega)}=0$ and   \eqref{local-2-eq1}, \eqref{local-2-eq2}. Moreover,
\[ \left\| u(\cdot,t;t_0,u_0) \right\|_{C^0(\bar \Omega)}  + \,     \left\| v(\cdot,t;t_0,u_0) \right\|_{C^0(\bar \Omega)} \leq C\]
for all $t\ge t_0$, where $C=C( \left\| u_0 \right\|_{C^0(\bar \Omega)}),$ i.e, $C$ depends only on $\left\| u_0 \right\|_{C^0(
\bar \Omega)}$,
and
$$
0 \leq  \int_{\Omega}u(x,t;t_0,u_0) dx\leq \max\Big\{\int_{\Omega}u_0(x),\frac{|\Omega|a_{0,\sup}}{\inf_{t \in \mathbb{R}} \big \{ a_{1,\inf}(t)-|\Omega|\Big(a_{2,\inf}(t)\Big)_-\big \}} \Big\} \, \forall t \geq t_0.
$$
\end{itemize}
\end{theorem}

\begin{remark}
\begin{itemize}
\vspace{-0.1in}\item[(1)] When $a_2(t,x) \geq 0$, {\bf (H2)$^{'}$} becomes $a_{1,\inf}>\max\{\frac{\chi(n-2)}{n},0\}$.
In particular, if $a_2(t,x)=0$, Theorem \ref{thm-002}  is consistent with the result by Tello and Winkler in \cite{TW07}.

\vspace{-0.1in}\item[(2)] When the coefficients are constant, the condition {\bf (H2)} becomes $a_1-|\Omega|(a_2)_->(\chi)_+$ which is consistent with the result of  global existence  by Negreanu and Tello in \cite{NT13}.

\vspace{-0.1in}\item[(3)]  {\bf (H2)} implies  {\bf (H2)$^{'}$}. Therefore the global existence of bounded classical solutions of $\eqref{u-v-eq1}$ in Theorem \ref{thm-002}(1) follows from Theorem \ref{thm-002}(2). However the explicit bound given by  \eqref{global-1-eq1} is not included in Theorem \ref{thm-002}(2). Note that the explicit bound \eqref{global-1-eq1} will be used in the proof of the existence of periodic solutions (resp.  steady state solutions) when the coefficients $a_i(t,x)$ are periodic (resp. when $a_i(t,x)=a_i(x)$) (see Theorem 1.3).

\vspace{-0.1in}\item[(4)] In general, assuming that $\inf_{t \in \mathbb{R}} \big\{a_{1,\inf}(t)-|\Omega|\Big( a_{2,\inf}(t)\Big)_-\big\}>0$, it remains open whether for any given $t_0\in\RR$ and $u_0\in C^0(\bar \Omega)$,
\eqref{u-v-eq1} has a global classical solution $(u(x,t;t_0,u_0)$, $v(x,t;t_0,u_0))$. This is open even in the case that $a_i(t,x)\equiv a_i$
for $i=0,1$ and
 $a_2(t,x)=0$.

\vspace{-0.1in}\item[(5)] When $\chi\leq 0,$ conditions {\bf (H2)} and {\bf (H2)$^{'}$} become $\inf_{t \in \mathbb{R}}\big \{ a_{1,\inf}(t)-|\Omega|\big(a_{2,\inf}(t)\big)_-\big \}>0.$ Thus global existence holds under this weak condition.
\end{itemize}
\end{remark}

We now state our main result on the existence  of  entire bounded positive solutions of \eqref{u-v-eq1}.

\begin{theorem}
\label{thm-003} Suppose that {\bf (H2)} holds.
  Then there is an entire positive bounded classical solution $(u,v)=(u^*(x,t),v^*(x,t))$ of \eqref{u-v-eq1}.  Moreover,  the following hold.
\begin{itemize}
\vspace{-0.05in}\item[(1)] If there is $T>0$ such that $a_i(t+T,x)=a_i(t,x)$ for $i=0,1,2$, then  $\eqref{u-v-eq1}$  has a  positive  periodic solution $(u,v)=(u^*(x,t), v^*(x,t))$ with period $T$.

\vspace{-0.05in}\item[(2)]  If  $a_i(t,x)\equiv a_i(t)$ for $i=0,1,2$,   then  $\eqref{u-v-eq1}$  has a unique entire  positive spatially homogeneous
 solution  $(u,v)=(u^*(t),v^*(t))$ with $v^*(t)=u^*(t)$, and if $a_i(t)$ $(i=0,1,2)$ are periodic or almost periodic, so is   $(u^*(t),v^*(t))$.

\vspace{-0.05in}\item[(3)] If   $a_i(t,x)\equiv a_i(x)$ for $i=0,1,2$, then   $\eqref{u-v-eq1}$  has a positive  steady state solution  $(u,v)=(u^*(x), v^*(x))$.
\end{itemize}
\end{theorem}

\begin{remark}
\begin{itemize}
\item[(1)] When the coefficients are only time dependent,   i.e,   $a_i(t,x)=a_i(t)$ for $i=0,1,2$,  every entire positive solution $u(t)$ of the ODE \vspace{-0.1in}
     $$u_t=u[a_0(t)-(a_1(t)+|\Omega|a_2(t))u]
     \vspace{-0.1in}$$
is  an entire positive solution of  the  first equation of  \eqref{u-v-eq1} and then $(u(t),v(t))$ with $v(t)=u(t)$ is an entire positive solution of
\eqref{u-v-eq1}. Thus  $\eqref{u-v-eq1}$ has  an entire solution under the weaker assumption    $\inf_{t \in \mathbb{R}}
 \{ a_1(t)-|\Omega|(a_2(t))_-\}>0$ (see Lemma \ref{lem-00001}).
In general, due to the lack of comparison principle for system \eqref{u-v-eq1}, it is fairly nontrivial to prove the existence of  entire positive solutions.

\vspace{-0.1in}\item[(2)] It should be mentioned that there may be lots of entire positive solutions (see \cite{kuto_PHYSD}, \cite{TW07}).

\vspace{-0.1in}\item[(3)] The existence of entire positive bounded classical  of \eqref{u-v-eq1} also holds under the weaker assumption {\bf (H2)${'}$} (see Remarks \ref{remark-h3-1} and \ref{remark-h3-2}). However under {\bf (H2)$^{'}$}, it reminds open whether there are periodic solutions of \eqref{u-v-eq1} when the coefficients $a_i(t,x)$ are periodic (resp.  steady state solutions of \eqref{u-v-eq1} when  $a_i(t,x)\equiv a_i(x)$).
    \end{itemize}
\end{remark}

Finally we state the main results on the stability and uniqueness of entire positive solutions and asymptotic behavior of positive solutions of \eqref{u-v-eq1}.

\begin{theorem}
\label{thm-004}
\begin{itemize}
\item[(1)] If  $a_i(t,x)\equiv a_i(t)$ for $i=0,1,2$ and
\vspace{-0.05in}\begin{equation}
\label{stability-cond-1-eq1}
\inf_t \big\{a_1(t)-|\Omega|\abs{a_2(t)}\big\} >2(\chi)_+,
\vspace{-0.05in}\end{equation}
  then  for any
$t_0\in\RR$ and  $u_0 \in C^0(\bar{\Omega})$ with $u_0\ge 0$ and
$u_0\not\equiv 0$,  the  unique global classical solution
$(u(x,t;t_0,u_0), v(x,t;t_0,u_0))$ of $\eqref{u-v-eq1}$ satisfies
\vspace{-0.05in}\begin{equation}
\label{global-stability-1-eq1}
\lim_{t \to \infty}\big(  \left\| u(\cdot,t;t_0,u_0)-u^*(t) \right\|_{C^0(\bar\Omega)} +\left\| v(\cdot,t;t_0,u_0)-u^*(t) \right\|_{C^0(\bar\Omega)}\big)=0,
\vspace{-0.05in}\end{equation}
where   $u^*(t)$ is the unique spatially homogeneous entire positive solution  of \eqref{u-v-eq1}.

\vspace{-0.1in}\item[(2)] Suppose that
\vspace{-0.05in}\begin{equation}
\label{stability-cond-2-eq1}
  \inf_{t \in \mathbb{R}}\Big \{a_{1,\inf}(t)-|\Omega|\Big(a_{2,\inf}(t)\Big)_-\Big\}>\Big\{(\chi )_++\frac{a_{0,\sup}}{
  a_{0,\inf}}\Big((\chi )_++|\Omega|\sup_{t \in \mathbb{R}}\Big({ a_{2,\sup}(t)}\Big)_+ \Big)\Big\}
\vspace{-0.05in}\end{equation}
and
\vspace{-0.05in}\begin{equation}
\label{stability-cond-2-eq2}
\limsup_{t-s\to\infty}\frac{1}{t-s}\int_{s}^{t}(L_2(\tau)- L_1(\tau))d\tau <0,
\vspace{-0.05in}\end{equation}
where
 \vspace{-0.05in}\begin{equation}
 \label{L-eq1}
 L_1(t)={2r_2\big(a_{1,\inf}(t)+|\Omega|(a_{2,\inf}(t))_+\big)},
 \vspace{-0,05in}\end{equation}
\vspace{-0,05in} \begin{align}
 \label{L-eq2}
 L_2(t)=a_{0,\sup}(t)+\frac{\chi }{2}(r_1-r_2)+\frac{(\chi  r_1)^2}{2}+|\Omega|r_1\big (2(a_{2,\inf}(t))_-+  { \big(a_{2,\sup}(t)\big)_+}\big),
\vspace{-0.05in}\end{align}
and
\begin{equation}
\label{r-eq1}
r_1=\frac{ \sup_{t \in\RR} \{a_{1,\sup}(t)-|\Omega|(a_{2,\sup}(t))_--(\chi)_+\} a_{0,\sup}- a_{0,\inf} \big((\chi)_+ +
|\Omega|\inf_t(a_{2,\inf}(t))_+\big)} {h(\chi)},
\end{equation}
\begin{equation}
\label{r-eq2}
r_2=\frac{ \inf_{t \in\RR}\{ { a_{1,\inf}(t)}-|\Omega|(a_{2,\inf}(t))_--(\chi)_+ \} a_{0,\inf}- a_{0,\sup}\big((\chi)_+ +
|\Omega|\sup_t (a_{2,\sup}(t))_+\big)} { h(\chi)},
\end{equation}
\begin{align*}
h(\chi)=& \inf_{t \in\RR} \{a_{1,\inf}(t)-|\Omega|(a_{2,\inf}(t))_--(\chi)_+\}\sup_{t \in\RR} \{a_{1,\sup}(t)-|\Omega|(a_{2,\sup}(t))_--(\chi)_+\}\\
&-\big((\chi)_++|\Omega|\inf_{t \in \mathbb{R}}(a_{2,\inf}(t))_+\big)\big((\chi)_+
+|\Omega|\sup_{t \in \mathbb{R}}(a_{2,\sup}(t))_+\big).
\end{align*}

 Then \eqref{u-v-eq1} has a unique entire positive solution $(u^*(x,t),v^*(x,t))$, and,  for any
$t_0\in\RR$ and  $u_0 \in C^0(\bar{\Omega})$ with $u_0\ge 0$ and
$u_0\not\equiv 0$,  the  global classical solution
$(u(x,t;t_0,u_0)$,  $v(x,t;t_0,u_0))$ of $\eqref{u-v-eq1}$ satisfies
  \begin{equation}
  \label{global-stability-2-eq1}
  \lim_{t \to \infty}\Big(\|u(\cdot,t;t_0,u_0)-u^*(\cdot,t)\|_{C^0(\bar\Omega)}+\|v(\cdot,t;t_0,u_0)-v^*(\cdot,t)\|_{C^0(\bar\Omega)}\Big)=0.
  \end{equation}
 If, in addition, $a_i(t,x)\equiv a_i(x)$  (resp. $a_i(t+T,x)=a_i(t,x)$,  $a_i(t,x)$  is almost periodic in $t$ uniformly with respect to $x$) for $i=0,1,2$, then
\eqref{u-v-eq1} has a unique positive steady state solution $(u^*(x),v^*(x))$ (resp. \eqref{u-v-eq1} has a unique time periodic positive solution $(u^*(x,t),v^*(x,t))$ with period $T$, \eqref{u-v-eq1} has
a unique time almost periodic solution $(u^*(x,t),v^*(x,t))$).
\end{itemize}
\end{theorem}

\begin{theorem}
\label{thm-005} Suppose that \eqref{stability-cond-2-eq1} holds and $r_1$ and $r_2$ are as in Theorem \ref{thm-004}(2). Then
\begin{itemize}
\item[(1)] For any $t_0\in\RR$,   $u_0\in C^0(\bar \Omega)$ with $u_0\ge 0$ and $u_0\not \equiv 0$, and $\epsilon>0,$  there exists $t_\epsilon$ such that
 \vspace{-0.05in}
 $$r_2-\epsilon \le u(x,t;t_0,u_0) \le r_1+\epsilon,\,\,\,r_2-\epsilon \le v(x,t;t_0,u_0) \le r_1+\epsilon
 \vspace{-0.05in}$$
 for all $x\in\bar\Omega$ and $t\ge t_0+t_\epsilon$.

\item[(2)] Moreover if the coefficients $a_i$ are periodic in $t$ with period   $\,T>0$ (resp. $a_i$ are almost periodic in $t$), then there are $T$-periodic functions  $m(t)$ and $M(t)$ (resp. almost periodic functions $m(t)$ and $M(t)$)  with
   \vspace{-0,05in} $$r_2\le \inf_{t\in\RR}m(t)\le m(t)\le M(t)\le \sup_{t\in\RR}M(t)\le r_1
   \vspace{-0,05in}$$
   such that for any $t_0\in\RR$,  $u_0\in C(\bar \Omega)$ with $u_0\ge 0$ and $u_0\not \equiv 0$, and $\epsilon>0$, there is $t_\epsilon>0$ such that
\vspace{-0.05in}$$
m(t)-\epsilon\le u(x,t;t_0,u_0)\le M(t)+\epsilon,\quad
m(t)-\epsilon\le v(x,t;t_0,u_0)\le M(t)+\epsilon\quad \forall \,
x\in\bar\Omega,\,\, t\ge t_0+t_\epsilon.
\vspace{-0,05in}$$

\end{itemize}
\end{theorem}

\begin{remark}
\begin{itemize}
\item[(1)] When $a_i$ $(i=0,1,2)$ are constants,   the condition \eqref{stability-cond-1-eq1}    becomes
\vspace{-0.05in}\begin{equation}
\label{stability-cond-1-eq2}
a_1-|\Omega|\abs{a_2}> 2(\chi)_+.
\vspace{-0.05in}\end{equation}
Theorem \ref{thm-004}(1) is then  an extension of \cite[Theorem 0.1 ] {NT13}  by Negreanu and Tello.  When the  nonlocal term is zero, the result in Theorem \ref{thm-004}(1)  is consistent with the result by Tello and Winkler in \cite{TW07}.

\vspace{-0.1in}\item[(2)] In Theorem \ref{thm-004}(2),
when $a_i$ $(i=0,1,2)$ are constants ,  we have $r_1=r_2=\frac{a_0}{a_1+|\Omega|a_2}$ and
\vspace{-0.05in}$$
L_1(t)=\frac{2a_0(a_1+|\Omega|(a_2)_+)}{a_1+|\Omega|a_2},\quad L_2(t)=a_0+\frac{\chi^2 a_0^2}{2(a_1+|\Omega|a_2)^2}+\frac{a_0|\Omega|[2(a_2)_-+(a_2)_+]}{a_1+|\Omega|a_2}.
\vspace{-0.05in}$$
Hence the condition \eqref{stability-cond-2-eq1} becomes \eqref{stability-cond-1-eq2} and
the condition  \eqref{stability-cond-2-eq2}  becomes
 \vspace{-0.05in}\begin{equation}
 \label{stability-cond-2-eq3}
  { \frac{\chi^2a_0}{2{(a_1+|\Omega|a_2)}}<  a_1-|\Omega|(a_2)_-.}
 \vspace{-0.05in} \end{equation}
 Furthermore when $\chi=0,$ the condition \eqref{stability-cond-2-eq2}   becomes
 \vspace{-0.05in}\begin{equation}
 \label{stability-cond-2-eq3}
a_1-{ |\Omega|(a_2)_-}>0.
  \vspace{-0.05in}\end{equation}

\vspace{-0.1in}\item[(3)] In Theorem \ref{thm-004}(2), if $a_i(t+T,x)=a_i(t,x)$ $(i=0,1,2)$, then \eqref{stability-cond-2-eq2} becomes
\vspace{-0.05in}$$
\int_{0}^{T}\big(L_2(t)-L_1(t)\big)dt<0.
\vspace{-0.05in}$$

\vspace{-0.1in}\item[(4)] It is seen from Theorem \ref{thm-003} that  \eqref{global-existence-cond-eq} ensures the existence of entire positive solutions of
\eqref{u-v-eq1}. In the case that $a_i(t,x)\equiv a_i(t)$ $(i=0,1,2)$, the condition \eqref{stability-cond-1-eq1} ensures the stability and uniqueness
of entire positive solutions of \eqref{u-v-eq1}. In the general case,
 Theorem \ref{thm-005} provides some positive attracting set for positive solutions of \eqref{u-v-eq1} under the  condition \eqref{stability-cond-2-eq1}.
 It remains open whether in the general case, the condition \eqref{stability-cond-2-eq1} also ensures the stability and uniqueness
of entire positive solutions of \eqref{u-v-eq1}.

\vspace{-0.1in}\item[(5)] The reader is referred to Definition \ref{almost-periodic-function-def} for the definition of almost periodic functions.
\end{itemize}
\end{remark}

The rest of the paper is organized as follows. In section 2, we collect some important results from literature that will be used in the proofs
of our main results. In section 3, we study the local existence of classical solutions of \eqref{u-v-eq1} with given initial functions and prove Theorem \ref{thm-001}. In section 4,
we  investigate the global existence of classical solutions of \eqref{u-v-eq1} with given initial functions and
prove Theorem \ref{thm-002}. We consider the existence of entire positive solutions of \eqref{u-v-eq1} and prove Theorem \ref{thm-003}
 in section 5. Finally, in section 6, we study the asymptotic behavior of global positive solutions and prove Theorems \ref{thm-004} and \ref{thm-005}.

\medskip

\noindent {\bf Acknowledgment.} The authors would like to thank Professors J. Ignacio Tello and Michael Winkler
 for valuable discussions, suggestions,  and references.

\section{Preliminaries}

In this section, we recall  some standard definitions and lemmas from semigroup theory.
We also present some  known results on non-autonomous logistic equations and  Lotka-Volterra competition systems.

\subsection{Semigroup theory}

In this subsection, we recall some standard definitions and lemmas from semigroup theory. The reader is referred to \cite{DH77}, \cite{A. Pazy} for the details.

Recall  that
for given $1\le p<\infty$,   $A=-\Delta+I$ with
  \vspace{-0.05in}$$D(A)=\big\{ u \in W^{2,p}(\Omega) \, |  \, \frac{\p u}{\p n}=0 \quad \text{on } \, \p \Omega \big\}
 \vspace{-0.05in} $$
  and $X^{\alpha}=D(A^{\alpha})$ equipped with the graph norm  $\|x\|_{\alpha}=\|A^{\alpha}x\|_{p}$.
Note that $X^0=L^p(\Omega)$.

\begin{lemma}(See \cite[Theorem 1.6.1]{DH77})
\label{lem-002}
Let $1\leq p<\infty.$ For  any $0\leq \alpha \leq 1,$ we have
\vspace{-0.05in}$$
X^{\alpha} \subset C^{\nu}(\bar{\Omega}) \,\, \text{when}\,\, 0\leq \nu < 2\alpha -\frac{n}{p},
\vspace{-0.05in}$$
 where the inclusion is continuous.
In particular when $\frac{n}{2p}<\alpha \leq 1,$  we get   $X^{\alpha} \subset C^0(\bar{\Omega}).$
\end{lemma}

\begin{lemma}(See \cite[Lemma 2.1]{HW05})
\label{lem-003}
Let $\beta \geq 0$ and $p \in (1,\infty).$ Then for any $\epsilon >0$ there exists $C(\epsilon)>0$ such that for any $w \in  C^{\infty}_0(\Omega)$ we have
\begin{equation}
\label{001}
\|A^{\beta}e^{-tA}\nabla\cdot w\|_{L^p(\Omega)}  \leq C(\epsilon) t^{-\beta-\frac{1}{2}-\epsilon} e^{-\mu t} \|w\|_{L^p(\Omega)} \quad \text{for all}\,\,  t>0 \, \text{and  some } \, \mu>0.
\end{equation}
Accordingly, for all $t>0$ the operator $A^{\beta}e^{-tA}\nabla\cdot $ admits a unique extension to all of $L^p(\Omega)$ which is again denoted by $A^{\beta}e^{-tA}\nabla\cdot$  and  satisfies $\eqref{001}$ for all $w \in L^p(\Omega).$
\end{lemma}

Consider
\begin{equation}
\label{new-evolution-eq1}
\begin{cases}
u_t+A u=F(t,u) , \, t>t_0\cr
u(t_0)=u_0.
\end{cases}
\end{equation}
We assume that $F$ maps some open set $U$ of $\RR\times X^\alpha$ into $X^0$ for some $0\le \alpha<1$, and $F$ is locally H\"older continuous in $t$ and
locally Lipschitz continuous in $u$ for $(t,u)\in U$.

\begin{definition} [Mild solution]
\label{mild-solution-def}
For given $u_0\in X^\alpha$.
 A continuous function $u: [t_0, t_1) \to X^0$ is called a {\rm mild solution} of $\eqref{lem-003}$ on $t_0< t<t_1$  if $u(t)\in X^\alpha$ for
  $t\in [t_0,t_1)$ and the following integral equation holds on $t_0< t<t_1$,
\begin{equation}
\label{mild-solution}
u(t)=e^{-A( t-t_0)}u_0 +\int_{t_0}^te^{-A( t-s)} F(s,u(s))ds.
\end{equation}
\end{definition}

\begin{definition}[Strong solution] (see \cite[Definition 3.3.1]{DH77})
\label{strong-solution-def}
A {\rm  strong solution}  of the Cauchy problem $\eqref{lem-003}$ on $(t_0,t_1)$ is a continuous function  $u: [t_0, t_1) \to X^0$ such that $u(t_0)=u_0$,  $u(t) \in D(A)$ for $t\in (t_0,t_1)$,  $\frac{d u}{d t}$ exists for $t\in (t_0,t_1)$, $(t_0,t_1)\ni t \to F(t,u(t))\in X^0$ is locally H\"older continuous, and $\int^{t_0+\sigma}_{t_0} \|F(t,u(t))\|dt < \infty$ for some $\sigma>0,$ and the differential equation $u_t+A u=F(t,u) $ is satisfied on $(t_0,t_1).$
\end{definition}

\begin{lemma}[Existence of mild/strong solutions]
\label{mild-strong-solution-lm}
\begin{itemize}
\item[(1)] For any $(t_0,u_0) \in U$ there exists $T_{\max}=T_{\max}(t_0,u_0)>0$  such that $\eqref{lem-003}$ has a unique strong solution
 $u(t;t_0,u_0)$ on $(t_0,t_0+T_{\max})$ with initial value $u(t_0;t_0,u_0)=u_0.$ Moreover, $u(\cdot;t_0,u_0)\in C([t_0,t_0+T_{\max}),X^\alpha)$ and
    if $T_{\max}< \infty,$ then
$$\limsup_{t \nearrow T_{\max}}  \left\|u(t+t_0;t_0,u_0) \right\|_{{X^{\alpha}}}  =\infty.$$

\item[(2)] For given $(t_0,u_0)\in U$, if  $u(t)$ is a strong  solution of  $\eqref{lem-003}$ on $(t_0,t_1),$ then  $u$ satisfy the integral equation \eqref{mild-solution}. Conversely, if $u(t)$ is continuous function from $(t_0,t_1)$ into $ X^\alpha,$  $\int^{t_0+\sigma}_{t_0} \|F(t,u(t))\|dt < \infty$ for some $\sigma>0,$ and if the integral equation \eqref{mild-solution} holds for $t_0< t<t_1,$ then $u(t)$ is a strong solution of the differential equation  $\eqref{lem-003}$ on $(t_0,t_1).$ Furthermore,
    \vspace{-0.05in}
    $$u\in C^\delta((t_0,t_1), X^\alpha)\,\,\text{for all}\,\, \delta\,\,\text{such that}\,\, 0<\delta<1-\alpha
    \vspace{-0.05in}$$
\end{itemize}
\end{lemma}

\begin{proof}
(1)  It  follows from \cite[Theorem 3.3.3 ]{DH77} and  \cite[Theorem 3.3.4 ]{DH77}.

(2) The equivalence part follows from \cite[Lemma 3.3.2]{DH77} and $u \in C^\delta((t_0,t_1);X^\alpha)$ follows from the proof of \cite[Lemma 3.3.2]{DH77}.
\end{proof}

\subsection{Nonautonomous logistic equations and Lotka-Volterra competition systems}

In this subsection, we  first recall the definition
of almost periodic functions and some basic properties of almost periodic functions. We then review some known results for nonautonomous logistic equations and Lotka-Volterra competition systems.

\begin{definition}
\label{almost-periodic-function-def}
\begin{itemize}
\vspace{-0.05in}\item[(1)] A continuous function $f: \mathbb{R} \to \mathbb{C}$ is Bohr almost periodic if for any $\epsilon>0,$ the set of $\epsilon$-periods $\{\tau \,|\, |f(t+\tau)-f(t)|<\epsilon \}$ is relatively dense in $\mathbb{R},$ i.e, there exists an $l=l(\epsilon)$ such that every interval of the form $[t,t+l]$ intersects the set of $\epsilon$-periods.

\vspace{-0.1in}\item[(2)] Let $g(t,x)$ be a continuous function of $(t,x)\in\RR\times \bar\Omega$. $g$ is said to be {\rm almost periodic in $t$ uniformly with respect to $x\in \bar\Omega $}  if
$g$ is uniformly continuous in $t\in\RR$ and  $x\in\bar\Omega$,  and for each $x\in\bar\Omega $, $g(t,x)$ is almost periodic in $t$.
\end{itemize}
\end{definition}

\begin{lemma}
\label{almost-periodic-function-lm}
 Let $g(t,x)$ be a continuous function of $(t,x)\in\RR\times\bar\Omega$. $g$ is almost periodic in $t$ uniformly with respect to
$x\in\bar\Omega$  if and only if
 $g$ is uniformly continuous in $t\in\RR$ and $x\in\bar\Omega$,  and for any sequences $\{\beta_n^{'}\}$,
$\{\gamma_n^{'}\}\subset \RR$, there are subsequences $\{\beta_n\}\subset\{\beta_n^{'}\}$, $\{\gamma_n\}\subset\{\gamma_n^{'}\}$
such that
\vspace{-0.05in}$$
\lim_{n\to\infty}\lim_{m\to\infty}g(t+\beta_n+\gamma_m,x)=\lim_{n\to\infty}g(t+\beta_n+\gamma_n,x)\quad \forall\,\, (t,x)\in\RR\times\bar\Omega.
\vspace{-0.05in}$$
\end{lemma}

\begin{proof}
See \cite[Theorems 1.17 and 2.10]{Fink}.
\end{proof}

Consider the following nonautonomous  logistic equations
\vspace{-0.05in}\begin{equation}
\label{ode-000}
\frac{du}{dt}=u(a(t)-b(t) u),
\vspace{-0.05in}\end{equation}
where $a(t)$ and $b(t)$ are continuous functions. For given $u_0\in\RR$, let $u(t;t_0,u_0)$ be the solution of
\eqref{ode-000} with $u(t_0;t_0,u_0)=u_0$.


\begin{lemma}(see \cite{NMN00},Theorems 2.1, 3.1 and 4.1)
\label{lem-00001} Suppose that $a(t)$ and $b(t)$ are continuous and satisfy  that
$0<\inf_{t\in\RR}a(t)\le \sup _{t\in\RR} a(t)<\infty, \quad 0<\inf_{t\in\RR}b(t)\le \sup _{t\in\RR} b(t)<\infty.$
Then
\begin{itemize}
\vspace{-0.1in}\item[(1)]The non-autonomous equation \eqref{ode-000} has exactly one bounded entire solution $u^*(t)$ that is positive and satisfies $$\frac{\inf_{t\in\RR}a(t)}{\sup_{t\in\RR} b(t)}\leq u^*(t) \leq \frac{\sup_{t\in\RR}a(t)}{\inf_{t\in\RR} b(t)}\quad \forall\,\, t\in\RR.$$

\vspace{-0.05in}\item[(2)] $u^*(\cdot)$ is an attractor for all positive solutions of \eqref{ode-000}, that is, for any $u_0>0$ and $t_0\in\RR$,
\vspace{-0.05in} $$\lim_{t \to \infty}\|u(t+t_0;t_0,u_0)-u^*(t+t_0)\|=0.
\vspace{-0.05in}$$

\vspace{-0.1in}\item[(3)]If furthermore $a(t)$ and $b(t)$ are periodic with period $T$ (resp. almost periodic), $u^*(t)$ is also periodic with period $T$
 (resp. almost periodic).
\end{itemize}
\end{lemma}

Consider now the following nonautonomous  Lotka-Volterra competition systems
\vspace{-0.05in}\begin{equation}
\label{ode-001}
\begin{cases}
\frac{du}{dt}=u(a_1(t)-b_1(t)u-c_1(t) v) \\
\frac{d v}{dt}=v(a_2(t)-b_2(t)u-c_2(t)v),
\end{cases}
\vspace{-0.05in}\end{equation}
where $a_i(t)$, $b_i(t)$,  and $c_i(t)$ $(i=1,2)$  are continuous and bounded above and below by positive constants.

Given a function $f(t),$ which is bounded above and below by  positive constants,  we let
 \vspace{-0.05in}$$f^L=\inf_{t\in\RR}f(t)\quad {\rm and}\quad
 f^M=\sup_{t\in\RR}f(t).\vspace{-0.05in}$$

\begin{lemma}
\label{lem-00002} Suppose that  $a_1^L>\frac{c_1^M a_2^M}{c_2^L}$ and   $a_2^L>\frac{a_1^M b_2^M}{b_1^L}.$
\begin{itemize}
\vspace{-0.05in}\item[(1)] Suppose that $(u_1(t),v_1(t))$ and $(u_2(t),v_2(t))$ are  two solutions of the system  \eqref{ode-001} with $u_k(t_0)>0$, $v_k(t_0)>0$
     $(k=1,2)$.  Then $u_1(t)-u_2(t) \to 0$ and  $v_1(t)-v_2(t) \to 0$ as $t \to \infty.$

\vspace{-0.1in}\item[(2)] For any $t_0\in\RR$, there exists a solution $(u_0(t),v_0(t))$ of  system  \eqref{ode-001} for $t\ge t_0$ such that
\vspace{-0.05in}$$0<\frac{a_1^L c_2^L-c_1^M a_2^M}{b_1^M c_2^L-c_1^Mb_2^L} \le u_0(t) \leq  \frac{a_1^M c_2^M-c_1^L a_2^L}{b_1^Lc_2^M-c_1^Lb_2^M}\quad\forall\,\, t\ge t_0,
\vspace{-0.05in}$$
$$
0< \frac{ b_1^La_2^L-a_1^M b_2^M}{b_1^Lc_2^M-c_1^L b_2^M} \le v_0(t) \leq  \frac{ b_1^M a_2^M-a_1^L b_2^L}{b_1^M c_2^L-c_1^M b_2^L}\quad \forall\,\, t\ge t_0.
\vspace{-0.05in}$$

\vspace{-0.1in}\item[(3)] If moreover the coefficients are positive and $T$-periodic,  then there exist exactly one $T$-periodic solution of the system  \eqref{ode-001} with positive components, which attracts all solutions that begin in the open first quadrant.

\vspace{-0.1in}\item[(4)]      If moreover the coefficients are positive and almost periodic,  then there exist exactly one almost periodic solution of the system  \eqref{ode-001} with positive components, which attracts all solutions that begin in the open first quadrant.
\end{itemize}
\end{lemma}

\vspace{-0.1in}\begin{proof}
(1), (2),  (3) follow from  \cite[Theorems 1  and 2]{Ahm}, and (4) follows from \cite[Theorem C]{HeSh}.
\end{proof}

\section{Local existence and uniqueness of classical solutions}

In this section, we  study the local existence and uniqueness of classical solutions  of   $\eqref{u-v-eq1}$  with given initial
 functions and prove Theorem  \ref{thm-001}.

First, observe that $C^0(\bar\Omega)\subset L^p(\Omega)$ for any $1\le p<\infty$. Throughout this section, unless specified otherwise,
$p>1$ and $\alpha\in (1/2,1)$ are such that $X^\alpha \subset C^1(\bar\Omega)$, where $X^\alpha =D(A^\alpha)$ with the graph norm
$\|u\|_\alpha=\|A^\alpha u\|_{L^p(\Omega)}$ and $A=I-\Delta$ with domain $D(A)=\{u\in W^{2,p}(\Omega)\,|\, \frac{\p u}{\p n}=0$ on $\p\Omega\}$.
 Note that $A:  D(A) \to X^0(=L^p(\Omega))$ is a linear, bounded bijection, and $A^{-1} \colon X^{0} \to X^{\alpha}$ is compact.

 Next, we note that if $(u(x,t;t_0,u_0),v(x,t;t_0,u_0))$ is a classical solution of \eqref{u-v-eq1} satisfying
 the properties in Theorem \ref{thm-001} (1) or (2), then $v(\cdot,t;t_0,u_0)=A^{-1} u(\cdot,t;t_0,u_0)$ and $u(x,t;t_0,u_0)$
 is a classical solution of
 \begin{equation}
\label{u-eq00}
\begin{cases}
u_t=(\Delta -1)u+f(t,x,u),\quad x\in\Omega\cr
\frac{\p u}{\p n}=0,\quad x\in\p\Omega\end{cases}
\end{equation}
with $u(x,t_0;t_0,u_0)=u_0(x)$,
where
\vspace{-0.05in}$$f(t,x,u)=-\chi\nabla u \cdot \nabla A^{-1}u +\chi u(u-A^{-1}u)+u\Big(1+a_0(t,x)-a_1(t,x)u-a_2(t,x)\int_{\Omega}u\Big).
\vspace{-0.05in}$$
Conversely,  if $u_0\in X^\alpha$ (resp. $u_0\in C^0(\bar\Omega)$) and $u(x,t;t_0,u_0)$ is a classical solution of
\eqref{u-eq00} satisfying the properties in Theorem \ref{thm-001} (1) (resp. (2)), then
$(u(x,t;t_0,u_0),v(x,t;t_0,u_0))$ is a classical solution of \eqref{u-v-eq1} satisfying the properties in  Theorem \ref{thm-001} (1) (resp. (2)),
where $v(\cdot,t;t_0,u_0)=A^{-1} u(\cdot,t;t_0,u_0)$.

We now prove Theorem \ref{thm-001}  and we will only present the proof for $\chi>0$ since the case $\chi\leq 0$ is similar.
In the rest of this section, $C$ denotes a constant independent of the initial conditions and the solutions under consideration, unless otherwise specified.

\begin{proof} [Proof of Theorem \ref{thm-001}]
(1) We  use the semigroup approach to prove (1) and divide the proof into four steps.

\noindent {\bf Step 1.} (Existence of strong  solution). In this step, we  prove the existence of a unique strong solution $u(\cdot,t;t_0,u_0)$ of \eqref{u-eq00}
in $X^\alpha$ with $u(\cdot,t_0;t_0,u_0)=u_0$ and satisfying \eqref{local-1-eq1} and \eqref{local-1-eq2}. In order to do so, we write \eqref{u-eq00} as
\begin{equation}
\label{u-eq0}
u_t+A u=F(t,u),
\end{equation}
where
$F(t,u)=-\chi\nabla u \cdot \nabla A^{-1}u +\chi u(u-A^{-1}u)+u\Big(1+a_0(t,\cdot)-a_1(t,\cdot)u-a_2(t,\cdot)\int_{\Omega}u\Big).$

We claim that $F:\RR\times X^\alpha \to X^0$ is locally H\"older continuous in $t$ and locally Lipschitz continuous in $u$.
 In fact, for any $s,t\in\RR$  and $u,w \in X^{\alpha}$, we have
\vspace{-0.05in}\begin{eqnarray*}
  \|F(t,u)-F(s,w)\|_{L^p(\Omega)} \le   C \Big[\left(\|u\|_{C^1}+\|w\|_{C^1}+1\right)\|u-w\| _{C^1} +  \left( \|u\|_{C^1}+\|w\|^2_{C^1}+\|w\|_{C^1}  \right) |t-s|^{\nu}\Big].
\vspace{-0.05in}\end{eqnarray*}
Since $X^{\alpha} \subset C^1(\bar{\Omega}),$ we get
\vspace{-0.05in}\begin{eqnarray*}
   \|F(t,u)-F(s,w)\|_{L^p(\Omega)} \leq  C \Big[ \left(\|u\|_{\alpha}+\|w\|_{\alpha}+1\right)\|u-w\| _{\alpha}+
     \left( \|u\|_{\alpha}+\|w\|^2_{\alpha}+\|w\|_{\alpha}  \right) |t-s|^{\nu}\Big].
\vspace{-0.05in}\end{eqnarray*}
The claim then follows.
Then by Lemma \ref{mild-strong-solution-lm}(1), for every $(t_0,u_0 ) \in \mathbb{R}\times X^{\alpha},$ there exists $T_{\max}=T_{\max}(t_0,u_0)>0$  such that $\eqref{u-eq0}$ has a unique strong solution $ u(\cdot,t;t_0,u_0)$ on $(t_0,t_0+T_{\max})$.
 Moreover  if $T_{\max}< \infty,$ then
$\limsup_{t \nearrow T_{\max}}   \left \| u(\cdot,t+t_0;t_0,u_0) \right\|_{{X^{\alpha}}}  =\infty.$

Note that $u(\cdot,\cdot;t_0,u_0)\in C([t_0,t_0+T_{\max}),X^\alpha)$. By  Lemma \ref{mild-strong-solution-lm}(2), $u \in C^{\delta}((t_0,t_0+T_{\max}),X^{\alpha}) \cap  C^1((t_0,t_0+T_{\max}),X_0)$ for any $\delta $ satisfying  $0<\delta<1-\alpha$. Hence \eqref{local-1-eq1}
holds.  Moreover,  $u(x,t):=u(x,t;t_0,u_0)$ is a mild solution of \eqref{u-eq0} given by
 \begin{align*}
\label{u-mild0}
{u}(\cdot,t)=&e^{-A (t-t_0)}u_0 -\chi\int_{t_0}^t e^{-A(t-s)} \nabla u(\cdot,s) \cdot \nabla A^{-1}u(\cdot,s)ds\\
 &+\chi\int_{t_0}^t e^{-A(t-s)}u(\cdot,s)\big(u(\cdot,s)-A^{-1}u(\cdot,s)\big)ds\\
&+\int_{t_0}^t e^{-A(t-s)}u(\cdot,s)\Big(1+ a_0(s,\cdot)-a_1(s,\cdot)u(\cdot,s)-a_2(s,\cdot)\int_{\Omega}u(\cdot,s)\Big)ds.
\end{align*}


It remains to show that, if $T_{\max}<\infty$, then  \eqref{local-1-eq2} holds.
Assume by contradiction that $T_{\max}<\infty$ and
$\limsup_{t \nearrow T_{\max}}\left (    \left\| u(\cdot,t+t_0) \right\|_{C^0(\bar \Omega)} \right)=L<\infty.$
Then there exists $\delta_0>0$ such that
\begin{equation}
\label{local-existence-eq1}
L \leq \sup_{T_{\max}-\delta_0 <t < T_{\max}}\left\| u(\cdot ,t+t_0) \right\|_{C^0(\bar \Omega)}<L+1.
 \end{equation}
 Let $t$ be such that $ t_0+T_{\max}-\delta_0 <t < t_0+T_{\max}$.
  We have
\begin{align*}
&\|F(t,u(\cdot,t))\|_{L^p(\Omega)}\\
&\leq \|-\chi\nabla u \cdot \nabla A^{-1}u \|_{L^p(\Omega)}+\|\chi u(u-A^{-1}u)+u\Big(1+a_0(t,\cdot)-a_1(t,\cdot)u-a_2(t,\cdot)\int_{\Omega}u\Big)\|_{L^p(\Omega)}\\
&\leq \chi \|u(\cdot,t)\|_{ C^1(\bar{\Omega})}\|\nabla A^{-1}u(\cdot,t) \|_{L^p(\Omega)}\\
&+\| u(\cdot,t)\|_{C^0(\bar\Omega)}\Big(A_0+1+|\Omega|(\chi+A_1+|\Omega|A_2) \| u(\cdot,t)\|_{C^0(\bar\Omega)}+\chi \|A^{-1}u(\cdot,t) \|_{L^p(\Omega)}\Big).
 \end{align*}
 By the regularity and a priori estimates for parabolic equations, \eqref{local-existence-eq1}, and $X^{\alpha} \subset C^1(\bar{\Omega}),$  there a positive constant $C=C(L)$ independent of $t$ such that
\vspace{-0.05in}$$
\|F(t,u(\cdot,t))\|_{L^p(\Omega)} \leq C(1+\|u(\cdot,t)\|_{ \alpha}).
\vspace{-0.05in}$$
Fix $t_1$ such that $ t_0+T_{\max}-\delta_0 <t_1 < t_0+T_{\max}.$ Note that, on $[t_1, t_0+T_{\max})$, we have
\vspace{-0.05in}$$u(\cdot,t)=e^{-A (t-t_1)}u(\cdot,t_1)+\int_{t_1}^t e^{-A(t-s)}F(s, u(\cdot,s))ds.
\vspace{-0.05in}$$
By \cite[Theorem 1.4.3]{DH77} , we get
\vspace{-0.1in}\begin{align*}
\|u(\cdot,t)\|_{\alpha} & \leq \| e^{-A (t-t_1)}u(\cdot,t_1) \|_{\alpha}+\chi \int_{t_1}^t \| A^{\alpha}e^{-A(t-s)}\|
 \|F(s,u(\cdot,s))\| _{p}ds \\
&  \leq C (t-t_1)^{-\alpha} \|u(\cdot,t_1)\|_{C^0(\bar\Omega)}+ C\int_{t_1}^t  (t-s)^{-\alpha}(1+\|u(\cdot,s)\|_{\alpha})ds\\
& \leq \left(C\|u(\cdot,t_1)\|_{C^0(\bar\Omega)}+C(T_{\max}-t_0)\right)(t-t_1)^{-\alpha}+C\int_{t_1}^t  (t-s)^{-\alpha}\|u(\cdot,s)\|_{\alpha}ds.
\end{align*}
Then by Gronwall's inequality (see \cite[page 6]{DH77}), there exists a constant $M=M(C\|u(\cdot,t_1)\|_{C^0(\bar\Omega)}+C(T_{\max}-t_1),\alpha,T_{\max})$ such that
\vspace{-0.05in}\begin{align*}
\|u(\cdot,t)\|_{\alpha} \leq \{C\|u(\cdot,t_1)\|_{C^0(\bar\Omega)}+C(T_{\max}-t_1)\}M(t-t_1)^{-\alpha}.
\vspace{-0.05in}\end{align*}
Thus $\limsup_{t\nearrow T_{\max}}\|u(\cdot,t+t_0)\|_{X^{\alpha}}<\infty$,  a contradiction. Hence if $T_{\max}<\infty$, then \eqref{local-1-eq2}
holds.


\medskip

\noindent {\bf Step 2.}  (Regularity). In this step,  we prove that $u(x,t):=u(x,t;t_0,u_0)$ obtained in (i) is a classical solution of \eqref{u-eq00} on $(t_0,t_0+T_{\max})$ and then $(u(x,t;t_0,u_0),v(x,t;t_0,u_0))$ is a classical solution of \eqref{u-v-eq1} on $(t_0,t_0+T_{\max})$ satisfying the properties in Theorem \ref{thm-001}(1), where $v(\cdot,t;t_0,u_0)=A^{-1}u(\cdot,t;t_0,u_0)$.

Fix $t_0<t_1<T<t_0+T_{\max}$ and consider the problem
\vspace{-0.05in}\begin{equation}
\label{u-tilde-eq0}
\begin{cases}
\tilde{u}_t(x,t)=(\Delta-1)\tilde{u}(x,t)+g(x,t), &\text{$x \in \Omega,\quad t \in (t_1,T)$} \cr
\tilde{u}(x,t_1)=u(x,t_1), &\text{$x \in \Omega$} \cr
\frac{d \tilde{u}}{d n}=0, \quad x\in\p \Omega,
\end{cases}
\end{equation}
where
\vspace{-0.1in}\begin{align*}
g(x,t)= &-\chi\nabla A^{-1}u(x,t) \cdot \nabla u(x,t)  \\
 & +\Big(1+ a_0(x,t)-\chi A^{-1}u(x,t)+ \chi u(x,t)-a_1(x,t)u(x,t)-a_2(x,t)\int_{\Omega}u(\cdot,t)\Big)u(x,t).
 \vspace{-0.1in}\end{align*}
By Lemma \ref{lem-002}, $t\to g(\cdot,t)\in C^\theta(\bar\Omega)$ is H\"older continuous in $t\in (t_0,t_0+T_{\max})$ for some $\theta\in(0,1)$.
Then by \cite[Theorem 15.1, Corollary 15.3]{Amann},  \eqref{u-tilde-eq0} has a unique classical solution
$\tilde u\in C^{2,1}(\bar {\Omega}\times (t_1,T))\cap C^0(\bar {\Omega}\times [t_1,T))$. Moreover, by Lemma \ref{mild-strong-solution-lm},
\vspace{-0.1in}\begin{align*}
\label{u-mild0}
\tilde{u}(\cdot,t)=&e^{-A(t-t_1)}u(\cdot,t_1) -\chi\int_{t_1}^t e^{-A(t-s)}\Big(\nabla  u(\cdot,s)\cdot  \nabla A^{-1}u(\cdot,s)- u(\cdot,s)\big(u(\cdot,s)-A^{-1}u(\cdot,s)\big)\Big)ds\\
&+\int_{t_1}^t e^{-A(t-s)}u(\cdot,s)\Big(1+ a_0(s,\cdot)-a_1(s,\cdot)u(\cdot,s)-a_2(s,\cdot)\int_{\Omega}u(\cdot,s)\Big)ds.
\vspace{-0.01in}\end{align*}
Thus $\tilde{u}(x,t)=u(x,t)$ for $t\in [t_1,T)$ and $u\in C^{2,1}(\bar {\Omega}\times (t_1,T))\cap C^0(\bar {\Omega}\times [t_1,T)).$
Letting $t_1\to t_0$ and $T\to T_{\max}$, we have $u\in C^{2,1}(\bar {\Omega}\times (t_0,t_0+T_{\max}))\cap C^0(\bar {\Omega}\times [t_0,t_0+T_{\max})).$

Let $v(\cdot,t;t_0,u_0)=A^{-1}u(\cdot,t;t_0,u_0)$. We then have that $(u(x,t;t_0,u_0),v(x,t;t_0,u_0))$ is a classical solution of \eqref{u-v-eq1}
on $(t_0,t_0+T_{\max})$ satisfying the properties in Theorem \ref{thm-001}.

\smallskip

\noindent {\bf Step 3.}  (Uniqueness). In this step,  we prove the uniqueness of classical solutions of \eqref{u-v-eq1} satisfying the properties in Theorem \ref{thm-001}(1).

  Suppose that $(u_1(x,t),v_1(x,t))$ and $(u_2(x,t),v_2(x,t))$  are two classical solutions of $\eqref{u-v-eq1}$ on $(t_0,t_0+T_{\max})$  satisfying the properties in Theorem \ref{thm-001}.  First, set $u=u_1-u_2$ and $v=v_1-v_2$.  Then $(u,v)$ satisfies
\vspace{-0.1in}\[
\begin{cases}
u_t=\Delta u-\chi\nabla (u\cdot \nabla v_1)-\chi\nabla (u_2\cdot \nabla v)\\
\quad\quad +u\left(a_0(t,x)-a_1(t)(u_1+u_2\right)-a_2(t,x)\int_{\Omega}u_1)-a_2(t,x)\left(\int_{\Omega}u\right)u_2, &\text{ $x\in \Omega , t>t_0$}\\
\Delta u +u-v=0, &\text{ $x\in \Omega , t>t_0$} \\
\frac{\p u}{\p n}=\frac{\p v}{\p n}=0, &\text{on $\p \Omega$}\\
u(x,t_0)=0, &\text{on $x \in \Omega$.}
\end{cases}
\]

Next,
fix $t_1,T$ such that $t_0<t_1<T<t_0+T_{\max}$. It is clear that, for $t\in [t_1,t_0+T_{\max})$,
\vspace{-0.1in}\begin{align}
\label{uniqueness-eq1}
u(\cdot,t)=&\,  e^{-A(t-t_1)}\big(u_1(\cdot,t_1)-u_2(\cdot,t_1)\big)-\chi\int_{t_1}^t e^{-A(t-s)}\nabla \big[u(\cdot,s)\cdot \nabla v_1(\cdot,s)+u_2(\cdot,s)\cdot \nabla v(\cdot,s)\big]ds\nonumber\\
&+\int_{t_1}^t e^{-A(t-s)}u(\cdot,s)\Big(1+ a_0(s,\cdot)-a_1(s,\cdot)(u_1(\cdot,s)+u_1(\cdot,s))-a_2(s,\cdot)\int_{\Omega}u_1(\cdot,s)\Big)ds\nonumber\\
&-\int_{t_1}^t e^{-A(t-s)}a_2(s,\cdot)\big(\int_{\Omega}u(\cdot,s)\big)u_2(\cdot,s)ds.
\end{align}

Now, fix  $0<\beta <\frac{1}{2}.$  By  regularity and a  priori estimates for
elliptic equations,  \cite[Theorem 1.4.3]{DH77},   Lemma \ref{lem-003}, and \eqref{uniqueness-eq1},   for any $\epsilon \in (0,\frac{1}{2}-\beta),$ we have
\vspace{-0.1in}\begin{align}
\label{uniqueness-eq2}
\|u(\cdot,t)\|_{X^\beta}\leq&  C\|u(\cdot,t_1)\|_{X^\beta}+ C\chi\max_{t_1\leq s\leq T}\|\nabla v_1(\cdot,s)\|_{C^0(\bar{\Omega})}\int_{t_1}^t(t-s)^{-\beta-\epsilon-\frac{1}{2}}\|u(\cdot,s)\|_{X^\beta}ds\nonumber\\
&+C\chi\max_{t_1\leq s\leq T}\|u_2(\cdot,s)\|_{C^0(\bar{\Omega})}\int_{t_1}^t(t-s)^{-\beta-\epsilon-\frac{1}{2}}\|u(\cdot,s)\|_{X^\beta}ds\nonumber\\
&+C\int_{t_1}^t (t-s)^{-\beta}\{1+ A_0+A_1[\max_{t_1\leq s\leq T}(\|u_1(\cdot,s)\|_{C^0(\bar{\Omega})}+\|u_2(\cdot,s)\|_{C^0(\bar{\Omega})})]\}\|u(\cdot,s)\|_{X^\beta}ds\nonumber \\
&+C\int_{t_1}^t (t-s)^{-\beta}A_2|\Omega|\max_{t_1\leq s\leq T}\|u_1(\cdot,s)\|_{C^0(\bar{\Omega})}\|u(\cdot,s)\|_{X^\beta}ds\nonumber\\
& +C\int_{t_1}^t(t-s)^{-\beta}A_2\|u_2(\cdot,s)\|_{C^0(\bar{\Omega})}\|u(\cdot,s)\|_{X^\beta}ds.
\end{align}
Thus there exists a constant positive $C=C(u_1,u_2,v_1,v_2,A_i,T,\beta, \epsilon)$ such that for all $t \in [t_1,T]$,
\vspace{-0.1in}$$\|u(\cdot,t)\|_{X^\beta}\leq  C_1 \|u(\cdot,t_1)\|_{X^\beta}+C \int_{t_1}^t(t-s)^{-\beta-\epsilon-\frac{1}{2}}\|u(\cdot,s)\|_{X^\beta}ds.
$$
Letting $t_1\to t_0$, we have
\vspace{-0.1in}\begin{equation}
\label{uniqueness-eq3}
\|u(\cdot,t)\|_{X^\beta}\leq  C \int_{t_1}^t(t-s)^{-\beta-\epsilon-\frac{1}{2}}\|u(\cdot,s)\|_{X^\beta}ds\quad {\rm for}\quad t\in [t_0,T].
\end{equation}
By \eqref{uniqueness-eq3} and the generalized Gronwall's inequality (see \cite[page 6]{DH77}), we get $u(\cdot,t)=0$ for $t\in [t_0,T]$.
Letting $T \to t_0+T_{\max},$  we get $u(\cdot,t)=0$  for $t\in [t_0,t_0+T_{\max}).$ Since $v(\cdot,t)=A^{-1}u(\cdot,t), $   $v(\cdot,t)=0$ for $t\in [t_0,t_0+T_{\max}).$   Therefore $(u_1(x,t),v_1(x,t))=(u_2(x,t),v_2(x,t))$ for $(x,t)\in \bar{\Omega} \times[t_0,t_0+T_{\max}).$

\medskip

\noindent {\bf Step 4.} (Nonnegativity). In this last step,  we prove the nonnegativity of the classical solutions.
  Since $u(x,t;t_0,u_0)$ is classical solution of  $\eqref{u-eq00},$  by  maximum principle for parabolic equations, we have that
   $u(x,t;t_0,u_0)$ is nonnegative
  (see \cite[Theorem 7 on page 41]{FA64}). And now, since $u(x,t;t_0,u_0)$ is nonnegative, by maximum principle for elliptic equations,
   $v(x,t;t_0,u_0)$ is nonnegative (see \cite[Theorem 18 on  page 53]{FA64}).

\medskip

(2)  We prove (2) by  Banach Fixed Point Theorem and some arguments in (1) and divide the proof into
three steps.  To this end, we first introduce the notion of  generalized mild solution of \eqref{u-eq0}.
 A function $u \in  C^0([t_0, t_0+T), C^0(\bar{\Omega})) $ is called a {\it generalized mild solution} of \eqref{u-eq0} with $u(t_0)=u_0$ if 
\vspace{-0.1in}\begin{align*}
u(t)=&e^{-A(t-t_0)}u_0 -\chi\int_{t_0}^t e^{-A(t-s)}\nabla\cdot ( u(s)  \nabla A^{-1}u(s))ds\\
&+\int_{t_0}^t e^{-A(t-s)}u(s)\Big(1+ a_0(s,\cdot)-a_1(s,\cdot)u(s)-a_2(s,\cdot)\int_{\Omega}u(s)\Big)ds
\end{align*} 
for $t\in [t_0,t_0+T)$.

\noindent {\bf Step 1.}  (Existence of generalized mild  solution). In this step, we  prove  the existence of a unique  generalized mild  solution $u(\cdot,t;t_0,u_0)$ of \eqref{u-eq0}.

 In order to do so,
fix $t_0\in\RR$ and $u_0 \in C^0(\bar{\Omega})$.  For given  $T>0$ and $R>\|u_0\|_{C^0(\bar\Omega)}$, let
\vspace{-0.05in}$$\mathcal{X}_T= C^0([t_0, t_0+T ], C^0(\bar{\Omega}) )
\vspace{-0.05in}$$
with the supremum norm $\|u\|_{\mathcal{X}_T}=\max_{t_0\leq t \leq t_0+T}\|u(t)\|_{ C^0(\bar{\Omega}) }$,  and let
\vspace{-0.05in}$$\mathcal{S}_{T,R}=\left\{u \in \mathcal{X}_T \,|\, \|u\|_{\mathcal{X}_T} \leq R \right \}.
\vspace{-0.05in}$$
 Note that $\mathcal{S}_{T,R}$ is a closed subset of the Banach space $\mathcal{X}_T$.

First, we claim that, for given $u\in \mathcal{S}_{T,R}$ and $t\in [t_0,t_0+T]$, $(Gu)(t)$ is well defined, where
\vspace{-0.05in}\begin{align*}
(Gu)(t)=&e^{-A(t-t_0)}u_0 -\chi\int_{t_0}^t e^{-A(t-s)}\nabla\cdot ( u(s)  \nabla A^{-1}u(s))ds\\
& +\int_{t_0}^t e^{-A(t-s)}u(s)\Big(1+ a_0(s,\cdot)-a_1(s,\cdot)u(s)-a_2(s,\cdot)\int_{\Omega}u(s)\Big)ds
\vspace{-0.05in}\end{align*}
and the integrals are taken in $C^0(\bar\Omega)$.
In fact, for any $u\in \mathcal{S}_{T,R}$, there is $\{u_n\}\subset \mathcal{S}_{T,R}$ satisfying that $\frac{\p u_n}{\p x_i}\in \mathcal{S}_{T,R}$,
$\frac{\p^2 u_n}{\p x_i\p x_j}\in \mathcal{S}_{T,R}$ and
$\|u_{n}-u\|_{\mathcal{X}_T}\to 0\quad {\rm as}\quad n\to\infty.$
This implies that
\vspace{-0.05in}\begin{align*}
(t_0,t)\ni s\to  & e^{-A(t-s)}\nabla\cdot ( u(s)  \nabla A^{-1}u(s))\\
 &\,\, + e^{-A(t-s)}u(s)\Big(1+ a_0(s,\cdot)-a_1(s,\cdot)u(s)-a_2(s,\cdot)\int_{\Omega}u(s)\Big)\in C^0(\bar\Omega)
\end{align*}
 is measurable. Moreover,  choose $p$ such that $n<p$ and then choose $\beta$ such that $\frac{n}{2p}<\beta < \frac{1}{2}$.  By Lemma \ref{lem-002}, we get  $X^{\beta} \subset C^0(\bar{\Omega}).$ Then we have
 \begin{align*}
\|G(u)(t)\|_{ C^0(\bar{\Omega})}  & \leq \|e^{-A(t-t_0)}u_0\|_{C^0(\bar{\Omega})}+\chi\int_{t_0}^t\| e^{-A(t-s)}\nabla\cdot ( u(s)  \nabla A^{-1}u(s))\|_{ C^0(\bar{\Omega})}ds \nonumber\\
&+\int_{t_0}^t\| e^{-A(t-s)}u(s)\Big(1+ a_0(s,\cdot)-a_1(s,\cdot)u(s)-a_2(s,\cdot)\int_{\Omega}u(s)\Big)\|_{ C^0(\bar{\Omega}}ds \nonumber\\
 & \leq \|u_0\|_{C^0(\bar{\Omega})}+C\int_{t_0}^t\| A^{\beta}e^{-A(t-s)}\nabla\cdot ( u(s)  \nabla A^{-1}u(s))\|_{ L^p(\Omega)}ds \nonumber\\
&+C\int_{t_0}^t\| A^{\beta}e^{-A(t-s)}u(s)\Big(1+ a_0(s,\cdot)-a_1(s,\cdot)u(s)-a_2(s)\int_{\Omega}u(s)\Big)\|_{ L^p(\Omega)}ds
\end{align*}
for $t\in [t_0,t_0+T]$.
By regularity and a priori estimates for parabolic equations and Lemma \ref{lem-003}, for any  $\epsilon \in (0,\frac{1}{2}-\beta),$ we have
\vspace{-0.1in}\begin{align}
\label{new-old-eq1}
\|G(u)(t)\|_{ C^0(\bar{\Omega})}  & \leq \|u_0\|_{C^0(\bar{\Omega})}+CR^2\int_{t_0}^t (t-s)^{-\beta-\frac{1}{2}-\epsilon}e^{-\mu(t-s)}ds \nonumber\\
& \quad +CR[1+A_0+R(A_1+|\Omega|A_2)]\int_{t_0}^t (t-s)^{-\beta}e^{-\mu(t-s)}ds\nonumber\\
&\leq \|u_0\|_{C^0(\bar{\Omega})} + CR^2T^{\frac{1}{2}-\beta -\epsilon} + CR[1+A_0+R(A_1+|\Omega|A_2)]T^{1-\beta}
\end{align}
for $t\in[t_0,t_0+T]$.
 The claim then follows.

Next,   fix $R>\|u_0\|_{C^0(\bar\Omega)}$.  We claim that   $G$ maps $\mathcal{S}_{T,R}$ into itself and is a contraction  for $0<T\ll 1$.

We first show that   $G$ maps $\mathcal{S}_{T,R}$ into itself for $0<T\ll 1$. To this end,
 for any $u\in \mathcal{S}_{T,R}$, $0\le \beta <\frac{1}{2}$,  and $t_0<t<t_0+T$,
we have
\begin{align}
\|(Gu)(t)\|_{ X^{\beta}}  & \leq \|A^{\beta}e^{-A(t-t_0)}u_0\|_{L^p(\Omega)}+\chi\int_{t_0}^t\| A^{\beta}e^{-A(t-s)}\nabla\cdot ( u(s)  \nabla A^{-1}u(s))\|_{L^p(\Omega)}ds \nonumber\\
&+\int_{t_0}^t\| A^{\beta}e^{-A(t-s)}u(s)\Big(1+ a_0(s,\cdot)-a_1(s,\cdot)u(s)-a_2(s,\cdot)\int_{\Omega}u(s)\Big)\|_{L^p(\Omega)}ds. \nonumber
\end{align}
Then by \cite[Theorem 1.4.3]{DH77} and  Lemma \ref{lem-003}, for any  $\epsilon \in (0,\frac{1}{2}-\beta),$ we have
\[ \|(Gu)(t))\|_{ X^{\beta}}   \leq C \|u_0\|_{C^0(\bar{\Omega})}(t-t_0)^{-\beta}+ C R^2(t-t_0)^{\frac{1}{2}-\beta-\epsilon}+ CR[1+A_0+R(A_1+|\Omega|A_2)](t-t_0)^{1-\beta}. \]
Thus  $u(t) \in X^{\beta}.$
Choose $\beta$ such $\frac{n}{2p}<\beta<\frac{1}{2}$ and then choose $\theta$ such that $0<\theta<2\beta-\frac{n}{p},$  by Lemma \ref{lem-003},  we have
that  $X^{\beta} \subset C^{\theta}(\bar{\Omega}).$ Thus $u(t) \in  C^{\theta}(\bar{\Omega})$ for any $t_0<t<t_0+T$.

Choose $\sigma$ such that $\beta+\sigma<\frac{1}{2}$  where $\beta$  is as in the above.
Fix any $t_0<T<t_0+T_{\max}$. Then for any $t,h>0$ such that  $t_0<t<t+h<T<t_0+T$,
\begin{align}
(Gu)(t+h)-(Gu)(t) &=\underbrace{ (e^{-Ah}-I)e^{-A(t-t_0)}u_0}_{I_{1}}-\chi\underbrace{\int_{t_0}^t(e^{-Ah}-I) e^{-A(t-s)}\nabla\cdot ( u(s)  \nabla A^{-1}u(s))ds}_{I_{2}}\nonumber\\
&+\underbrace{\int_{t_0}^t(e^{-Ah}-I) e^{-A(t-s)}u(s)(1+ a_0(s,\cdot)-a_1(s,\cdot)u(s)-a_2(s,\cdot)\int_{\Omega}u(s))ds}_{I_{3}} \nonumber\\
&-\chi\underbrace{\int_t^{t+h}e^{-A(t+h-s)}\nabla\cdot ( u(s)  \nabla A^{-1}u(s))ds}_{I_{4}}\nonumber\\
&+\chi\underbrace{\int_t^{t+h}e^{-A(t+h-s)}u(s)(1+ a_0(s,\cdot)-a_1(s,\cdot)u(s)-a_2(s,\cdot)\int_{\Omega}u(s))ds}_{I_{5}}.
\vspace{-0.1in} \end{align}
 Since $e^{-A(t-t_0)}u_0 \in D(A),$ we have
\[A^{\beta}(e^{-Ah}-I)A^{-\sigma}(A^{\sigma}e^{-A(t-t_0)}u_0)=(e^{-Ah}-I)A^{-\sigma}(A^{\sigma+\beta}e^{-A(t-t_0)}u_0).\]
Therefore by \cite[Theorem 1.4.3]{DH77}, we get
\begin{align*}
\|I_1\|_{ X^{\beta}} &=\|(e^{-Ah}-I)A^{-\sigma}(A^{\sigma+\beta}e^{-A(t-t_0)}u_0)\|_{L^p(\Omega)}\leq Ch^{\sigma} \|A^{\sigma+\beta}e^{-A(t-t_0)}u_0\|_{L^p(\Omega)}\\
&\leq Ch^{\sigma}(t-t_0)^{-\sigma-\beta}e^{-(t-t_0)}\|u_0\|_{C^0(\bar\Omega)}.
\end{align*}
By \cite[Theorem 1.4.3]{DH77} and  Lemma \ref{lem-003}, for any  $\epsilon \in (0,\frac{1}{2}-\beta),$ we have
\begin{align*}
\|I_2\|_{ X^{\beta}} &\leq \int_{t_0}^t\|(e^{-Ah}-I)A^{-\sigma}(A^{\sigma+\beta}e^{-A(t-s)}\nabla\cdot ( u(s)  \nabla A^{-1}u(s)))\|_{L^p(\Omega)}ds\nonumber\\
 &\leq Ch^{\sigma}  \int_{t_0}^t\|A^{\sigma+\beta}e^{-A(t-s)}\nabla\cdot ( u(s)  \nabla A^{-1}u(s))\|_{L^p(\Omega)}ds\nonumber\\
 &\leq Ch^{\sigma} \int_{t_0}^t(t-s)^{-\beta-\sigma-\frac{1}{2}-\epsilon}\|u(s)  \nabla A^{-1}u(s)\|_{L^p(\Omega)}ds\\
 &\le Ch^{\sigma}\sup_{t_0\le s\le T}\|u(s)\|^2_{C^0(\bar\Omega)} (t-t_0)^{\frac{1}{2}-\beta-\sigma-\epsilon}.
\end{align*}
By similar arguments as in the estimates of $I_2$, we have
\[\|I_3\|_{ X^{\beta}} \leq  Ch^{\sigma} \sup_{t_0\le s\le T}\|u(s)\|_{C^0(\bar\Omega)} \big(1+\sup_{t_0\le s\le T}\|u(s)\|_{C^0(\bar\Omega)}\big) (t-t_0)^{\frac{1}{2}-\beta-\sigma-\epsilon}.\]
By  Lemma \ref{lem-003}, we have
\begin{align*}
\|I_4\|_{ X^{\beta}} &\leq  \int_t^{t+h}\| A^{\beta}e^{-A(t+h-s)}\nabla\cdot ( u(s)  \nabla A^{-1}u(s))\|_{L^p(\Omega)}ds\nonumber\\
  &\leq C\int_t^{t+h}(t+h-s)^{-\beta}\|u(s)  \nabla A^{-1}u(s)\|_{L^p(\Omega)}ds\leq Ch^{1-\beta} \sup_{t_0\le s\le T}\|u(s)\|^2_{C^0(\bar\Omega)}.
\vspace{-0.1in}\end{align*}
Similarly,  we have
\[ \|I_5\|_{ X^{\beta}} \leq C h^{1-\beta} \sup_{t_0\le s\le T}\|u(s)\|_{C^0(\bar\Omega)} \big(1+\sup_{t_0\le s\le T}\|u(s)\|_{C^0(\bar\Omega)}\big) .\]
Therefore,  $(Gu)(t)$ is locally H\"{o}lder  continuous from $(t_0,t_0+T)$ to $X^{\beta} $ with exponent $\sigma.$
It is clear that  $(Gu)(t)$ is continuous at $t_0$ in $C^0(\bar\Omega)$. Therefore $(Gu)(t)$ is continuous in
$t\in [t_0,t_0+T]$ in $C^0(\bar\Omega)$.  Then by \eqref{new-old-eq1},  $G$ maps $\mathcal{S}_{T,R}$ into itself for $0<T\ll 1$.

 We next show that  $G$ maps $\mathcal{S}_{T,R}$ is a contraction for $0<T\ll 1$. To this end,
 for given $u,w \in \mathcal{S}_{T,R}$, we have
\vspace{-0.1in}\begin{align*}
&\|G(u)(t)-G(w)(t)\|_{ C^0(\bar{\Omega})} \nonumber\\
 & \leq C\int_{t_0}^t\| A^{\beta}e^{-A(t-s)}\nabla\cdot \Big((u(s)-w(s))  \nabla A^{-1}u(s)+w(s) \nabla A^{-1}(u(s)-w(s))\Big)\|_{L^p(\Omega)}ds \nonumber\\
&+C\int_{t_0}^t\| A^{\beta}e^{-A(t-s)}(u(s)-w(s))\Big(1+ a_0(s,\cdot)-a_1(s,\cdot)(u(s)+w(s))-a_2(s,\cdot)\int_{\Omega}u(s)\Big)\|_{ L^p(\Omega)} ds \nonumber \\
&+ C\int_{t_0}^t\| A^{\beta}e^{-A(t-s)}a_2(s,\cdot)w(s)\int_{\Omega}(u(s)-w(s))\|_{ L^p(\Omega)}ds\\
& \leq CRT^{\frac{1}{2}-\beta-\epsilon} \|u-w\|_{\mathcal{X}_T}+CRT^{1-\beta}\Big(A_0+1+2R(A_1+|\Omega|A_2)\Big) \|u-w\|_{\mathcal{X}_T}.
\end{align*}
 Then $G$ is a contraction for $0<T\ll 1$ and the claim follows.

 Now,
By Banach fixed point Theorem, $G$ has a unique fixed point $u \in \mathcal{S}_{T,R}.$ That means $u \in  C^0([t_0, t_0+T ], C^0(\bar{\Omega})) $ and
\vspace{-0.1in}\begin{align*}
\label{u-mild}
u(t)=&e^{-A(t-t_0)}u_0 -\chi\int_{t_0}^t e^{-A(t-s)}\nabla\cdot ( u(s)  \nabla A^{-1}u(s))ds\\
&+\int_{t_0}^t e^{-A(t-s)}u(s)\Big(1+ a_0(s,\cdot)-a_1(s,\cdot)u(s)-a_2(s,\cdot)\int_{\Omega}u(s)\Big)ds.
\end{align*}
Hence $u(\cdot,t;t_0,u_0):=u(t)(x)$ is a  generalized mild solution of \eqref{u-eq0}.
The  generalized mild solution $u(\cdot,t;t_0,u_0)$ may be prolonged by standard method into a maximal  interval $[t_0,t_0+T_{\max})$  such that if $T_{\max}<\infty$ then  $\limsup_{t \nearrow T_{\max}}\| u(\cdot,t+t_0;t_0,u_0) \|_{C^0(\bar{\Omega})}=\infty.$

\smallskip

\noindent {\bf Step 2.} (Regularity). In this step, we prove that $u(t)=u(\cdot,t;t_0,u_0)$ is a classical solution of \eqref{u-eq00} satisfying the properties in Theorem \ref{thm-001}(2), where $u(\cdot,t;t_0,u_0))$ is obtained in Step 1. Then $(u(x,t;t_0,u_0),v(x,t;t_0,u_0))$
with $v(\cdot,t;t_0,u_0)=A^{-1} u(\cdot,t;t_0,u_0)$ is a classical solution of
\eqref{u-v-eq1} satisfying the properties in Theorem \ref{thm-001}(2).

 First, for any $0\le \beta <\frac{1}{2}$ and $\sigma$ such that $\beta+\sigma<\frac{1}{2}$, by the arguments in Step 1,
 $u(t)$ is locally H\"{o}lder  continuous from $(t_0,t_0+T_{\max})$ to $X^{\beta} $ with exponent $\sigma.$

Next, fix $\frac{1}{2}<\alpha<1$. We define the map $B(t): (t_0,t_0+T_{\max}) \to \mathcal{L}(X^{\alpha},L^p(\Omega)) $ by
 \[B(t)\tilde{u}=-\chi\nabla A^{-1}u(t) \cdot \nabla \tilde{u}   +\Big( a_0(t,\cdot)-\chi A^{-1}u(t)+ \chi u(t)-a_1(t,\cdot)u(t)-a_2(t,\cdot)\int_{\Omega}u(t)\Big)\tilde{u}.\]
 We claim that $B$ is well defined and is H\"older continuous in $t$.
 Indeed,
$B(t)$ is linear in $\tilde{u},$ and since $X^{\alpha} \subset C^1(\bar{\Omega}),$ we get
\begin{eqnarray}
\label{Umild001}
\|B(t)\tilde{u}\|_{L^p(\Omega)} &\le& \chi\|\nabla A^{-1}u(t)\|_{L^p(\Omega)} \|\tilde{u}\|_{C^1(\bar{\Omega})}+A_0\|\tilde{u}\|_{C^1(\bar{\Omega})}\nonumber\\
&&+(\chi+A_1+|\Omega|^{p-1}A_2)\|u(t)\|_{L^p(\Omega)}\|\tilde{u}\|_{C^1(\bar{\Omega})}+\chi\|A^{-1}u(t)\|_{L^p(\Omega)}\|\tilde{u}\|_{C^1(\bar{\Omega})} \nonumber\\
&\le& C \big(1+\|u(t)\|_{C^0(\bar\Omega)} \big)\|\tilde{u}\|_{X^{\alpha}}.
\end{eqnarray}
Thus $B(t) \in  \mathcal{L}(X^{\alpha},L^p(\Omega))$ and $B$ is well defined. Moreover,
\begin{eqnarray*}
\label{Umild0001}
&&\|B(t+h)\tilde{u}-B(t)\tilde{u}\|_{L^p(\Omega)}\nonumber\\
&& \le \chi\|\nabla A^{-1}u(t+h)-\nabla A^{-1}u(t)\|_{L^p(\Omega)} \|\tilde{u}\|_{C^1(\bar{\Omega})}+\|a_0(t+h,\cdot)-a_0(t,\cdot)\|_{C^0(\bar\Omega)} \|\tilde{u}\|_{C^1(\bar{\Omega})}\nonumber\\
&&+\|u(t)\|_{L^p(\Omega)}\|a_1(t+h,\cdot)-a_1(t,\cdot)\|_{C^0(\bar\Omega)}\|\tilde{u}\|_{C^1(\bar{\Omega})}\nonumber\\
&&+(\chi+A_1+|\Omega|^{p-1}A_2)\|u(t+h)-u(t)\|_{L^p(\Omega)}\|\tilde{u}\|_{C^1(\bar{\Omega})} \nonumber\\
&&+\chi \|A^{-1}u(t+h)-A^{-1}u(t)\|_{L^p(\Omega)}   \|\tilde{u}\|_{C^1(\bar{\Omega})}+|\Omega|^{p-1}\|u(t)\|_{L^p(\Omega)}\|a_2(t+h,\cdot)-a_2(t,\cdot)\|_{C^0(\bar\Omega)}  \|\tilde{u}\|_{C^1(\bar{\Omega})}.
\end{eqnarray*}
Then by elliptic regularity, {\bf(H1)}, and the fact that $u(t)$ is locally H\"older continuous with respect to the $C^0(\bar{\Omega})$-norm, we have that $B(t)$ is H\"older continuous in $t$.

Finally, fix any $t_1\in (t_0,t_0+T_{\max})$. By \cite[Theorem 7.1.3]{DH77}, we have that
\begin{equation}
\label{u-tilde-eq1}
\begin{cases}
\tilde{u}_t=\Delta \tilde{u} +B(t)\tilde{u}, \quad  t \in (t_1,t_0+T_{\max}) \cr
\tilde{u}(t_1)=u(t_1)
\end{cases}
\end{equation}
has a unique strong solution $\tilde{u}$ which satisfy $\tilde u(t) \in X^{\gamma}$ for any $\gamma<1$ and $t_1<t<t_0+T_{\max}.$  By Lemma \ref{mild-strong-solution-lm}(2), $\tilde{u}$ is given by the formula
\begin{align*}
\label{u-mild}
\tilde{u}(t)=&e^{-A(t-t_1)}u(t_1) -\chi\int_{t_1}^t e^{-A(t-s)}\nabla\cdot ( \tilde{u}(s)  \nabla A^{-1}u(s))ds\\
&+\int_{t_1}^t e^{-A(t-s)}\tilde{u}(s)\Big(1+ a_0(s,\cdot)-a_1(s,\cdot)u(s)-a_2(s,\cdot)\int_{\Omega}u(s)\Big)ds.
\end{align*}
Fix $t_0<t_1<t_2<t_0+T_{\max}$. We have by Lemma \ref{lem-003} with $\beta<\frac{1}{2}$ and $\epsilon \in (0, \frac{1}{2}-\beta)$  that
\begin{align*}
\|\tilde{u}(t)-u(t)\|_{C^0(\bar{\Omega})} &\leq   C \int_{t_1}^t(t-s)^{-\frac{1}{2}-\beta-\epsilon} \| \tilde{u}(s)-u(s)\|_{C^0(\bar{\Omega})} ds  + C  \int_{t_1}^t(t-s)^{-\beta} \| \tilde{u}(s)-u(s)\|_{C^0(\bar{\Omega})}ds\\
&\leq C\int_{t_1}^t(t-s)^{-\frac{1}{2}-\beta-\epsilon} \| \tilde{u}(s)-u(s)\|_{C^0(\bar{\Omega})} ds
\end{align*}
for $t_1\le t\le t_2$ and some $C=C(\sup_{t_1\le t\le t_2}\|u(t)\|_{C^0(\bar\Omega)})$.
Then by generalized Gronwall's inequality (see \cite[page 6]{DH77}), we get $\tilde{u}(t)=u(t)$ in $C^0(\bar{\Omega})$ on $[t_1,t_2]$. Letting
$t_1\to t_0$ and $t_2\to t_0+T_{\max}$, we have $\tilde{u}(t)=u(t)\in X^\gamma$ for any $0\le \gamma<1$ and $t\in (t_0,t_0+T_{\max})$.
It then follows from Theorem \ref{thm-001}(1) that $u(x,t;0,u_0):=u(t)(x)$ is a classical solution of \eqref{u-eq00} satisfying the properties
in Theorem \ref{thm-001}(2).

\medskip

\noindent{\bf Step 3.} (Nonnegativity and uniqueness) By the similar arguments as in Steps 3 and 4 in the proof of  Theorem \ref{thm-001}(1),
we have that $(u(x,t;t_0,u_0),v(x,t;t_0,u_0))$ is the unique nonnegative classical solution of \eqref{u-v-eq1} satisfying Theorem \ref{thm-001},
where $v(\cdot,t;t_0,u_0)=A^{-1}u(\cdot,t;t_0,u_0)$.
\end{proof}

\begin{remark}
\label{hull-rk}
Let $\{t_n\}\subset \RR$. Suppose that $\lim_{n\to\infty} a_i(t+t_n,x)=\hat a_i(t,x)$ locally uniformly in $(t,x)\in\RR\times\bar\Omega$.
Then $\hat a_i(t,x)$ $(i=0,1,2)$ also satisfy the hypothesis (H1)  in the introduction. Hence for any $t_0\in\RR$ and $u_0\in X^\alpha$ or $u_0\in C^0(\bar\Omega)$, \eqref{u-v-eq1} with $a_i(t,x)$ being replaced by $\hat a_i(t,x)$ $(i=0,1,2)$ has also a unique solution
$(\hat u(x,t;t_0,u_0), \hat v(x,t;t_0,u_0))$ satisfying the properties in Theorem \ref{thm-001}(1) or (2).
\end{remark}

The following corollary  follows directly from Theorem \ref{thm-001} and its proof.

\begin{corollary}
\label{cor1}
\begin{itemize}
\item[(1)]
Let $t_0\in\RR$ and $u_0\in X^\alpha$ or $C^0(\bar \Omega)$ be given and let $(u(x,t;t_0,u_0)$, $v(x,t;t_0,u_0))$ be the unique
solution of \eqref{u-v-eq1} with initial condition $u(\cdot,t_0;t_0,u_0)=u_0(\cdot)$ in Theorem \ref{thm-001}(1) or (2). For any $t_0<t_1<t_2<t_0+T_{\max}$, there holds
$$
(u(x,t_2;t_0,u_0),v(x,t_2;t_0,u_0))=(u(x,t_2;t_1,u(\cdot,t_1;t_0,u_0)),v(x,t_2;t_1,u(\cdot,t_1;t_0,u_0)).
$$

\item[(2)] Let $(u(x,t;t_0,u_0)$, $v(x,t;t_0,u_0))$ be the unique
solution of \eqref{u-v-eq1} with initial condition $u(\cdot,t_0;t_0,u_0)=u_0(\cdot)\in X$ in Theorem \ref{thm-001}(1) or (2), where $X=X^\alpha$ or
$C^0(\bar\Omega)$. Then $\RR\times X\ni (t_0,u_0)\mapsto (u(\cdot,t;t_0,u_0),v(\cdot,t;t_0,u_0))\in X\times X$ is continuous locally uniformly
with respect to $t\in (t_0, t_0+T_{\max})$.

\item[(3)]
Let $\{t_n\}\subset \RR$. Suppose that $\lim_{n\to\infty} a_i(t+t_n,x)=\hat a_i(t,x)$ locally uniformly in $(t,x)\in\RR\times\bar\Omega$.
For given $t_0\in\RR$ and $u_0\in X^\alpha$ or $C^0(\bar\Omega)$, let $(u_n(x,t;t_0,u_0),v_n(x,t;t_0,u_0))$ be the solution of \eqref{u-v-eq1}
with $a_i(t,x)$ being replaced by $a_i(t+t_n,x)$ $(i=0,1,2)$ and with
 initial condition $u_n(\cdot,t_0;t_0,u_0)=u_0(\cdot)$ and $(\hat u(x,t;t_0,u_0),\hat v(x,t,;t_0,u_0))$ be the solution of \eqref{u-v-eq1}
 on $(t_0,t_0+\hat T_{\max})$
with $a_i(t,x)$ being replaced by $\hat a_i(t,x)$ $(i=0,1,2)$ and with initial condition $\hat u(\cdot,t_0;t_0,u_0)=u_0(\cdot)$. Then
for any $t\in (t_0,t_0+\hat T_{\max})$,
$$
\lim_{n\to\infty} (u_n(\cdot,t;t_0),v_n(\cdot,t;t_0,u_0))=(\hat u(\cdot,t;t_0,u_0),\hat v(\cdot,t;t_0,u_0))
\quad {\rm in}\quad C^0(\bar\Omega).$$
\end{itemize}

\end{corollary}

\section{Global existence and uniform boundedness of classical solutions}

In this section, we investigate the global existence and the uniform boundedness of classical solutions of $\eqref{u-v-eq1}$ with given initial functions and prove Theorem \ref{thm-002}. We first prove two important lemmas.

Consider the following  Lotka-Volterra Competition system of ordinary differential equations,
\vspace{-0.05in}\begin{equation}
\label{ode0}
\begin{cases}
\overline{u}'=(\chi)_+ \overline{u}(\overline {u}-\underline{u})+ \overline{u}\left[a_{0,\sup}(t)-a_{1,\inf}(t)\overline u-
|\Omega|\Big(a_{2,\inf}(t)\Big)_+ \underline{u}+|\Omega|\Big(a_{2,\inf}(t)\Big)_-\overline{ u}\right]\\
\underline{u}'=(\chi)_+  \underline{u}(\underline {u}-\overline{u})+ \underline{u}\left[a_{0,\inf}(t)-a_{1,\sup}(t)\underline u-
 |\Omega|\Big(a_{2,\sup}(t)\Big)_+ \overline{u}+|\Omega|\Big(a_{2,\sup}(t)\Big)_-\underline u\right].
\end{cases}
\vspace{-0.05in}\end{equation}
For given $u_0\in C^0(\bar\Omega)$ with $u_0(x)\ge 0$ and $t_0\in\RR$, let
$\overline{u}_0=\max_{x \in \bar{\Omega}} u_0(x)$,
$\underline{u}_0=\min_{x \in \bar{\Omega}} u_0(x)
$
and
\begin{equation}
\label{u-under-above-bar-eq}(\overline{u}(t),\underline{u}(t))=(\overline{u}(t;t_0,\overline{u}_0,\underline{u}_0), \underline{u}(t;t_0,\overline{u}_0,\underline{u}_0))
\end{equation}
 be the solution of
\eqref{ode0} with $(\overline{u}(t_0;t_0,\overline{u}_0,\underline{u}_0), \underline{u}(t_0;t_0,\overline{u}_0,\underline{u}_0))=(\overline{u}_0,\underline{u}_0)$.

\begin{lemma}
\label{lem-004}
Suppose  $\inf_{t\ge t_0} \Big\{a_{1,\inf}(t)-|\Omega|\Big(a_{2,\inf}(t)\Big)_-\Big\}>(\chi)_+ $. Then
$(\bar u(t),\underline u(t))$ exists for all $t>t_0$ and
\begin{equation}
\label{eq-sup-sub-001}
0\leq \underline{u}(t) \leq \overline{u}(t) \, \,\, \forall  t\ge t_0.
\end{equation}
 Moreover, $0\leq \overline{u}(t) \leq \max\Big\{\overline{u}_0, \frac{a_{0,\sup}}{ \inf_{t\ge t_0} \Big\{a_{1,\inf}(t)-|\Omega|\Big(a_{2,\inf}(t)\Big)_--(\chi)_+ \Big\}}\Big\}$.

\end{lemma}

\begin{proof} First, note that
\vspace{-0.05in}$$
\inf_{t\ge t_0} \Big\{a_{1,\sup}(t)-|\Omega| \Big(a_{2,\sup}(t)\Big)_-\Big\}\ge \inf_{t \in \mathbb{R}} \Big\{a_{1,\inf}(t)-|\Omega|\Big(a_{2,\inf}(t)\Big)_-\Big\}>(\chi)_+ .
\vspace{-0.05in}$$
The  existence of $(\bar u(t),\underline u(t))$ for all $t>t_0$  is then clear.
 For any $\epsilon >0$,  let $\overline{u}^\epsilon_0=\overline{u}_0+\epsilon$ and  $a_{0,\sup}^\epsilon(t)=a_{0,\sup}^\epsilon(t)+\epsilon$. Let $$(\overline{u}^\epsilon(t),\underline{u}^\epsilon (t))=(\overline{u}^\epsilon(t;t_0,\overline{u}_0^\epsilon,\underline{u}_0), \underline{u}^\epsilon(t;t_0,\overline{u}_0^\epsilon,\underline{u}_0)),$$
  where $(\overline{u}^\epsilon(t;t_0,\overline{u}_0^\epsilon,\underline{u}_0), \underline{u}^\epsilon(t;t_0,\overline{u}_0^\epsilon,\underline{u}_0))$ is the solution of \eqref{ode0} with $a_{0,\sup}(t)$ being replaced
 by $a_{0,\sup}^\epsilon(t)$ and $(\overline{u}^\epsilon(t_0;t_0,\overline{u}_0^\epsilon,\underline{u}_0), \underline{u}^\epsilon(t_0;t_0,\overline{u}_0^\epsilon,\underline{u}_0))= (\overline{u}_0^\epsilon,\underline{u}_0)$.
We claim that  $0\leq \underline{u}^\epsilon(t) \leq \overline{u}^\epsilon(t)$ for all $t\ge t_0$.
 Suppose by contradiction that this claim  does not hold.
Then  since $0\leq \underline{u}_0 <\overline{u}^\epsilon_0,$ there exist $\overline{t} \in (t_0,\infty)$ such that
 $$\underline{u}^\epsilon(t)<\overline{u}^\epsilon(t),\, \forall t \in [t_0, \overline{t}) \,\,\, \text{and} \, \,\, \underline{u}(\overline{t})=\overline{u}^\epsilon(\overline{t}).$$
Thus $(\overline{u}^\epsilon-\underline{u}^\epsilon)'(\overline{t}) \leq 0$.  Note that $\overline{u}^\epsilon(t)>0$ for $t\ge t_0$.
Using   \eqref{ode0} at $t=\overline{t}, $ we get
\begin{align*}
(\overline{u}^\epsilon-\underline{u})^{'}(\overline{t})= &\overline{u}^\epsilon(\overline{t})\Big[a_{0,\sup}^\epsilon(\overline{t})-
a_{0,\inf}(\overline{t})\\
& +\{a_{1,\sup}(\overline{t})-a_{1,\inf}(\overline{t})+|\Omega|(a_{2,\sup}(\overline{t})-a_{2,\inf}(\overline{t}))\}  \overline{u}^\epsilon(\overline{t})\Big].
\end{align*}
It then follows that $(\overline{u}^\epsilon-\underline{u})'(\overline{t})\ge 0$,
which implies that $(\overline{u}^\epsilon-\underline{u})'(\overline{t})=0$ and then
$$
 0=a_{0,\sup}^\epsilon(\overline{t})-
a_{0,\inf}(\overline{t})+\{a_{1,\sup}(\overline{t})-a_{1,\inf}(\overline{t})+|\Omega|(a_{2,\sup}(\overline{t})-a_{2,\inf}(\overline{t}))\}  \overline{u}^\epsilon(\overline{t})>0,
$$
which is a contradiction.
 Thus the claim holds.
Letting $\epsilon \to 0$ and using continuity of solutions of \eqref{ode0} with respect to initial data and coefficients,  \eqref{eq-sup-sub-001} follows.

 Furthermore, we have
\vspace{-0.1in}\begin{align}
                                  \overline{u}'&=(\chi)_+  \overline{u}(\overline {u}-\underline{u})+ \overline{u}\left[a_{0,\sup}(t)-a_{1,\inf}(t)\overline u- |\Omega|\Big(a_{2,\inf}(t)\Big)_+ \underline{u}+|\Omega|\Big(a_{2,\inf}(t)\Big)_-\overline{ u}\right]\nonumber\\
                                                                                               & \leq  \overline{u}\Big[a_{0,\sup}(t)-
                                  \Big \{ a_{1,\inf}(t)-|\Omega|\Big(a_{2,\inf}(t)\Big)_--(\chi)_+ \Big\} \overline{u}\Big].
\end{align}
Thus if $\inf_{t\ge t_0} \Big \{ a_{1,\inf}(t)-|\Omega|\Big(a_{2,\inf}(t)\Big)_-\Big\}>(\chi)_+ $,  by comparison principle, we have
\vspace{-0.1in} \[0<\overline{u}(t) \leq \max\Big\{\overline{u}_0, \frac{a_{0,\sup}}{\inf_{t\ge t_0}  \big \{ a_{1,\inf}(t)-|\Omega|\Big(a_{2,\inf}(t)\Big)_--(\chi)_+ \big\}}\Big\}.\]
\end{proof}

\vspace{-0.1in}\begin{lemma}
\label{lem-004b}
Suppose $\inf_{t \in \mathbb{R}} \Big \{ a_{1,\inf}(t)-|\Omega|\Big(a_{2,\inf}(t)\Big)_-\Big \}>0.$ Then
\[
0 \leq  \int_{\Omega}u(t) \leq \max\Big\{\int_{\Omega}u_0(x),\frac{|\Omega| a_{0,\sup}}{\inf_{t \in \mathbb{R}} \big \{ a_{1,\inf}(t)-|\Omega|\Big(a_{2,\inf}(t)\Big)_-\big \}} \Big\}:=M_0(\|u_0\|_{L^1},a_0,a_1,a_2,|\Omega|)
\]
for all $t\in [t_0,t_0+T_{\max})$, where $u(t)=u(\cdot,t;t_0,u_0)$ and $u_0\in C^0(\bar\Omega)$ with $u_0(x)\ge 0$.
\end{lemma}

\begin{proof}
By integrating the  first equation of  $\eqref{u-v-eq1}$  over $\Omega,$ we get for any $t\in[t_0,t_0+T_{\max})$ that
\begin{align*}
\frac{d}{dt}\int_{\Omega}u(t)&= \int_{\Omega}u(t) \Big\{ a_0(t,x)-a_1(t,x)u(t)-a_2(t,x) \int_{\Omega}u(t) \Big\} \\
&\leq \int_{\Omega}u(t)\Big \{ a_{0,\sup}-a_{1,\inf}(t)u(t)-\big(a_{2,\inf}(t)\big)_+\int_{\Omega}u(t)+\big(a_{2,\inf}(t)\big)_-\int_{\Omega}u(t)\Big\}\nonumber\\
&\leq \int_{\Omega}u(t)\Big \{ a_{0,\sup}-\frac{1}{|\Omega|}\Big [{ a_{1,\inf}(t)}-|\Omega|\big(a_{2,\inf}(t)\big)_- \Big ]\int_{\Omega}u(t)\Big\}
\end{align*}
Thus if $\inf_{t \in \mathbb{R}} \Big \{ a_{1,\inf}(t)-|\Omega|\big(a_{2,\inf}(t)\big)_-\Big \}>0,$ we get by comparison principle for ODEs that
\vspace{-0.05in}$$
0 \leq  \int_{\Omega}u(t) \leq \max\Big\{\int_{\Omega}u_0(x),\frac{|\Omega| a_{0,\sup}}{\inf_{t \in \mathbb{R}} \big \{ a_{1,\inf}(t)-|\Omega|\big(a_{2,\inf}(t)\big)_-\big \}} \Big\}.
$$
\end{proof}

	We now prove Theorem \ref{thm-002} for $\chi>0$ and the proof for the  case $\chi\leq 0$ follows from similar arguments with proper adaptations.

\vspace{-0.1in}\begin{proof} [Proof of Theorem \ref{thm-002}]

(1) Let $(\bar u(t),\underline u(t))$ be as in \eqref{u-under-above-bar-eq}.
 It suffices to prove that $0\leq \underline{u}(t) \leq u(x,t;t_0,u_0) \leq \overline{u} (t)$ for all $ t_0\le t<t_0+T_{\max}$ and $x \in \bar{\Omega}$.

 Observe that for any $\epsilon>0$,      there exists $t_0<t_{\epsilon }<t_0+T_{\max}$ such that
\vspace{-0.1in}$$
\underline{u}(t)-2\epsilon <u(x,t;t_0,u_0) <\overline{u}(t)+2\epsilon, \quad \text{ for all $(x, t) \in \Omega \times [t_0, t_{\epsilon})$.}
\vspace{-0.1in}$$
Let
\vspace{-0.1in}$$T_\epsilon=\sup\{t_\epsilon\in (t_0,t_0+T_{\max})\,|\, \underline{u}(t)-2\epsilon <u(x,t;t_0,u_0) <\overline{u}(t)+2\epsilon \quad \forall\,\, (x, t) \in \Omega \times [t_0, t_{\epsilon})\}.
\vspace{-0.05in}$$
It then suffices to prove  that $T_\epsilon= t_0+T_{\max}$.

Assume by contradiction that $T_\epsilon<t_0+T_{\max}$.
Then there is  $x_0\in\bar\Omega$ such that
\vspace{-0.05in}$$
{\rm either}\,\,\, u(x_0, T_{\epsilon};t_0,u_0)=\underline{u}(T_{\epsilon })-2\epsilon\,\,\, {\rm    or}\,\,\, u(x_0, T_{\epsilon};t_0,u_0)=\overline{u}(T_{\epsilon })+2\epsilon.
$$
Let $\overline{U}(x,t)=u(x,t;t_0,u_0)-\overline{u}(t)$ and $\underline{U}(x,t)=u(x,t;t_0,u_0)-\underline{u}(t).$

Note that  for $t\in (t_0,t_0+T_{\max})$, $\overline{U}$ satisfies
 \vspace{-0.05in} \begin{align*}
  \overline{U}_t-\Delta \overline{U}\le &-\chi  \nabla\overline{U} \cdot \nabla v
   +\overline{U} \left[a_{0,\sup}(t)-\Big(a_{1,\inf}(t)-\chi\Big)(u+\overline{u})  -  \chi \underline{u}\right] \\
   &-\chi u(v-\underline{u})
                                                                     -a_{2,\inf}(t) (\int_{\Omega}u)u+  |\Omega|\Big(a_{2,\inf}(t)\Big)_+ \underline{u}\overline{u}-|\Omega|\Big(a_{2,\inf}(t)\Big)_-\overline{ u}^2 . \end{align*}
By $
 -a_{2,\inf}(t) (\int_{\Omega}u)u= -\Big(a_{2,\inf}(t) \Big)_+(\int_{\Omega}u)u+\Big(a_{2,\inf}(t) \Big)_-(\int_{\Omega}u)u$,
we get for $t\in (t_0,t_0+T_{\max})$ that
  \vspace{-0.05in}\begin{align}
  \label{new-eq1}
  \overline{U}_t-\Delta \overline{U}\le &-\chi  \nabla\overline{U} \cdot \nabla v
   +\overline{U} \left[a_{0,\sup}(t)-\Big(a_{1,\inf}(t)-\chi\Big)(u+\overline{u})  -  \chi \underline{u}\right] \nonumber\\
   &-\chi u(v-\underline{u})
                                                                    -\Big(a_{2,\inf}(t) \Big)_+\Big((\int_{\Omega}u)u-|\Omega|\underline{u}\overline{u}\Big)+\Big(a_{2,\inf}(t) \Big)_-(\int_{\Omega}u)u-|\Omega|\Big(a_{2,\inf}(t)\Big)_-\overline{ u}^2\nonumber\\
                                                                     &\le -\chi  \nabla\overline{U} \cdot \nabla v
   +\overline{U} \left[a_{0,\sup}(t)-\Big(a_{1,\inf}(t)-\chi\Big)(u+\overline{u})  -  \chi \underline{u}\right] \nonumber\\
   &-\chi u(v-\underline{u})
                                                                    -\Big(a_{2,\inf}(t) \Big)_+\Big((\int_{\Omega}u)u-|\Omega|\underline{u}\overline{u}\Big)\nonumber\\
&+\Big(a_{2,\inf}(t) \Big)_-\big(u\int_{\Omega}(u-\overline{u})\big)-|\Omega|\Big(a_{2,\inf}(t)\Big)_-\overline{ u}(\overline{ u}- u)\nonumber\\
&\le -\chi  \nabla\overline{U} \cdot \nabla v
   +\overline{U} \left[a_{0,\sup}(t)-\Big(a_{1,\inf}(t)-\chi\Big)(u+\overline{u})  -  \chi \underline{u}+|\Omega|\Big(a_{2,\inf}(t)\Big)_-\overline{ u}\right]\nonumber \\
   &-\chi u(v-\underline{u})
                                                                    -\Big(a_{2,\inf}(t) \Big)_+\Big((\int_{\Omega}u)u-|\Omega|\underline{u}\overline{u}\Big)
+\Big(a_{2,\inf}(t) \Big)_-(\int_{\Omega} \overline{U})u.  \end{align}

We claim that $\int_{\Omega}\overline U_+^2(x,t)dx$  is weakly differentiable in $t$ and moreover
\begin{equation}
\label{new-eq2}
\frac{d}{dt}\int_{\Omega}\overline U_+^2(x,t)dx=2\int_{\Omega}\overline U_+(x,t)\bar U_t(x,t)dx\quad {\rm for}\quad a.e.\, t\in (t_0,t_0+T_{\max}),
\end{equation}
and
\begin{equation}
\label{new-eq2-1}
\int_\Omega \overline U_+^2(x,t)dx=\int_\Omega \overline U_+^2(x,t_0)dx+\int_{t_0}^t \Big(\frac{d}{dt}\int_{\Omega}\overline U_+^2(x,\tau)dx\Big)d\tau\quad \forall\,\, t\in (t_0,t_0+T_{\max}).
\end{equation}
   In order to prove the claim we define for $r>0,$
\[
F_r(z)=
\begin{cases}
(z^{2}+r)^\frac{1}{2}-r, &\text{if $ z>0$;}\\
0, &\text{if $z \leq 0$.}
\end{cases}
\]
Then $F_r \in C^1(\mathbb{R}),$
\[
F_r^{'}(z)=
\begin{cases}
z(z^{2}+r)^{-\frac{1}{2}}, &\text{if $ z>0$;}\\
0, &\text{if $z \leq 0$.}
\end{cases}
\]
Note that $|F{_r}^{'}|\leq 1$ and that we have the following pointwise convergence,
$$\overline U_+(x,t)=\lim_{r \to 0}F_r( \overline U(x,t)).
$$
This implies that
\begin{equation}
\label{new-eq3}
\int_\Omega \overline U_+^2(x,t)dx=\lim_{r\to 0}\int_\Omega F_r^2(\overline U(x,t))dx\quad \forall \,\, t\in (t_0,t_0+T_{\max}).
\end{equation}
Note also that  $\int_\Omega F_r^2(\overline  U(x,t))dx$ is differentiable in $t$ and
\begin{equation}
\label{new-eq3-1}
\frac{d}{dt}\int_\Omega F_r^2(\overline  U(x,t))dx=2\int_\Omega\big((\overline U_+^2(x,t)+r)^{\frac{1}{2}}-r\big) \overline U_+(x,t)(\overline U_+^2(x,t)+r)^{-\frac{1}{2}}\bar U_t(x,t)dx.
\end{equation}
 By \eqref{new-eq3-1}, for any $\delta>0$, there is $M_\delta>0$ such that for any $r>0$
\begin{equation}
\label{new-eq3-2}
\Big|\int_\Omega F_r^2(\overline U(x,t_1))dx-\int_\Omega  F_r^2(\overline U(x,t_2))dx\Big|\le M_\delta |t_1-t_2|\quad \forall\,\, t_1,t_2\in [t_0+\delta,t_0+T_{\max}-\delta].
\end{equation}
Then by \eqref{new-eq3} and \eqref{new-eq3-2}, we have
\begin{equation}
\label{new-eq3-3}
\Big| \int_\Omega \overline U_+^2(x,t_1)dx-\int_\Omega \overline U_+^2(x,t_2)dx\Big|\le M_\delta |t_1-t_2|\quad \forall\,\, t_1,t_2\in [t_0+\delta,t_0+T_{\max}-\delta].
\end{equation}
Let $\phi \in C^\infty_c((t_0,t_0+T_{\max}))$. We have by integration by part that
\begin{equation}
\label{new-eq4}
\int_{t_0}^{T_{\max}}\frac{d}{dt}\left(\int_{\Omega}F_r(\overline U(x,t))^2dx\right)\phi(t)dt=-\int_{t_0}^{T_{\max}}\left(\int_{\Omega}F_r(\overline U(x,t))^2dx\right)\phi_t(t)dt.
\end{equation}
By Lebesgue Dominated Theorem we get  from \eqref{new-eq3} that
$$\lim_{r \to 0}\left(-\int_{t_0}^{T_{\max}}\left(\int_{\Omega}F_r(\overline U(x,t))^2dx\right)\phi_t(t)dt\right)=-\int_{t_0}^{T_{\max}} \int_{\Omega}(\overline U_+(x,t))^2dx\phi_t(t)dt,$$
and  from \eqref{new-eq3-1} that
$$\lim_{r \to 0}\int_{t_0}^{T_{\max}}\frac{d}{dt}\left(\int_{\Omega}F_r(\overline U(x,t))^2dx\right)\phi(t)dt=2\int_{t_0}^{T_{\max}}\int_{\Omega}
 \overline U_+(x,t) \bar U_t(x,t)dx\phi(t)dt.$$
Thus it follows from equations \eqref{new-eq4} that
$$
\int_{t_0}^{T_{\max}} \int_{\Omega}(\overline U_+(x,t))^2dx\phi_t(t)dt=-2\int_{t_0}^{T_{\max}} \int_{\Omega} \overline U_+(x,t) \bar U_t(x,t)dx\phi(t)dt.
$$
This implies that $\int_{\Omega}\overline U_+^2(x,t)dx$  is weakly differentiable and \eqref{new-eq2} holds.
 By \eqref{new-eq2}, \eqref{new-eq3-3}, and the Fundamental Theorem of Calculus for Lebesgue Integrals,
we have for any $t,t_1\in (t_0,t_0+T_{\max})$ that
\begin{equation}
\label{new-eq2-2}
\int_\Omega \overline U_+^2(x,t)dx=\int_\Omega \overline U_+^2(x,t_1)dx+\int_{t_1}^t \Big(\frac{d}{dt}\int_\Omega \overline U_+^2(x,\tau)dx\Big)d\tau.
\end{equation}
Letting $t_1\to t_0$, \eqref{new-eq2-1} follows.

By \eqref{new-eq2},  multiplying \eqref{new-eq1} by $\overline{U}_{+}$ and integrating with respect to $x$ over $\Omega$, we get
\begin{eqnarray*}
  & & \frac{1}{2}\frac{d}{dt}\int_{\Omega}(\overline{U}_{+})^2       +\int_{\Omega}|\nabla (\overline{U}_{+}) |^2\\
 &  & \le \int_{\Omega}(\overline{U}_{+})^2 \left[  a_{0,\sup}(t)+\chi \frac{1}{2}u- \frac{1}{2} \chi v  -\Big(a_{1,\inf}(t)-\chi \Big)(u+\overline{u})  -  \chi \underline{u}  + |\Omega|\Big(a_{2,\inf}(t)\Big)_-\overline{ u}  \right]                 \nonumber \\
&&\quad  -\chi \int_{\Omega}(\overline{U}_{+})u(v-\underline{u}) -\Big(a_{2,\inf}(t)\Big)_+\int_{\Omega}(\overline{U}_{+}) \left[ (\int_{\Omega}u)u-|\Omega|\underline{u}\overline{u}   \right] +\Big(a_{2,\inf}(t) \Big)_-\int_{\Omega}\overline{U}_+(u\int_{\Omega} \overline{U})  \nonumber
\end{eqnarray*}
for  a.e. $t\in (t_0,t_0+T_{\max})$.
Note that\begin{align*}
\Big(a_{2,\inf}(t) \Big)_-\int_{\Omega}\overline{U}_+(u\int_{\Omega} \overline{U})
\leq&\Big(a_{2,\inf}(t) \Big)_-\int_{\Omega}\overline{U}_+u(\int_{\Omega} \overline{U}_+)\le |\Omega| \Big(a_{2,\inf}(t) \Big)_-(\overline{u}+2\epsilon)\int_{\Omega}\overline{U}_+^2
\end{align*}
and
\vspace{-0.1in}\begin{align*}
-\Big(a_{2,\inf}(t) \Big)_+\int_{\Omega}\overline{U}_+\Big((\int_{\Omega}u)u-|\Omega|\underline{u}\overline{u}\Big)=&-\Big(a_{2,\inf}(t) \Big)_+\int_{\Omega}\overline{U}_+^2(\int_{\Omega}u)
-\Big(a_{2,\inf}(t)\Big)_+\int_{\Omega}\overline{U}_+\overline{u}\int_{\Omega}\underline{U}\\
 \leq & -\Big(a_{2,\inf}(t)\Big)_+\int_{\Omega}\overline{U}_+\overline{u}\int_{\Omega}\underline{U}\le   \Big(a_{2,\inf}(t)\Big)_+\overline{u}\int_{\Omega}\overline{U}_+\int_{\Omega}\underline{U}_-\\
\leq & |\Omega|\Big(a_{2,\inf}(t)\Big)_+ \frac{\overline{u}}{2}\Big(\int_{\Omega}\overline{U}_+^2+\int_{\Omega}\underline{U}_-^2\Big).
\end{align*}
Moreover by  using the second equation of  $\eqref{u-v-eq1},$  we get
\vspace{-0.1in}\[   \int_{\Omega}| \nabla(v-\underline{u})_-|^2 +\int_{\Omega}(v-\underline{u})^2_-  =  - \int_{\Omega}(\underline{U}) (v-\underline{u})_-    \leq  \int_{\Omega}(\underline{U})_- (v-\underline{u})_-  .    \]
Thus by Young's inequality, we have $\int_{\Omega}(v-\underline{u})^2_- \leq \int_{\Omega}(\underline{U}_-)^2$.
Therefore
\vspace{-0.1in}\[
-\chi \int_{\Omega}(\overline{U}_{+})u(v-\underline{u}) \leq \frac{\chi(\overline{u}+2\epsilon)}{2}\Big(\int_{\Omega}(\overline{U}_{+})^2+ \int_{\Omega}(\underline{U}_-)^2\Big).
\]
Combining all these inequalities, we get
\begin{eqnarray}
\label{eq-u-upper}
& & \frac{1}{2}\frac{d}{dt}\int_{\Omega}(\overline{U}_{+})^2       +\int_{\Omega}|\nabla (\overline{U}_{+}) |^2\nonumber \\
& & \leq \int_{\Omega}(\overline{U}_{+})^2 \left[   a_{0,\sup}(t)+\chi \frac{1}{2}u +2 |\Omega| \Big(a_{2,\inf}(t) \Big)_-(\overline{u}+\epsilon) +  |\Omega|\Big(a_{2,\inf}(t)\Big)_+ \frac{\overline{u}}{2}+ \frac{\chi(\overline{u}+2\epsilon)}{2} \right] \nonumber\\
& &   + \left[  \frac{\chi(\overline{u}+2\epsilon)}{2}+  |\Omega|\Big(a_{2,\inf}(t)\Big)_+ \frac{\overline{u}}{2} \right] \int_{\Omega}(\underline{U}_-)^2\nonumber\\
 && \leq \int_{\Omega}(\overline{U}_{+})^2 \left[   a_{0,\sup}(t)+\Big(2 |\Omega| \big(a_{2,\inf}(t) \big)_-+\chi\Big)(\overline{u}+2\epsilon) +  |\Omega|\Big(a_{2,\inf}(t)\Big)_+ \frac{\overline{u}}{2}\right]\nonumber\\
 && + \left[  \frac{\chi(\overline{u}+2\epsilon)}{2}+  |\Omega|\Big(a_{2,\inf}(t)\Big)_+ \frac{\overline{u}}{2} \right] \int_{\Omega}(\underline{U}_-)^2\quad {\rm for}\,\,a.e.\,\,  t\in(t_0,T_\epsilon].
\end{eqnarray}

Similarly,  we have that $\int_{\Omega}\underline U_-^2(x,t)dx$  is weakly differentiable in $t$ and moreover
\begin{equation}
\label{new-eqq2}
\frac{d}{dt}\int_{\Omega}\underline U_-^2(x,t)dx=2\int_{\Omega}\underline U_-(x,t)\underline U_t(x,t)dx\quad {\rm for}\quad a.e.\, t\in (t_0,t_0+T_{\max}),
\end{equation}
and
\begin{equation}
\label{new-eqq2-1}
\int_\Omega \underline U_-^2(x,t)dx=\int_\Omega \underline U_-^2(x,t_0)dx+\int_{t_0}^t \Big(\frac{d}{dt}\int_{\Omega}\underline U_-^2(x,\tau)dx\Big)d\tau\quad \forall\,\, t\in (t_0,t_0+T_{\max}).
\end{equation}
Also we have
\vspace{-0.1in} \begin{align*}
  \underline{U}_t-\Delta \underline{U}\ge &-\chi  \nabla\overline{U} \cdot \nabla v
   +\overline{U} \left[a_{0,\inf}(t) -\Big(a_{1,\sup}(t)-\chi\Big)(u+\underline{u})  -  \chi \overline{u}+|\Omega|\Big(a_{2,\sup}(t)\Big)_-\underline{ u}\right] \\
   &-\chi u(v-\overline{u})
                                                                    -\Big(a_{2,\sup}(t) \Big)_+\Big((\int_{\Omega}u)u-|\Omega|\underline{u}\overline{u}\Big)
+\Big(a_{2,\sup}(t) \Big)_-(\int_{\Omega} \underline{U})u . \end{align*}
 By multiplying the above inequality by $-\underline{U}_{-}$ and integrating with respect to $x$ over $\Omega$, we have
 \begin{align}
\label{eq-u-over}
 &\frac{1}{2}\frac{d}{dt}\int_{\Omega}(\underline{U}_{-}^2)       +\int_{\Omega}|\nabla (\underline{U}_{-}) |^2\nonumber\\
 & \leq \int_{\Omega}(\underline{U}_{-})^2 \left[   a_{0,\inf}(t)+(2 |\Omega| \Big(a_{2,\sup}(t) \Big)_-+\chi)(\overline{u}+2\epsilon) + |\Omega|\Big(a_{2,\sup}(t)\Big)_+ \frac{\underline{u}}{2}\right] \nonumber\\
 &\,\,\,\, + \left[  \frac{\chi(\overline{u}+2\epsilon)}{2}+  |\Omega|\Big(a_{2,\sup}(t)\Big)_+ \frac{\underline{u}}{2} \right] \int_{\Omega}{ (\bar{U}_+)^2}\quad {\rm for}\quad a.e.\,\,  t\in (t_0, T_\epsilon].
\end{align}

 By \eqref{new-eq2-1}, \eqref{eq-u-upper}, \eqref{new-eqq2-1}, and \eqref{eq-u-over},  we have
\begin{align*}
&\frac{1}{2}\int_\Omega \Big (\overline U_+^2(x,t)+\underline U_-^2(x,t)\Big)dx\\
&\le \frac{1}{2}\int_\Omega \Big (\overline U_+^2(x,t_0)+\underline U_-^2(x,t_0)\Big)dx\\
&\,\,\, +\int_{t_0}^t \int_{\Omega}\Big (\overline U_+^2(x,\tau)+\underline U_-^2(x,\tau)\Big) \left[   a_{0,\sup}(\tau)+\Big(2 |\Omega| \big(a_{2,\inf}(\tau) \big)_-+\chi\Big)(\overline{u}+2\epsilon) +  |\Omega|\Big(a_{2,\sup}(\tau)\Big)_+ \frac{\overline{u}}{2}\right]d\tau\\
 &\,\,\, + \int_{t_0}^t \left[  \frac{\chi(\overline{u}+2\epsilon)}{2}+  |\Omega|\Big(a_{2,\sup}(\tau)\Big)_+ \frac{\overline{u}}{2} \right] \int_{\Omega}\Big (\overline U_+^2(x,\tau)+\underline U_-^2(x,\tau)\Big)d\tau\quad \forall  t\in(t_0,T_\epsilon].
\end{align*}
This together with  $\overline{U}_{+}(\cdot,t_0)=\underline{U}_-(\cdot,t_0)=0$ and Gronwall's
 inequality implies  $\overline{U}_{+}(x,t)=\underline{U}_- (x,t)=0$ for  $(x,t)\in\Omega\times [t_0,T_{\epsilon}]$.
Therefore,
\vspace{-0.1in}$$\underline{u}(t) \leq u(x,t;t_0,u_0) \leq \overline{u} (t)\quad (x, t) \in \overline{\Omega} \times [t_0, T_{\epsilon}].
$$
This is a contradiction. Therefore,
$T_{\epsilon}=t_0+T_{\max}$. We then have $T_{\max}=\infty$ and \eqref{global-1-eq1} holds.

\medskip

 (2) We divide  the proof in three steps.  Note that the statements in these steps   have already been establish in the case of constant coefficients and $a_2=0,$  by  Tello and Winkler in \cite[Lemma 2.2, 2.3 and 2.4]{TW07}. For simplicity in notation, we denote $(u(\cdot,t;t_0,u_0),v(\cdot,t;t_0,u_0))$ by $(u(t),v(t))$.

\smallskip

\noindent{\bf Step 1.} In this step, we prove that  for any $\gamma \in \left(1,\frac{\chi}{\left(\chi-a_{1,\inf}\right)_+} \right),$   there is $C=C(\gamma,\|u_0\|_{L^{\gamma}}$, $\|u_0\|_{L^1},A_0,A_2,a_1,|\Omega|)$  such that
\vspace{-0.1in}\begin{equation}
\label{new-global-eq1}
 \int_{\Omega}u^{\gamma} (t)     \leq C   \quad \forall\,\, t \in [t_0, t_0+T_{\max}),
 \end{equation}
and
  \begin{equation}
  \label{new-global-eq2}
  \int_{t_0}^T\int_{\Omega}u^{\gamma+1}(t)        +   \int_{t_0}^T\int_{\Omega}  |\nabla u^{\frac{\gamma}{2}}(t)|^2 \leq \widetilde{C}(T+1)  \quad \forall T \in (t_0,t_0+T_{\max}).
  \end{equation}

By multiplying the  first equation of  $\eqref{u-v-eq1}$  by $u^{\gamma-1}(t)$ and integrating with respect to $x$ over $\Omega,$  we have for $t\in (t_0,t_0+T_{\max})$ that
      \begin{align*}
      \frac{1}{\gamma}\frac{d}{dt}\int_{\Omega}u^{\gamma}(t)+\frac{4(\gamma-1)}{\gamma^2}\int_{\Omega}  |\nabla u^{\frac{\gamma}{2}}(t)|^2=&(\gamma-1)\chi\int_{\Omega}u^{\gamma-1}(t)\nabla u (t)\cdot \nabla v(t) \\   &+\int_{\Omega}u^{\gamma}(t)\Big[a_0(t,\cdot)-a_1(t,\cdot)u(t)-a_2(t,\cdot)\int_{\Omega}u(t) \Big].                      \end{align*}
By multiplying the  second equation of  $\eqref{u-v-eq1}$  by $u^{\gamma}(\cdot)$ and integrating over $\Omega,$ we get
\[  (\gamma-1)\chi\int_{\Omega}u^{\gamma-1}(t)\nabla u (t)\cdot \nabla v(t) =-\frac{\chi(\gamma-1)}{\gamma}\int_{\Omega}v(t) u^{\gamma}(t)    +\frac{\chi(\gamma-1)}{\gamma}\int_{\Omega}u^{\gamma+1}(t)  .           \]
Thus we have for $t\in (t_0,t_0+T_{\max})$ that
\begin{eqnarray*}
&&\frac{1}{\gamma}\frac{d}{dt}\int_{\Omega}u^{\gamma}(t)+\frac{4(\gamma-1)}{\gamma^2}\int_{\Omega}  |\nabla u^{\frac{\gamma}{2}}(t)|^2\\
&& =-\frac{\chi(\gamma-1)}{\gamma}\int_{\Omega}v u^{\gamma}(t)    +\frac{\chi(\gamma-1)}{\gamma}\int_{\Omega}u^{\gamma+1}(t)  +\int_{\Omega}u^{\gamma}(t)\Big[a_0(t,\cdot)-a_1(t,\cdot)u(t)-a_2(t,\cdot)\int_{\Omega}u(t) \Big].
\end{eqnarray*}
By Lemma \ref{lem-004b}, we have $\Big(a_{2,\inf}(t)\Big)_-\int_{\Omega}u(t)\leq A_2 M_0$. Therefore
\vspace{-0.1in}\begin{eqnarray}
\label{new-added-eq1}
&&\frac{1}{\gamma}\frac{d}{dt}\int_{\Omega}u^{\gamma}(t)+\frac{4(\gamma-1)}{\gamma^2}\int_{\Omega}  |\nabla u^{\frac{\gamma}{2}}(t)|^2\nonumber\\
&&\leq \frac{\chi(\gamma-1)}{\gamma}\int_{\Omega}u^{\gamma+1}(t)  +\int_{\Omega}u^{\gamma}(t)\left[a_0(t,\cdot)-a_{1,\inf}u(t)+\Big(a_{2,\inf}(t)\Big)_-\int_{\Omega}u(t) \right]            \nonumber\\
&&\leq \frac{\chi(\gamma-1)}{\gamma}\int_{\Omega}u^{\gamma+1}(t)  +\int_{\Omega}u^{\gamma}(t)\left[A_0+A_2M_0-a_{1,\inf}u(t) \right]            \nonumber\\
&&\le  -\left[a_{1,\inf}-\frac{\chi(\gamma-1)}{\gamma}\right]\int_{\Omega}u^{\gamma+1} (t) +(A_0+A_2M_0)\int_{\Omega}u^{\gamma}(t).
\end{eqnarray}
Note that $\mu:=a_{1,\inf}-\frac{\chi(\gamma-1)}{\gamma}>0$.  By Young's inequality, we have
\vspace{-0.1in} \[{ (A_0+ A_2M_0)}\int_{\Omega}u^{\gamma} (t)\leq \frac{1}{2}\mu\int_{\Omega}u^{\gamma+1}(t)+C(\gamma,A_0,A_2,a_1,\|u_0\|_{L^1},|\Omega|).    \]
Thus
\begin{equation}\label{eq-u000}
\frac{1}{\gamma}\frac{d}{dt}\int_{\Omega}u^{\gamma}(t)+\frac{4(\gamma-1)}{\gamma^2}\int_{\Omega}  |\nabla u^{\frac{\gamma}{2}}(t)|^2 \leq -\frac{\mu}{2}\int_{\Omega}u^{\gamma+1}(t)+C(\gamma,A_0, A_2,a_1,\|u_0\|_{L^1},|\Omega|) .
\end{equation}
This together with  H\"older's inequality implies that
\vspace{-0.1in}\[\frac{d}{dt}\int_{\Omega}u^{\gamma}(t) \leq -\frac{\mu \gamma}{|\Omega|^{\frac{1}{\gamma}}}\left(\int_{\Omega}u^{\gamma}  (t) \right)^{\frac{\gamma+1}{\gamma}}+C(\gamma,A_0,A_2,a_1,\|u_0\|_{L^1},|\Omega|).  \]
It then follows that
\vspace{-0.1in}\[  \int_{\Omega}u^{\gamma}  (t)    \leq \max\Big\{ \int_{\Omega}u_0^{\gamma} ,\left(\frac{C(\gamma,A_0,A_2,a_1,\|u_0\|_{L^1},|\Omega|)}{\mu}\right)^{\frac{\gamma}{\gamma+1}} \Big\}  \quad \forall t \in [t_0,t_0+ T_{\max}). \]
Now by integrating $~\eqref{eq-u000}$ on  $(t_0,T),$ we get
\[   \int_{t_0}^T\int_{\Omega}u^{\gamma+1}  (t)      +   \int_{t_0}^T\int_{\Omega}  |\nabla u^{\frac{\gamma}{2}}(t)|^2 \leq \widetilde{C}(T+1) .                   \]
\eqref{new-global-eq1} and \eqref{new-global-eq2} then follow.

\medskip

\noindent{\bf Step 2.} In this step, we prove that  for any $ \gamma >1$, there is $C=C(\gamma,\|u_0\|_{L^{\gamma}},\|u_0\|_{L^1},A_0,A_2,a_1,|\Omega|)$ such that
 \begin{equation}
 \label{new-global-eq3}
 \int_{\Omega}u^{\gamma}(t) \leq C \quad \forall t \in [t_0, t_0+T_{\max}).
 \end{equation}

Since $a_{1,\inf}>\frac{\chi(n-2)}{n},$ we get $\frac{n}{2}<\frac{\chi}{(\chi-a_{1,\inf})_+}.$ Choose $\gamma_0 \in (\frac{n}{2},\frac{\chi}{(\chi-a_{1,\inf})_+}),$ then by \eqref{new-global-eq1},  we have
 \[ \int_{\Omega}u^{\gamma_0}(t) \leq C=C(\gamma_0,\|u_0\|_{L^{\gamma_0}},\|u_0\|_{L^1},A_0,A_2,a_1,|\Omega|).\]
Let $\gamma>1.$ If $1<\gamma \leq \gamma_0,$ the result follows by the continuous inclusion $L^{\gamma_0}(\Omega) \subset L^{\gamma}(\Omega). $ Suppose $\gamma>\gamma_0.$
Let { $\tilde \mu=2|(a_{1,\inf}-\frac{\chi(\gamma-1)}{\gamma})|+1>0$}. By Young's inequality we get
 \[ { (A_0+A_2M_0)} \int_{\Omega}u^{\gamma}(t) \leq { \frac{\tilde \mu}{2}}\int_{\Omega}u^{\gamma+1}(t)+C(\gamma,A_0,A_2,a_1,\|u_0\|_{L^1},|\Omega|) .   \]
This together with \eqref{new-added-eq1} implies that
\[\frac{1}{\gamma}\frac{d}{dt}\int_{\Omega}u^{\gamma}(t)+\frac{4(\gamma-1)}{\gamma^2}\int_{\Omega}  |\nabla u^{\frac{\gamma}{2}}(t)|^2 \leq { \tilde \mu}\int_{\Omega}u^{\gamma+1}(t)+C(\gamma,A_0, A_2,a_1,\|u_0\|_{L^1},|\Omega|).  \]
Note that \[ \int_{\Omega}u^{\gamma_0}(t)=\|u^{\frac{\gamma}{2}}(t)\|_{L^{\frac{2\gamma_0}{\gamma}}}^{\frac{2\gamma_0}{\gamma}} \leq C\quad {\rm and}\quad \int_{\Omega}u^{\gamma+1}(t)=\|u^{\frac{\gamma}{2}}(t)\|_{L^{\frac{2(\gamma+1)}{\gamma}}}^{\frac{2(\gamma+1)}{\gamma}} .\]
By Gagliardo-Nirenberg inequality, there exists $C_0$ depending on the domain $\Omega$ and $\gamma$ such that
\begin{eqnarray*}
 \int_{\Omega}u^{\gamma+1}(t)=  \|u^{\frac{\gamma}{2}}(t)\|_{L^{\frac{2(\gamma+1)}{\gamma}}}^{\frac{2(\gamma+1)}{\gamma}} &\leq &C_0 \|\nabla u^{\frac{\gamma}{2}}(t)\|^{\frac{2(\gamma+1)a}{\gamma}}_{L^2}  \|u^{\frac{\gamma}{2}}(t)\|_{L^{\frac{2\gamma_0}{\gamma}}}^{\frac{2(\gamma+1)(1-a)}{\gamma}}+C_0\|u^{\frac{\gamma}{2}}(t)\|_{L^{\frac{2\gamma_0}{\gamma}}}^{\frac{2\gamma_0}{\gamma}}  \nonumber\\
&\leq& C(\gamma,\|u_0\|_{L^{\gamma}},\|u_0\|_{L^1},A_0,A_2,a_1,|\Omega|) \Big( \|\nabla u^{\frac{\gamma}{2}}(t)\|^{\frac{2(\gamma+1)a}{\gamma}}_{L^2}+1 \Big),
\end{eqnarray*}
where $a=\frac{\frac{n\gamma}{2\gamma_0}-\frac{n\gamma}{2(\gamma+1)}}{1+\frac{n}{2}(\frac{\gamma}{\gamma_0}-1)}.$
Since $\frac{n}{2}<\gamma_0<\gamma,$ we have $0<a<1$ and $2\frac{(\gamma+1)}{\gamma}a-2=-\frac{2-\frac{n}{\gamma_0}}{1+\frac{n}{2}(\frac{\gamma}{\gamma_0}-1)}<0.$
By applying Young's Inequality, we get for any $\epsilon>0$
 \begin{align*}
 &C(\gamma,\|u_0\|_{L^{\gamma}},\|u_0\|_{L^1},A_0,A_2,a_1,|\Omega|) \|\nabla u^{\frac{\gamma}{2}}(t)\|^{\frac{2(\gamma+1)a}{\gamma}}_{L^2} \\
 &\leq \epsilon\|\nabla u^{\frac{\gamma}{2}}(t)\|^2_{L^2}+C(\epsilon,\gamma,\|u_0\|_{L^{\gamma}},\|u_0\|_{L^1},A_0,A_2,a_1,|\Omega|).
 \end{align*}
Therefore
\[\int_{\Omega}u^{\gamma+1} (t)\leq  \epsilon\|\nabla u^{\frac{\gamma}{2}}(t)\|^2_{L^2}+C(\epsilon,\gamma,\|u_0\|_{L^{\gamma}},\|u_0\|_{L^1},A_0,A_2,a_1,|\Omega|) \]
and then
\[-\frac{4(\gamma-1)}{\gamma^2}\int_{\Omega}  |\nabla u^{\frac{\gamma}{2}}(t)|^2 \leq -\frac{4(\gamma-1)}{\epsilon \gamma^2} \int_{\Omega}u^{\gamma+1}(t) +\frac{4(\gamma-1)}{\epsilon \gamma^2}C(\epsilon,\gamma,\|u_0\|_{L^{\gamma}},\|u_0\|_{L^1},A_0,A_2,a_1,|\Omega|).  \]
It then follows that
\[\frac{1}{\gamma}\frac{d}{dt}\int_{\Omega}u^{\gamma}(t) \leq -\Big({ \frac{4(\gamma-1)}{\epsilon \gamma^2}-\tilde\mu}\Big) \int_{\Omega}u^{\gamma+1}(t)+C(\epsilon,\gamma,\|u_0\|_{L^{\gamma}},\|u_0\|_{L^1},A_0,A_2,a_1,|\Omega|). \]
By choosing $\epsilon=\frac{4(\gamma-1)}{\gamma^2(1+\tilde \mu)},$ we get
\vspace{-0.1in}\[\frac{1}{\gamma}\frac{d}{dt}\int_{\Omega}u^{\gamma} (t)\leq - \int_{\Omega}u^{\gamma+1}(t)+C(\gamma,\|u_0\|_{L^{\gamma}}) \leq -\frac{1}{|\Omega|^{\frac{1}{\gamma}}}\Big(\int_{\Omega}u^\gamma (t)\Big)^{\frac{\gamma+1}{\gamma}}+ C(\gamma,\|u_0\|_{L^{\gamma}}).\]
This implies that
 \vspace{-0.1in} \[ \int_{\Omega}u^{\gamma}(t) \leq C(\gamma,\|u_0\|_{L^{\gamma}},\|u_0\|_{L^1},A_0,A_2,a_1,|\Omega|) \quad \forall t \in [t_0, t_0+T_{\max}). \]
\eqref{new-global-eq3} then  follows.

\medskip

\noindent{\bf Step 3.} In this sept, we prove that  there is $ C=C(\|u_0\|_{L^{\infty}})$  such that
\begin{equation}
\label{new-global-eq4}
\|u(t)\|_{C^0(\bar\Omega)}+ \|v(t)\|_{C^0(\bar\Omega)} \leq C \quad \forall t \in [t_0, t_0+T_{\max}).
\end{equation}

By the variation of constant formula, we have
\vspace{-0.1in}\begin{align*}
u(t)&=e^{-A(t-t_0)}u_0-\chi\int_{t_0}^t e^{-(t-s)A}\nabla(u (s)\cdot\nabla v(s))ds \\
&\,\,+\int_{t_0}^t e^{-A(t-s)}u(s)\Big[\underbrace{1+a_0(s,\cdot)-a_1(s,\cdot)u(s)-(a_2(s,\cdot))_+\int_{\Omega}u(s) +(a_2(s,\cdot))_-\int_{\Omega}u(s)}_{I_0(\cdot,s)}\Big] ds,
\end{align*}
Note that $u(s)I_0(\cdot,s)\leq u(s)[\underbrace{1+A_2M_0+a_0(\cdot,s)-a_1(s,\cdot)u(s)}_{I_1(\cdot,s)}]$ and by parabolic comparison principle, we get $\int_{t_0}^t e^{-A(t-s)}u(s)I_0(\cdot,s)ds\leq \int_{t_0}^t e^{-A(t-s)}u(s)I_1(\cdot,s)ds.$ Therefore
$$u(t)\leq  u_1(t)+u_2(t)+u_3(t),$$
where
\vspace{-0.05in}$$u_1(t)=e^{-A(t-t_0)}u_0,\quad u_2(t)=-\chi\int_{t_0}^t e^{-(t-s)A}\nabla(u (s)\cdot\nabla v(s))ds$$
 and
 \vspace{-0.05in} $$u_3(t,x)=\int_{t_0}^t e^{-A(t-s)}u(s)\left[1+A_2M_0+a_0(\cdot,s)-a_1(s,\cdot)u(s) \right] ds.$$
 Note that there are $c_0,c_1>0$ such that $(1+A_2M_0+a_0(t,x))r-a_1(t,x)r^2\le c_0-c_1 r^2$ for all $t\in\RR$, $x\in\Omega$, and $r\ge 0$.
We then  have that
\vspace{-0.05in} \[  \| u_1(t)\|_{L^\infty(\Omega)} \leq \|u_0\|_{L^{\infty}(\Omega)} \quad \forall \,\,  t\in [t_0,t_0+T_{\max})\]
and
\vspace{-0,05in}\[u_3(t) \leq  C \int_{t_0}^t  e^{-A(t-s)} ds \leq C \int_{t_0}^t  e^{-(t-s) }\leq C\quad \forall\,\, t\in [t_0,t_0+T_{\max}). \]
Choose $p>n$ and $\alpha \in (\frac{n}{2p},\frac{1}{2}).$ Then $X^{\alpha} \subset L^{\infty}(\Omega)$ and the inclusion is continuous (see \cite{DH77} exercise 10, page 40.) Choose $\epsilon \in (0, \frac{1}{2}-\alpha),$ then we have
\vspace{-0.1in}\begin{eqnarray*}
\|u_2(t)\|_{L^{\infty}(\Omega)}&\leq & C\|A^{\alpha}u_2(t)\|_{L^p(\Omega)} \le C\chi \int_{t_0}^t \| A^{\alpha}e^{-(t-s)A}\nabla(u(s) \cdot\nabla v(s))\|_{L^p(\Omega)}ds \nonumber \\
  &\leq& C\int_{t_0}^t(t-s)^{-\alpha-\frac{1}{2}-\epsilon}e^{-\mu(t-s)} \| u(s) \cdot\nabla v(s)\||_{L^p(\Omega)}ds \nonumber\\
 &\leq&  C\int_{t_0}^t(t-s)^{-\alpha-\frac{1}{2}-\epsilon}e^{-\mu(t-s)} \| u(s)\| _{L^{p_1}(\Omega)}  \|\nabla v(s)\|_{L^{p_2}(\Omega)}ds \nonumber
\end{eqnarray*}
for $t\in[t_0,t_0+T_{\max})$,
where $p_1>p $ and $\frac{1}{p}=\frac{1}{p_1}+\frac{1}{p_2}.$
Note that $ \|\nabla v(s)\|_{L^{p_2}(\Omega)} \leq  C \| u(s)\|_{L^{p_2}(\Omega)}$. By \eqref{new-global-eq3}, we get
\vspace{-0.1in}\[\|u_2(t)\|_{L^{\infty}(\Omega)} \leq C(\|u_0\|_{L^{\infty}(\Omega)})\int_{t_0}^{\infty} (t-s)^{-\alpha-\frac{1}{2}-\epsilon}e^{-\mu(t-s) }ds <\infty .\]
Therefore
\vspace{-0.1in}\[\|v(t)\|_\infty\le \|u(t)\|_{L^{\infty}(\Omega)} \leq C(\|u_0\|_{L^{\infty}(\Omega)})\quad \forall\,\, t\in [t_0,t_0+T_{\max}). \]
\eqref{new-global-eq4} then   follows. Theorem \ref{thm-002}(2) is thus proved.
\end{proof}

\begin{remark}
\label{remark-h3-1}
\begin{itemize}
\item[(1)]  For the proof of Theorem \ref{thm-002}(1)  in the  case $\chi\le 0,$ we need to write
$$-\chi uv+(\chi)_+\overline{u}\underline{u}-(\chi)_-u^2=-(\chi)_+u(v-\underline{u})-(\chi)_+\underline{u}\overline{U}-(\chi)_-u\overline{U}+(\chi)_-u(v-\overline{u})$$ and the proof follows as in the case $\chi>0.$

\item[(2)]  For the proof of Theorem \ref{thm-002}(2)  in the   case $\chi\leq 0,$ the only major change is in {\bf Step 1.} Indeed in {\bf Step 1}, we need to consider $\gamma \in \left(1,\infty \right)$ and use \cite[Lemma 3.2]{ZhLi}. The proof then follows as in the case $\chi>0.$

\item[(3)]  Assume {\bf (H2)$^{'}$}. It follows from the proof of Theorem \ref{thm-002}.(2) that for any

 $M\ge \frac{a_{0,\sup}}{\inf_{t \in \mathbb{R}} \big \{ a_{1,\inf}(t)-|\Omega|\big(a_{2,\inf}(t)\big)_-\big \}}$, there is a positive constant
$C=C(M)$ depending only on $M$  such that for any $u_0\in C^0(\bar\Omega)$ with $u_0\ge 0$ and  $\|u_0\|_{C^0(\bar{\Omega})}\leq M$,
$0\leq u(\cdot,t;t_0;u_0) \leq C.$
\end{itemize}
\end{remark}

\section{Existence of entire positive solutions}

In this section, we explore the existence of entire positive solutions of \eqref{u-v-eq1} in the general case; the existence
of time almost periodic, time periodic, and time independent positive solutions of \eqref{u-v-eq1} in the case that the coefficients of \eqref{u-v-eq1}
are time almost periodic, time periodic, and time independent, respectively;  and prove Theorem
\ref{thm-003}.

We first prove three lemmas.
Throughout this section, we assume that $\chi>0$ (the case $\chi \leq 0$ follows from similar arguments), {\bf  (H2)} holds
and we let
\begin{equation}
\label{entire-eq2}
 M=\frac{a_{0,\sup}}{\inf_{t\in\RR} \Big\{a_{1,\inf}(t)-|\Omega|\Big(a_{2,\inf}(t)\Big)_-\Big\}-\chi}.
 \end{equation}

Let $(u(x,t;t_0,u_0),v(x,t;t_0,u_0))$ be the solution of \eqref{u-v-eq1} with $u(x,t_0;t_0,u_0)=u_0(x)$ ($u_0\in C^0(\bar\Omega)$).
By Corollary \ref{cor1},   for any $t_2>t_1>t_0$,
\vspace{-0.05in}$$
u(x,t_2;t_0,u_0)=u(x,t_2;t_1,u(\cdot,t_1;t_0,u_0)).
\vspace{-0.05in}$$
By Theorem \ref{thm-002}, the global existence of \eqref{u-v-eq1} holds,  and  for any $0\le u_0(\cdot)\le M$,
\vspace{-0.05in}\begin{equation}
\label{thm3-eq0}
0\le u(\cdot,t;t_0,u_0)\le M\quad {\rm for}\quad t\ge t_0.
\vspace{-0.05in}\end{equation}

\begin{lemma}
\label{lem-008} Fix a $T>0$.
For any $\epsilon>0$, there is $\delta=\delta(T)>0$ such that for any give $u_0(\cdot)\ge 0$ with
$\sup u_0<\delta$ and any $t_0\in\RR$, $u(x,t+t_0;t_0,u_0)<\epsilon$ for $0\le t\le T$.
\end{lemma}

\begin{proof}
It follows from the continuity with respect to initial conditions.
\end{proof}

Fix a $T>0$. Fix $\epsilon_0$ such that $\epsilon_0<\frac{a_{0,\inf}}{\chi+|\Omega|\cdot |a_{2,\sup}|}$. Let $\delta_0=\delta$ be as in Lemma \ref{lem-008} with
$\epsilon=\epsilon_0$. By Lemma \ref{lem-008}, for given $0\le u_0(x)<\delta_0$, $u(x,t+t_0;t_0,u_0)<\epsilon_0$
for $0\le t\le T$. This implies that $v(x,t+t_0;t_0,u_0)=V(u(\cdot,t+t_0;t_0,u_0)):=A^{-1}u(\cdot,t+t_0;t_0,u_0)<\epsilon_0$ for $0\le t\le T$.

\begin{lemma}
\label{lem-009}
For any $t_0\in\RR$ and  any $0<u_0(x)<\min\{\delta_0,\frac{a_{0,\inf}-\epsilon_0(\chi+|\Omega|\cdot |a_{2,\sup}|)}{a_{1,\sup}}\}$ for $x\in \Omega$, $u(x,t+t_0;t_0,u_0)>\inf u_0$ for $0<t\le T$ and $x\in \Omega$.
\end{lemma}

\begin{proof}
By Lemma \ref{lem-008}, $V(u(\cdot,t+t_0;t_0,u_0))<\epsilon_0$ for $0\le t\le T$. Hence
\begin{align*}
u_t&=\Delta u -\chi \nabla u\cdot \nabla V(u)-\chi u (V(u)-u)+u(a_0(t,x)-a_1(t,x)u-a_2(t,x)\int_{\Omega}u)\\
&\ge \Delta u -\chi \nabla u\cdot \nabla V(u)+u(a_0(t,x)-\epsilon_0\chi -a_1(t,x)u-a_2(t,x)\int_{\Omega}u)\\
&\ge \Delta u-\chi \nabla u\cdot\nabla V(u)+u(a_{0,\inf}-\epsilon_0(\chi+|\Omega|\cdot |a_{2,\sup}|)-a_{1,\sup} u).
\end{align*}
Then by comparison principle, we have
$$
u(x,t+t_0;t_0,u_0)\ge u(t;\inf u_0)\quad 0\le t\le T
$$
where $u(t;\inf u_0)$ is the solution of the ODE
\begin{equation}
\label{entire-eq3}
\dot u=u(a_{0,\inf}-\epsilon_0(\chi+|\Omega|\cdot |a_{2,\sup}|)-a_{1,\sup} u).
\end{equation}
with $u(0;\inf u_0)=\inf u_0$.
Note that $u(t;\inf u_0)$ increases as $t$ increases.
The lemma then follows.
\end{proof}

\begin{lemma}
\label{lem-0010}
There is $\delta^*=\delta^*(T)>0$ such that for any $0<\delta\le \delta^*$, $t_0\in\RR$, and   $u_0(\cdot)$ with $\delta\le \inf u_0\le \sup u_0\le M$,
$ u(x,t_0+T;t_0,u_0)\ge \delta$ for $x\in \Omega$.
\end{lemma}

\begin{proof}
We prove the lemma by contradiction. Assume that the lemma does not hold. Then there are $\delta_n\to 0$, $t_n\in\RR$, and $u_n(\cdot)$ with $\delta_n\le \inf u_n\le M$
such that $\inf u(\cdot,t_n+T;t_n,u_n)<\delta_n$. Without loss of generality, we assume that $\delta_n<\min\{\delta_0,\frac{a_{0,\inf}-\epsilon_0(\chi+|\Omega|\cdot |a_{2,\sup}|)}{a_{1,\sup}}\}$. By Lemma \ref{lem-009}, we must have $\sup u_n\ge \min\{\delta_0,\frac{a_{0,\inf}-\epsilon_0(\chi+|\Omega|\cdot |a_{2,\sup}|)}{a_{1,\sup}}\}$.
Let
$$
\Omega_n=\Big\{x\in\Omega\,|\, u_n(x)\ge \frac{1}{2}\min\{\delta_0,\frac{a_{0,\inf}-\epsilon_0(\chi+|\Omega|\cdot |a_{2,\sup}|)}{a_{1,\sup}}\}\Big\}.
$$
Without loss of generality, we may assume that $m_0=\lim_{n\to\infty}|\Omega_n|$ exists, where $|\Omega_n|$ is the Lebesgue measure of
$\Omega_n$.  Assume that $m_0=0$. Then there is $\tilde u_n\in C^0(\bar\Omega)$ such that
$$
\delta_n\le \tilde u_n(x)\le  \frac{1}{2} \min\{\delta_0,\frac{a_{0,\inf}-\epsilon_0(\chi+|\Omega|\cdot |a_{2,\sup}|)}{a_{1,\sup}}\}
$$
and
$$
\lim_{n\to\infty} \|u_n-\tilde u_n\|_{L^p(\Omega)}=0\quad \forall\,\, 1\le p<\infty.
$$
This implies that
$$
\lim_{n\to\infty} \|u(\cdot,t;t_n,u_n)-u(\cdot,t;t_n,\tilde u_n)\|_{L^p(\Omega)}=0
$$
uniformly in $t\in[t_n,t_n+T]$ for all $1\le p<\infty.$
 Indeed, let $G(\cdot)$ be as in the proof of Theorem 1.1(1). Then  $G(u(\cdot,t;t_n,u_n))(t)=u(\cdot,t;t_n,u_n),$ $G(u(\cdot,t;t_n,\tilde u_n))(t)=u(\cdot,t;t_n,\tilde u_n)$. Let 
$$
\hat G(u_n)(t)=G(u(\cdot,t;t_n,u_n))(t),\quad \hat G (\tilde u_n)(t)=G(u(\cdot,t;t_n,\tilde u_n))(t),
$$
$$w_n(\cdot,t)=G(u(\cdot,t;t_n,u_n))(t)- G(u(\cdot,t;t_n,\tilde u_n))(t)$$
 and 
 $$W_n(\cdot,t)=V(G(u(\cdot,t;t_n,u_n))(t))-V( G(u(\cdot,t;t_n,\tilde u_n))(t)) .
 $$
  Then
\vspace{-0.1in}\begin{align}
\label{prooflemm5.3-eq1}
&w_n(\cdot,t)= \nonumber\\
&\,  e^{-A(t-t_n)}\big(u_n-\tilde u_n\big)-\chi\int_{t_n}^t e^{-A(t-s)}\nabla \big[w_n(\cdot,s)\cdot \nabla V(\hat G(u_n)(s)+\hat G(\tilde u_n)(s)\cdot \nabla W_n(\cdot,s)  \big]ds\nonumber\\
&+\int_{t_n}^t e^{-A(t-s)}w_n(\cdot,s)\Big(1+ a_0(s,\cdot)-a_1(s,\cdot)(\hat G(u_n)+\hat G(\tilde u_n))(s)-a_2(s,\cdot)\int_{\Omega}\hat G(u_n)(s)\Big)ds\nonumber\\
&-\int_{t_n}^t e^{-A(t-s)}a_2(s,\cdot)\big(\int_{\Omega}w_n(\cdot,s)\big)\hat G(\tilde u_n)(s)ds.
\end{align}
Now, fix  $1<p <\infty.$  By  regularity and a  priori estimates for
elliptic equations,  \cite[Theorem 1.4.3]{DH77},   Lemma \ref{lem-003}, and \eqref{prooflemm5.3-eq1},   for any $\epsilon \in (0, \frac{1}{2}),$ we have
\vspace{-0.1in}\begin{align}
\label{prooflemm5.3-eq2}
&\|w_n(\cdot,t)\|_{L^p(\Omega)}\nonumber\\
&\leq \|u_n-\tilde u_n\|_{L^p(\Omega)}+ C\chi\max_{t_n\leq s\leq t_n+T}\|\nabla V(\hat G(u_n)(s))\|_{C^0(\bar{\Omega})}\int_{t_n}^t(t-s)^{-\epsilon-\frac{1}{2}}\|w_n(\cdot,s)\|_{L^p(\Omega)}ds\nonumber\\
&\,\, +C\chi\max_{t_n\leq s\leq t_n+T}\|\hat G(\tilde u_n)(s)\|_{C^0(\bar{\Omega})}\int_{t_n}^t(t-s)^{-\epsilon-\frac{1}{2}}\|w_n(\cdot,s)\|_{L^p(\Omega)}ds\nonumber\\
&\,\, +C\int_{t_n}^t \{1+ A_0+A_1[\max_{t_n\leq s\leq t_n+T}(\|\hat G(u_n)(s)\|_{C^0(\bar{\Omega})}+\|\hat G(\tilde u_n)(s)\|_{C^0(\bar{\Omega})})]\}\|w_n(\cdot,s)\|_{L^p(\Omega)}ds\nonumber \\
&\,\, +C\int_{t_n}^t A_2|\Omega|\max_{t_n\leq s\leq t_n+T}\|\hat G(u_n)(s)\|_{C^0(\bar{\Omega})}\|w_n(\cdot,s)\|_{L^p(\Omega)}ds\nonumber\\
&\,\,  +C\int_{t_n}^tA_2\|\hat G(\tilde u_n)(s)\|_{C^0(\bar{\Omega})}\|w_n(\cdot,s)\|_{L^p(\Omega)}ds.
\end{align}
Therefore there exists a positive constant $C_0$ independent of $t$ and $n$ such that
\begin{equation}
\label{prooflemm5.3-eq3}
\|w_n(\cdot,t+t_n)\|_{L^p(\Omega)}\leq \|u_n-\tilde u_n\|_{L^p(\Omega)}+ C_0\int_{0}^{t}(t-s)^{-\epsilon-\frac{1}{2}}\|w_n(\cdot,s+t_n)\|_{L^p(\Omega)}ds \qquad \forall t \in [0, T].
\end{equation}
By \eqref{prooflemm5.3-eq3} and the generalized Gronwall's inequality (see \cite[page 6]{DH77}), we get
$$
\lim_{n\to\infty} \|u(\cdot,t;t_n,u_n)-u(\cdot,t;t_n,\tilde u_n)\|_{L^p(\Omega)}=0,
$$
uniformly in $t\in[t_n,t_n+T]$ for all $1\leq p<\infty.$
Therefore,
$$
\lim_{n\to\infty}\|V(u(\cdot,t;t_n,u_n))-V(u(\cdot,t;t_n;\tilde u_n))\|_{C^1(\bar \Omega)}=0
$$
uniformly in $t\in[t_n,t_n+T]$. Note that
$$
V(u(\cdot,t;t_n,\tilde u_n))(x)\le  \epsilon_0
$$
for all $t\in [t_n,t_n+T]$ and $x\in\Omega$. It then follows that
$$
V(u(\cdot,t;t_n, u_n))(x)\le  2 \epsilon_0
$$
for all $t\in [t_n,t_n+T]$, $x\in\Omega$, and $n\gg 1$. Then by the arguments of Lemma \ref{lem-009}, $\inf u(\cdot,t_n+T;t_n,u_n)\ge\delta_n$, which
is a contradiction.
Therefore, $m_0\not =0$.

By $m_0\not =0$ and comparison principle for parabolic equations,  without loss of generality, we may assume that
 $$
 \liminf_{n\to\infty}\|e^{-At}u_n\|_{C^0(\bar\Omega)}>0\quad \forall\,\, t\in [0,T].
 $$
 This together with the arguments in the proof of Theorem 1.1(2) implies that there is $T_0>0$ and $\delta_\infty>0$ such that
 $$\sup u(\cdot,t_n+T_0;t_n,u_n)\ge \delta_\infty$$
 for all $n\gg  1$.
 By  a priori estimates for parabolic equations, without loss of
generality, we may assume that
\vspace{-0.05in}$$
u(\cdot,t_n+T_0;t_n,u_n)\to u_0^*,\quad u(\cdot,t_n+T;t_n,u_n)\to u^*
\vspace{-0.05in}$$
as $n\to\infty$. By (H1), without loss of generality, we may also assume that
\vspace{-0.05in}$$
a_i(t+t_n,\cdot)\to a_i^*(t,x)
\vspace{-0.05in}$$
as $n\to\infty$ locally uniformly in $(t,x)\in\RR\times\bar\Omega$. Then by Corollary \ref{cor1},
\vspace{-0.05in}$$
u^*(x)=u^*(x,T;T_0,u_0^*)\quad {\rm and}\quad  \inf u^*=0,
\vspace{-0.05in}$$
where $(u^*(x,t;T_0,u_0^*), v^*(x,t;T_0,u_0^*))$ with $v^*(\cdot,t;T_0,u_0^*)=A^{-1}u^*(\cdot,t;T_0,u_0^*)$ is the solution of \eqref{u-v-eq1} with
$a_i(t,x)$ being replaced by $a^*(t,x)$.
By comparison principle, we must have $u_0^*\equiv 0$. But
\vspace{-0.05in}$$
\sup u_0^*\ge   \delta_\infty.
\vspace{-0.05in}$$
This is a contradiction.
\end{proof}

\begin{proof} [Proof of Theorem \ref{thm-003}]
We first prove the existence of entire positive solutions of \eqref{u-v-eq1} in the general case.

Let $\delta^*>0$ be given by Lemma \ref{lem-0010} with $T=1$. Choose $u_0 \in C^0(\bar{\Omega})$ such that $\delta^* \le u_0(x) \le M.$ By Lemma \ref{lem-0010} and \eqref{thm3-eq0},
 \vspace{-0.05in}\begin{equation}
 \label{thm3-eq1}
 \delta^*\le u(x, t_0+n;t_0,u_0) \le M \quad \forall\,\, x\in\bar\Omega,\,\,  t_0\in\RR,\, n\in\NN.
 \vspace{-0.05in}\end{equation}

Set $t_n=-n$ and define $u_n(x)=u(x,0;t_n,u_0).$ Choose $\tilde{t}$ such that $-2<\tilde{t}<-1.$ Then there is $\tilde M>0$ such that for each $n\ge 3,$ we have
\vspace{-0.05in}
$$ \|u_n\|_{\alpha}=\|u(\cdot,0;t_n,u_0)\|_{\alpha} = \|u(\cdot,0;\tilde{t},u(\cdot,\tilde{t};t_n,u_0))\|_{\alpha}\leq \tilde{M}.
\vspace{-0.05in}$$
 Therefore by Arzela-Ascoli Theorem, there exist $n_k$, $u^*_0 \in C^0(\bar{\Omega})$ such that $u_{n_k}$  converges to $u^*_0$ in $C^0(\bar{\Omega})$ as $n_k \to \infty.$ Then by Corollary \ref{cor1}, we have
 \vspace{-0.05in}
  $$u(\cdot,t;t_{n_k},u_{0})=u(\cdot,t;0,u(\cdot,0;t_{n_k},u_0))=u(\cdot,t;0,u_{n_k})  \to u(\cdot,t;0,u^*_0)
  \vspace{-0.05in}$$
   in $C^0(\bar{\Omega})$ as $n\to\infty$ for $t\ge 0$. Moreover, by \eqref{thm3-eq0} and Lemma \ref{lem-0010},
  \vspace{-0.05in} \begin{equation}
   \label{thm3-eq2}
   \delta^*\le u(x,n;0,u_0^*)\le M\quad \forall\,\, x\in\bar\Omega,\,\,  n\in\NN.
\vspace{-0.05in}   \end{equation}

We need to prove that  $u(\cdot,t;0,u^*_0)$ has  backward extension. To see   that, fix $m\in \NN$.  Then  $u(\cdot,t;t_n,u_0)$ is defined for $t>-m$
and $n>m$. Observe that
\vspace{-0.05in}$$
u_n(\cdot)=u(\cdot,0;t_n,u_0)=u(\cdot,0;-m,u(\cdot,-m;t_n,u_0)).
\vspace{-0.05in}$$
Without loss of generality, we may assume that
$u(\cdot,-m;t_{n_k},u_0)\to  u_m^*(\cdot)$
in $C^0(\bar\Omega)$.
Then
\vspace{-0.05in}$$
u(\cdot,t;t_{n_k},u_0)=u(\cdot,t;-m,u(\cdot,-m;t_{n_k},u_0)) \to u(\cdot,t;-m, u_m^*)
\vspace{-0.05in}$$
for $t>-m$ and $u(\cdot,t;0,u_0^*)=u(\cdot,t;-m,u_m^*)$ for $t\ge 0$. This implies that $u^*(x,t;0,u_0^*)$ has a backward extension up to $t=-m$. Let $m\to \infty$, we have that
$u^*(x,t)$ has a backward extension on $(-\infty,0)$.

 Let $u^*(x,t)=u^*(x,t;0,u_0^*)$ and $v^*(x,t)=A^{-1} u^*(\cdot,t),$ Then $(v^*(x,t),u^*(x,t))$ is an entire nonnegative solution of  \eqref{u-v-eq1}.  Moreover,
\vspace{-0.05in}\begin{equation}
\label{thm3-eq3}
\delta^*\le u^*(x,n)\le M\quad \forall \,\, x\in\bar\Omega,\,\, n\in\ZZ.
\vspace{-0.05in}\end{equation}
This implies that
\vspace{-0.05in}$$
0<\inf_{x\in\bar\Omega,t\in\RR}u^*(x,t)\le M,\quad 0<\inf_{x\in\bar\Omega,t\in\RR}v^*(x,t)\le M.
\vspace{-0.05in}$$
Therefore,  $(v^*(x,t),u^*(x,t))$ is an entire positive bounded solution of  \eqref{u-v-eq1}.

Next, we prove (1), (2), and (3).

  (1) Assume that $a_i(t+T,x)=a_i(t,x)$ for $i=0,1,2$. Let $\delta^*=\delta^*(T)>0$ be given by Lemma \ref{lem-0010} and set
  \begin{equation}
  \label{entire-eq4}
  E(T)=\{u_0 \in C^0(\bar{\Omega}) : \delta^* \leq u_0 \leq M\}.
    \end{equation}
    Note that $E(T)$ is nonempty, closed, convex and bounded subset of $C^0(\bar{\Omega}).$ Define  the map $\mathcal{T}(T):E(T) \to C^0(\bar{\Omega})$ by $\mathcal{T}(T)u_0=u(\cdot,T;0,u_0)$.
Note that $\mathcal{T}(T)$ is well defined and continuous by continuity with respect to initial conditions.

Let  $u_0 \in E(T)$. Then by Theorem \ref{thm-002}, we have   $0<u(\cdot,T;0,u_0)\leq M$ and by Lemma \ref{lem-0010}, we have $u(\cdot,T;0,u_0)\ge \delta^*.$ Thus $u(\cdot,T;0,u_0) \in E(T)$ and  $\mathcal{T}(T)E(T) \subset E(T)$.

Let $\frac{n}{2p}<\alpha < \frac{1}{2},$ and $\epsilon \in (0,\frac{1}{2}-\alpha).$  By the  similar arguments as those in the proof of local existence, we have that
\vspace{-0.05in}$$
\|u(\cdot,T;0,u_0)\|_{\alpha}\leq CMT^{-\alpha} + CM^2T^{\frac{1}{2}-\alpha -\epsilon} + CM[1+A_0+k_1(A_1+|\Omega|A_2)]T^{1-\alpha}.
\vspace{-0.05in}$$
 Now choose $\nu$ such that $0\leq \nu < 2\alpha -\frac{n}{p}$, then
$X^{\alpha} \subset C^{\nu}(\bar{\Omega})$,
 where the inclusion is continuous. Thus by Arzela-Ascoli Theorem, $\mathcal{T}(T)E(T)$ is precompact.
Therefore by Schauder fixed point theorem, there exists $u^T \in E(T) $ such that $\mathcal{T}(T)u^T=u^T,$ i.e $u(\cdot,T;0,u^T)=u^T(\cdot).$ Since $u(\cdot,t+T;0,u^T)=u(\cdot,t;T,u(\cdot,T;0,u^T))=u(\cdot,t;0,u^T),$ $u(\cdot,t;0,u^T)$ is periodic with period $T.$ Now from the facts that   $u(.,t;0,u^T)$ is periodic with period $T$ and the uniqueness of solutions of
\vspace{-0.05in}\[
\begin{cases}
-\Delta v+v=u(x,t;0,u^T)  \quad x \in \Omega \\
\frac{\p v}{\p n}=0 \quad \text{on} \:\p \Omega,
\end{cases}
\vspace{-0.05in}\]
we get  $v(\cdot,t;0,u^T)=A^{-1}u(\cdot,t;0,u^T)$ is periodic with period $T.$ Then  $(u(\cdot,t;0,u^T),v(\cdot,t;0,u^T))$ is a positive periodic solution of \eqref{u-v-eq1}.

\medskip

 (2) Assume that $a_i(t,x)\equiv a_i(t)$.  Note that in this  case, every solution of the ODE
 \vspace{-0.05in}
 $$u_t=u(a_0(t)-(a_1(t)u+|\Omega|a_2(t))u)
 \vspace{-0.05in}$$
  is a solution
 of the first equation of the system \eqref{u-v-eq1} with Neumann boundary. (2) then follows from Lemma \ref{lem-00001}.

\medskip

 (3) Assume that $a_i(t,x)\equiv a_i(x)$ $(i=0,1,2$). In this case, each $\tau>0$ is a period for $a_i$. By (2), there exist $u^\tau \in E(\tau)$ such that $(u(\cdot,t;0,u^\tau),A^{-1}(u(\cdot,t;0,u^\tau)))$ is a positive  periodic solution of \eqref{u-v-eq1} with period $\tau$.  Note that there is $\tilde M>0$ such that  for each $\tau>0$ and  $u_0 \in E(\tau),$ $\|u(\cdot,t;0,u_0)\|_{\alpha} \leq \tilde M $ for each $1\leq t \leq 2.$ Let $\tau_n=\frac{1}{n},$ then  there exists $u_n \in E(\tau_n)$ such that $u(\cdot,t;0,u_n)$ is periodic with period $\tau_n$ and
\vspace{-0.05in}\begin{equation}
\label{entire-eq5}
\|u_n\|_{\alpha}=\|u(\cdot,\tau_n;0,u_n)\|_{\alpha}=\|u(\cdot, N\tau_n;0, u_n)\|_{\alpha} \leq \tilde M,
\vspace{-0.05in} \end{equation}
 where $N$ is such that $1\leq N\tau _n \leq 2.$

We claim that there is $\delta>0$ such that
\vspace{-0.05in}\begin{equation}
\label{entire-eq6}
\|u_n(\cdot)\|_{C^0(\bar\Omega)} \ge \delta\quad \forall\,\,  n\ge 1.
\vspace{-0.05in}\end{equation}
 Suppose by contradiction that this does not hold. Then there exists $n_k $  such that $\|u_{n_k}\|_{C^0(\bar\Omega)} < \frac{1}{n_k}$ for every  $k \ge 1$.  Let  $k_0$ such that $\frac{1}{n_{k}}<\epsilon_0$ for all $k\ge k_0.$  By the proof of Lemma \ref{lem-009} we get that $u(\cdot,t;0,u_{n_{k}}) \ge u(t;\inf u_{n_k})$ for all $t>0$ and $k\ge k_0,$ where $u(t;\inf u_{n_k})$ is the solution of \eqref{entire-eq3}
with $u(0;\inf u_{n_k})=\inf u_{n_k}$.
Let $\delta_*=\frac{ a_{0,\inf}-\epsilon_0(\chi+|\Omega|\cdot |a_{2,\sup}|)}{2a_{1,\sup}}$ and choose $k$ large enough such that $\frac{1}{n_k}<\delta_*$.
 There is  $t_0>0$ such that $ u(t;\inf u_{n_k})>\delta_*$  for all $t\ge t_0$.  Then we have
\vspace{-0.05in} $$
 u_{n_k}(x)=u(\cdot,m \tau_{n_k};0,u_{n_k})\ge u(m \tau_{n_k};\inf u_{n_k})>\delta^*
\vspace{-0.05in} $$
 for all $m\in \NN$ satisfying that $m\tau_{n_k}>t_0$. This is a contradiction. Therefore,
 \eqref{entire-eq6} holds.

 By \eqref{entire-eq5} and Arzela-Ascoli theorem, there exist $n_k$, $u^* \in C^0(\bar{\Omega})$ such that $u_{n_k}$  converges to $u^*$ in $C^0(\bar{\Omega}).$  By \eqref{entire-eq6},
 $\|u^*(\cdot)\|_{C^0(\bar\Omega)}\ge \frac{\delta}{2}.$
We claim that $(u(\cdot,t;0,u^*),v(\cdot,t;0,u^*))$ with $v(\cdot,t;0,u^*)=A^{-1}u(\cdot,t;0,u^*)$ is a steady state solution of \eqref{u-v-eq1}, that is,
\vspace{-0.05in}\begin{equation}
\label{entire-eq7}
u(\cdot,t;0,u^*)=u^*(\cdot)\quad \text{for all}\,\, t\ge 0.
\vspace{-0,05in}\end{equation}
In fact,
let $\epsilon>0$ be fix and let $t>0$. Note that
 \vspace{-0,05in}
 $$ [n_k  t]\tau_{n_k}=\frac{[n_kt]}{n_k}\leq t \leq \frac{[n_kt]+1}{n_k}=([n_k t]+1)\tau_{n_k}.
 \vspace{-0.05in}$$
By Corollary \ref{cor1}, we can  choose $k$ large enough such that
\vspace{-0.05in}
\[|u(x,t;0,u^*)-u(x,t;0,u_{n_k})|< \epsilon,\quad |u_{n_k}(x)-u^*(x)|< \epsilon,\quad |u(x,\frac{[n_kt]}{n_k};0,u_{n_k})-u(x,t;0,u_{n_k})|< \epsilon\]
for all $x\in\bar\Omega$.
We then have
\vspace{-0,05in}\begin{align*}
|u(x,t;0,u^*)-u^*|&\le |u(x,t;0,u^*)-u(x,t;0,u_{n_k})| +|u(x,t;0,u_{n_k})-u(x,[n_k t]\tau_{n_k};0,u_{n_k})|\\
&\quad +|u_{n_k}(x)-u^*(x)|<3 \epsilon\quad \forall\,\, x\in\bar\Omega.
\vspace{-0.05in}\end{align*}
Letting $\epsilon\to 0$, \eqref{entire-eq7} follows.
\end{proof}

\begin{remark}
\label{remark-h3-2}
It follows from the proof of the existence of entire positive solutions in  Theorem \ref{thm-003} and Remark \ref{remark-h3-1} that the existence of positive entire solutions also holds under the weaker condition {\bf (H2)$^{'}$}.
\end{remark}

\section{Asymptotic Stability of Solutions}

In this section, we investigate the stability and uniqueness of entire positive solutions of \eqref{u-v-eq1}, the asymptotic behavior of global positive solutions of \eqref{u-v-eq1}, and prove Theorems \ref{thm-004} and \ref{thm-005}.  Without loss of generality, we suppose throughout this section that $\chi>0,$ since the arguments for $\chi \leq 0$ are similar.
We first prove Theorem \ref{thm-005}.

\begin{proof} [Proof of Theorem \ref{thm-005}]
Suppose that \eqref{stability-cond-2-eq1} holds.
For given $u_0\in C^0(\bar\Omega)$ with $u_0(x)\ge 0$, $u_0(\cdot)\not =0$, and $t_0\in\RR$, let $(u(\cdot,t;t_0,u_0),v(\cdot,t;t_0,u_0))$ be
the solution of \eqref{u-v-eq1} satisfying the properties in Theorem \ref{thm-001}(2). By Theorem \ref{thm-002}, $(u(\cdot,t;t_0,u_0)$, $v(\cdot,t;t_0,u_0))$ exists for all $t>t_0$. Note that $u(x,t;t_0,u_0)>0$ for all $x\in\bar\Omega$ and $t>t_0$. Without loss of generality, we may assume that
$u_0(x)>0$ for all $x\in\bar\Omega$.

Let $(\bar u(t),\underline u(t))$ be as in \eqref{u-under-above-bar-eq}.
By Lemma \ref{lem-004} and the proof of Theorem \ref{thm-002},
\vspace{-0.05in}\begin{equation}
\label{thm-5-eq1}
\underline u(t)\le u(x,t;t_0,u_0)\le \bar u(t)\quad \forall\,\, x\in\bar\Omega,\,\, t\ge t_0.
\vspace{-0.05in}\end{equation}
Let $r_1$ and $r_2$ be as in \eqref{stability-cond-2-eq2} and \eqref{stability-cond-2-eq3}, respectively.

\medskip

(1) By Lemma \ref{lem-00002}(1) and (2), for any $\epsilon>0$, there is $t_\epsilon>0$ such that
\vspace{-0,05in}\begin{equation}
\label{thm-5-eq2}
r_2-\epsilon\le \underline u(t)\le\bar u(t)\le r_1+\epsilon \quad {\rm for} \quad t\ge t_0+t_\epsilon.
\vspace{-0.05in}\end{equation}
(1) then follows from \eqref{thm-5-eq1} and \eqref{thm-5-eq2}.

\smallskip

(2)  We first consider the case that $a_i(t,x)$ $(i=0,1,2)$ are periodic in $t$ with period $T$.
 By Lemma \ref{lem-00002}(1), (2), and (3), there are periodic functions $m(t)$ and $M(t)$ with period $T$ such that
\vspace{-0.1in}$$
r_2\le m(t)\le M(t)\le r_1\quad \forall\,\, t\in\RR
\vspace{-0.05in}$$
and
for any $\epsilon>0$, there is $t_\epsilon>0$ such that
\vspace{-0.05in}\begin{equation}
\label{thm-5-eq3}
m(t)-\epsilon\le \underline u(t)\le \bar u(t)\le M(t)+\epsilon\quad \forall \,\, t\ge t_0+t_\epsilon.
\vspace{-0.05in}\end{equation}
In this case, (2) then follows from \eqref{thm-5-eq1} and \eqref{thm-5-eq3}.

Next, we consider the cases that $a_i(t,x)$ ($i=0,1,2$) are almost periodic in $t$.
  By Lemma \ref{lem-00002}(1), (2), and (4), there are almost periodic functions $m(t)$ and $M(t)$ such that
$$
r_2 \le m(t)\le M(t)\le r_1\quad \forall\,\, t\in\RR
$$
and  for any $\epsilon>0$, there is $t_\epsilon>0$ such that \eqref{thm-5-eq3} holds.
(2) then follows from \eqref{thm-5-eq1} and \eqref{thm-5-eq3}.
\end{proof}

\vspace{-0.01in}We now prove  Theorem \ref{thm-004}

\vspace{-0.1in}\begin{proof} [Proof of Theorem \ref{thm-004}]

(1) Suppose that $a_i(t,x)\equiv a_i(t)$ for $i=0,1,2$ and
\vspace{-0.05in}\begin{equation}
\label{thm-4-eq0}
\inf_{t\in \mathbb{R}} \big\{a_1(t)-|\Omega|\abs{a_2(t)}\big\} >2\chi.
\vspace{-0.05in}\end{equation}
For given $u_0\in C^0(\bar\Omega)$ with $u_0(x)\ge 0$, $u_0(\cdot)\not =0$, and $t_0\in\RR$, let $(u(\cdot,t;t_0,u_0),v(\cdot,t;t_0,u_0))$ be
the solution of \eqref{u-v-eq1} satisfying the properties in Theorem \ref{thm-001}(2).  Again, by Theorem \ref{thm-002}, $(u(\cdot,t;t_0,u_0)$, $v(\cdot,t;t_0,u_0))$ exists for all $t>t_0$ and without loss of generality, we may assume that
$u_0(x)>0$ for all $x\in\bar\Omega$.

Let $(\bar u(t),\underline u(t))$ be as in \eqref{u-under-above-bar-eq}.
Let $(u^*(t),v^*(t))$ be the unique entire positive spatially homogeneous solution of \eqref{u-v-eq1} in Theorem \ref{thm-003}(2).
By Lemma \ref{lem-004} and the proof of Theorem \ref{thm-002},
\vspace{-0.05in}\begin{equation}
\label{thm-4-eq1}
\underline u(t)\le u(x,t;t_0,u_0)\le \bar u(t)\quad \forall\,\, x\in\bar\Omega,\,\, t\ge t_0.
\vspace{-0.05in}\end{equation}
  By Lemma \ref{lem-00002}(1), for any $\epsilon>0$, there is $t_\epsilon>0$ such that
 \vspace{-0.05in} \begin{equation}
  \label{thm-4-eq2}
  \underline u(t)-\epsilon\le u^*(t)\le \bar u(t)+\epsilon\quad \forall\,\, t\ge t_0+t_\epsilon.
  \vspace{-0.05in}\end{equation}

By \eqref{thm-4-eq1}  and \eqref{thm-4-eq2},  it suffices to show $0 \leq \ln \frac{\overline{u}(t)}{\underline{u}(t)} \longrightarrow 0  \, \text{as } t \to \infty.$ Assume  that $t>t_0$. By dividing the first equation of $\eqref{ode0}$ by $\overline{u},$ and the second by $\underline{u},$ we get
\vspace{-0.05in}\begin{equation*}
\begin{cases}
 \frac{\overline{u}'}{\overline{u}}=  \left[ a_0(t)-( a_1(t)-|\Omega| (a_2(t))_- -\chi) \overline{u}-( |\Omega| (a_2(t))_+ + \chi )\underline {u} \right] \cr
 \frac{\underline{u}'}{\underline{u}}=\left[a_0(t)-(a_1(t)-|\Omega|(a_2(t))_--\chi) \underline{u}-(|\Omega|(a_2(t))_+ +\chi)\overline {u}\right]
 \end{cases}
 \vspace{-0.05in}\end{equation*}
This together with \eqref{thm-4-eq0} implies that
\vspace{-0.05in} $$\frac{d}{dt} \Big(       \ln \frac{\overline{u}}{\underline{u}}      \Big)=
\frac{\overline{u}'}{\overline{u}} -   \frac{\underline{u}'}{\underline{u}}= -\left( a_1(t)-|\Omega|\abs{a_2(t)}-2\chi \right)  \left(  \overline {u}-\underline{u} \right) \le 0.
 \vspace{-0.05in}$$
Thus by integrating over $(t_0,t),$ we get
\vspace{-0.05in}$$
   0 \leq    \ln \frac{\overline{u}}{\underline{u}} \leq \ln \frac{\overline{u}_0}{\underline{u}_0}, \quad{\rm and\,\,\,  then}\quad  \frac{\overline{u}(t)}{\underline{u}(t)}  \leq  \frac{\overline{u}_0}{\underline{u}_0}.
\vspace{-0.05in}$$
 We have by mean value theorem that
  \vspace{-0.05in}
  $$\overline{u}-  \underline{u}=e^{\ln \overline{u}}   - e^{\ln \underline{u}}=e^{\ln \hat{u}} \Big(     \ln \frac{\overline{u}}{\underline{u}} \Big)=\hat{u} \Big(     \ln \frac{\overline{u}}{\underline{u}} \Big)       ,
  \vspace{-0.05in}$$
where $\underline{u} \leq \hat{u} \leq  \overline{u}.$ Therefore
\vspace{-0.05in}\[ \frac{d}{dt} \Big(       \ln \frac{\overline{u}}{\underline{u}}      \Big) \leq      -\Big( a_1(t)-|\Omega|\abs{a_2(t)}-2\chi \Big)  \Big(   { \inf_{t\ge t_0}} \overline{u}(t)   \frac{\underline{u}_0}{\overline{u}_0}     \Big)        \Big(     \ln \frac{\overline{u}}{\underline{u}} \Big).
\vspace{-0.05in} \]
By letting $\epsilon_0=\inf_{t\in\RR}\{a_1(t)-|\Omega|\abs{a_2(t)}-2\chi\}   \left(   \inf_{t\ge t_0} \overline{u}(t)   \frac{\underline{u}_0}{\overline{u}_0}     \right),$ we have $\epsilon_0>0$ and
 \vspace{-0.05in}\[   0\leq      \ln \frac{\overline{u}}{\underline{u}} \leq   \ln \frac{\overline{u}_0}{\underline{u}_0} e^{-\epsilon_0 t}   \to 0\quad {\rm as}\quad t\to\infty.
 \vspace{-0.05in}  \]

\smallskip

(2) Let $L_1(t)$ and $L_2(t)$ be as in \eqref{L-eq1} and \eqref{L-eq2}, respectively. By \eqref{stability-cond-2-eq2},
 \vspace{-0.05in}
 $$\mu=\limsup_{t-s\to\infty}\frac{1}{t-s}\int_s^t(L_1(\tau)- L_2(\tau))d\tau <0.
 \vspace{-0.05in}$$
 Fix $0<\epsilon<-\mu.$
 Let $r_1$ and $r_2$ be as in \eqref{r-eq1} and \eqref{r-eq2}, respectively. By \eqref{stability-cond-2-eq1},
   Theorem  \ref{thm-005}(1), and definition of $\mu$, for any $\epsilon>0$,  there exists $T_\epsilon>0$ such that
 \vspace{-0.05in}$$r_2-\epsilon \le u(\cdot,t_0+t;t_0;u_0)\le r_1+\epsilon, \, \,\,  r_2-\epsilon \le u^*(x,t) \le r_1+\epsilon\,\, \,  \forall x\in\bar\Omega,\, \,  t\ge t_0+T_\epsilon,
 \vspace{-0.05in}$$
 and
\vspace{-0.05in}$$\int_{t_0}^{t_0+t}(L_1(s)- L_2(s))ds \le (\mu+\epsilon)t, \,  \forall\,\, t_0\in\RR,\,\,  t\ge t_0+ T_\epsilon.
\vspace{-0.05in}$$

We first prove that for any entire positive solution $(u^*(x,t),v^*(x,t))$ of \eqref{u-v-eq1}, \eqref{global-stability-2-eq1} holds.
To simplify the notation, for given $t_0\in\RR$ and $u_0\in C^0(\bar\Omega)$ with $u_0(x)\ge 0$ and $u_0(\cdot)\not =0$,
set $u(t)=u(\cdot,t;t_0;u_0)$ and $u^*(t)=u^*(\cdot,t).$
 Let $w(t)=u(t)-u^*(t).$ Then $w$ satisfy the equation
\vspace{-0.05in}\begin{align}
\label{new-eqq1}
w_t=&\Delta w-\chi\nabla (w\cdot \nabla A^{-1}u)-\chi\nabla (u^*\cdot \nabla A^{-1}w) +w\left(a_0(t,x)-a_1(t,x)(u+u^*\right)\nonumber\\
&-a_2(t,x)\int_{\Omega}u)-a_2(t,x)\big(\int_{\Omega}w\big)u^*
\vspace{-0.05in}\end{align}
for $t>t_0$.
 By the similar arguments for \eqref{new-eq2}, we have that $\int_{\Omega}w_+^2$  is weakly differentiable and moreover  $$\frac{d}{dt}\int_{\Omega}w_+^2=2\int_{\Omega}w_+w_t\quad \forall\,\, a.e.\, t>t_0.$$

 Next, by  multiplying \eqref{new-eqq1} by $w_{+}$ and integrating it over $\Omega$, we get
\vspace{-0.05in}\begin{multline*}
\frac{1}{2}\frac{d}{dt}\int_{\Omega}w_{+}^2+\int_{\Omega}|\nabla w_{+}|^2=\chi\int_{\Omega}w_{+}\nabla w_{+} \cdot \nabla A^{-1}u+\chi\int_{\Omega}u^*\nabla w_{+} \cdot \nabla A^{-1}w\\
+\int_{\Omega}w_{+}^2\big(a_0(t,x)-a_1(t,x)(u+u^*)-a_2(t,x)\int_{\Omega}u\big)-\big(\int_{\Omega}w\big)\int_{\Omega}a_2(t,x)u^*w_{+}
\vspace{-0.05in}\end{multline*}
for  a.e  $t>t_0$.
Integrating by part and using the equation of $A^{-1}u,$ we get for  a.e  $t> t_0+T_\epsilon$
\vspace{-0.05in}\begin{align*}
\frac{1}{2}\frac{d}{dt}\int_{\Omega}w_{+}^2+\int_{\Omega}|\nabla w_{+}|^2&= \frac{\chi}{2}\int_{\Omega}w_{+}^2 (u- A^{-1}u)+\chi\int_{\Omega}u^*\nabla w_{+} \cdot \nabla A^{-1}w\\
&\,\,\quad  +\int_{\Omega}w_{+}^2\big(a_0(t,x)-a_1(t,x)(u+u^*)-a_2(t,x)\int_{\Omega}u\big)\\
&\,\,\quad -\big(\int_{\Omega}w\big)\int_{\Omega}a_2(t,x)u^*w_{+}\\
&\leq \frac{\chi}{2}\int_{\Omega}w_{+}^2 (u- A^{-1}u)+\chi\int_{\Omega}u^*\nabla w_{+} \cdot \nabla A^{-1}w\\
&\,\,\quad +\int_{\Omega}w_{+}^2\big(a_{0,\sup}-a_{1,\inf}(t)(u+u^*)-(a_{2,\inf}(t))_+\int_{\Omega}u+(a_{2,\inf}(t))_-\int_{\Omega}u\big)\\
&\,\,\quad -\big(\int_{\Omega}w\big)\int_{\Omega}a_2(t,x)u^*w_{+}.
\vspace{-0.05in}\end{align*}
We have by Young's inequality  that
\vspace{-0.1in}
$$\chi\int_{\Omega}u^*\nabla w_{+} \cdot \nabla A^{-1}w \leq \int_{\Omega}|\nabla w_{+}|^2+\frac{(\chi ( r_1+\epsilon))^2}{4}\int_{\Omega}|\nabla A^{-1}w|^2.
\vspace{-0.1in}$$
Using the equation of $A^{-1}u$,  we get
\vspace{-0.1in}$$\int_{\Omega}|\nabla A^{-1}w|^2 \leq \int_{\Omega}w_{+}^2+\int_{\Omega}w_{-}^2
\vspace{-0.1in}
$$
for $t>t_0$.
Also we have for $t>t_0+T_\epsilon$ that
\vspace{-0.1in}\begin{align*}
& -\int_{\Omega}w  \int_{\Omega}a_2(t,x)u^*w_{+}\\
&\le  \int_\Omega w_-  \int_{\Omega}{ \big(a_{2,\sup}(t)\big)_+}u^*w_{+}-(a_{2,\inf}(t))_+\int_\Omega w_+  \int_{\Omega}u^*w_{+}+(a_{2,\inf}(t))_-\int_\Omega w_+  \int_{\Omega}u^*w_{+}\\
 &\le [(r_1+\epsilon)(a_{2,\inf}(t))_- -( r_2-\epsilon) ( a_{2,\inf}(t))_+ ](\int_\Omega w_+)^2 + (r_1+\epsilon){ \big(a_{2,\sup}(t)\big)_+} (\int_\Omega w_-)(\int_\Omega w_+).
\vspace{-0.1in}\end{align*}
By combining all these inequalities we  have for a.e  $t>t_0+T_\epsilon$ that
\begin{align}
\label{proof-thm4-eq1}
\frac{1}{2}\frac{d}{dt}\int_{\Omega}w_{+}^2 \leq &  \Big(a_{0,\sup}(t)+\frac{\chi}{2}\big((r_1+\epsilon)-(r_2-\epsilon) \big)+\frac{(\chi (r_1+\epsilon))^2}{4}\Big)\int_{\Omega}w_{+}^2\nonumber\\
&-\Big((r_2-\epsilon)(2a_{1,\inf}(t)+|\Omega|(a_{2,\inf}(t))_+)\Big) \int_{\Omega}w_{+}^2\nonumber\\
&+2|\Omega|(r_1+\epsilon)(a_{2,\inf}(t))_-\int_\Omega w_+^2 +\big(\frac{(\chi (r_1+\epsilon))^2}{4}\big)\int_{\Omega}w_{-}^2\nonumber\\
& -( r_2-\epsilon) ( a_{2,\inf}(t))_+  (\int_\Omega w_+)^2 + (r_1+\epsilon){ \big(a_{2,\sup}(t)\big)_+} (\int_\Omega w_-)(\int_\Omega w_+).
\end{align}

Similarly we have that  $\int_{\Omega}w_-^2$  is weakly differentiable with  $\frac{d}{dt}\int_{\Omega}w_-^2=-2\int_{\Omega}w_-w_t,$ and   for  a.e  $t>t_0+T_\epsilon$
\begin{align}
\label{proof-thm4-eq2}
\frac{1}{2}\frac{d}{dt}\int_{\Omega}w_{-}^2 \leq & \Big(a_{0,\sup}(t)+\frac{\chi}{2}\big((r_1+\epsilon)-(r_2-\epsilon) \big)+\frac{(\chi (r_1+\epsilon))^2}{4}\Big) \int_{\Omega}w_{-}^2\nonumber\\
&-\Big((r_2-\epsilon)(2a_{1,\inf}(t)+|\Omega|{ \big(a_{2,\inf}(t)\big)_+})\Big) \int_{\Omega}w_{-}^2\nonumber\\
&+2|\Omega|(r_1+\epsilon)(a_{2,\inf}(t))_-\int_\Omega w_-^2 +\big(\frac{(\chi (r_1+\epsilon))^2}{4}\big)\int_{\Omega}w_{+}^2 \nonumber\\
&  -( r_2-\epsilon) ( a_{2,\inf}(t))_+  (\int_\Omega w_-)^2 + (r_1+\epsilon) { \big(a_{2,\sup}(t)\big)_+} (\int_\Omega w_-)(\int_\Omega w_+).
\end{align}

Note that
\vspace{-0.05in}\begin{align*}
&- (r_2-\epsilon) (a_{2,\inf}(t))_+ \Big((\int_\Omega w_+)^2+(\int_\Omega w_-)^2\Big) + 2 (r_1+\epsilon) { \big(a_{2,\sup}(t)\big)_+}(\int_\Omega w_-)(\int_\Omega w_+)\\
&\le 2\Big( (r_1+\epsilon) { \big(a_{2,\sup}(t)\big)_+} -(r_2-\epsilon) ( a_{2,\inf}(t))_+\Big) (\int_\Omega w_-)(\int_\Omega w_+)\\
&\le |\Omega|\Big[\epsilon\Big({ \big(a_{2,\sup}(t)\big)_+} +(a_{2,\inf}(t))_+\Big) +\Big( r_1{\big(a_{2,\sup}(t)\big)_+}-r_2(a_{2,\inf}(t))_+\Big)\Big]
\Big(\int_\Omega w_-^2+\int_\Omega w_+^2\Big).
\vspace{-0.05in}\end{align*}
Set
 \vspace{-0.05in}$$K(t,\epsilon)=\chi \epsilon+\chi^2\frac{\epsilon}{2}(2r_1+\epsilon)+2|\Omega|\epsilon(a_{2,\inf})_-+|\Omega|\epsilon(a_{2,\sup}(t) +(a_{2,\inf}(t))_+)+\epsilon(2a_{1,\inf}(t)+|\Omega|a_{2,\inf}(t)).
 \vspace{-0.05in}$$
Adding \eqref{proof-thm4-eq1} and \eqref{proof-thm4-eq2}, we then have
\vspace{-0.05in}\begin{align*}
\frac{d}{dt}\int_{\Omega}(w_{+}^2+w_{-}^2)(t) \leq& { 2\Big\{L_2(t)-L_1(t)+K(t,\epsilon)\Big\}\int_{\Omega}(w_{+}^2+w_{-}^2)}
\vspace{-0.05in}\end{align*}
for  a.e $t>t_0+T_\epsilon$. Therefore by the continuity with respect to time  of both sides of this last inequality, we get
\vspace{-0.05in}\begin{align*}
\frac{d}{dt}\int_{\Omega}(w_{+}^2+w_{-}^2)(t) \leq& { 2\Big\{L_2(t)-L_1(t)+K(t,\epsilon)\Big\}\int_{\Omega}(w_{+}^2+w_{-}^2)}
\vspace{-0.05in}\end{align*}
for  $t>t_0+T_\epsilon.$ Then by Gronwall's inequality,
\vspace{-0.05in}$$\int_{\Omega}(w_{+}^2(t)+w_{-}^2(t)) \leq \int_{\Omega}(w_{+}^2(t_0+T_\epsilon)+w_{-}^2(t_0+T_\epsilon)) e^{2\int_{t_0}^t(L_1(s)-L_2(s)+K(s,\epsilon))ds}  \quad \text{for all $t>t_0+T_\epsilon$.}
\vspace{-0.05in}$$

Note that $0\leq \sup_{t\in\RR}\abs{K(t,\epsilon)} \to 0$ as $\epsilon \to 0$ and choose $\epsilon_0 \ll 1$ ($\epsilon_0<-\mu$) such  that
 \vspace{-0.05in}
 $$0\leq \sup_{t\in\RR}\abs{K(t,\epsilon)}<\frac{-\mu-\epsilon_0}{2}.
 \vspace{-0.05in}$$
By $\int_{t_0}^t(L_1(s)- L_2(s))ds \le (\mu+\epsilon_0)(t-t_0)$ for $t\ge t_0+T_{\epsilon_0}$,  we have
\vspace{-0.05in}
\begin{align*}
\int_{\Omega}(w_{+}^2(t)+w_{-}^2(t))& \leq (\int_{\Omega}w_{+}^2(t_0+T_{\epsilon_0})+w_{-}^2(t_0+T_{\epsilon_0})) e^{2(\mu+\epsilon_0)(t-t_0)} e^{2(\frac{-\mu-\epsilon_0}{2})(t-t_0)} \\
&\leq (\int_{\Omega}w_{+}^2(t_0+T_{\epsilon_0})+w_{-}^2(t_0+T_{\epsilon_0})) e^{(\mu+\epsilon_0)(t-t_0)}   \quad \forall \,\, t>t_0+T_{\epsilon_0}.
\vspace{-0.05in}\end{align*}
Therefore
\vspace{-0.08in}\begin{equation}
\label{thm-4-eq3}
\lim_{t \to \infty}\|u(\cdot,t+t_0;t_0,u_0)-u^*(\cdot,t+t_0)\|_{L^2(\Omega)}=\lim_{t \to \infty}\|w(t+t_0)\|^2_{L^2(\Omega)}=0
\vspace{-0.05in}\end{equation}
uniformly in $t_0\in\RR$.

We claim that \eqref{global-stability-2-eq1} holds.
Suppose by contradiction that there is $t_0\in\RR$ such that $u(\cdot,t;t_0,u_0) \nrightarrow u^*(\cdot,t)$ in $ C^0(\bar\Omega) $ as  $t \to \infty.$
Then there exists $\epsilon_0>0$ and a sequence $t_n \to \infty$ as $n \to \infty$  such that \[\|u(\cdot,t_n;t_0,u_0)-u^*(\cdot,t_n)\|_{C^0(\bar\Omega)}>\epsilon_0.\] Since $u(\cdot,t_n;t_0,u_0) ,u^*(\cdot,t_n) \in C^0(\bar{\Omega})$ are uniformly bounded and equicontinuous,
 there exists up to subsequence $u^1,u_*^1 \in C^0(\bar{\Omega})$ such that $u(\cdot,t_n;t_0,u_0),\, u^*(\cdot,t_n)$ converges respectively  to $u^1,\, u_*^1$  in $ C^0(\bar{\Omega}).$ Therefore by dominated convergence theorem, $u(\cdot,t_n;t_0,u_0) \to u_1$  and $u^*(\cdot,t_n) \to u_*^1$ in $L^{2}(\Omega)$ as $t \to \infty.$  This implies that
\vspace{-0.05in}$$
\lim_{t_n \to \infty}\|u(\cdot,t_n; t_0,u_0)-u^*(\cdot,t_n)\|_{L^2(\Omega)}=0.
\vspace{-0.05in}$$
Hence
 we have  that $u^1=u_*^1.$ But also from
 $\|u(\cdot,t_n; t_0,u_0)-u^*(\cdot,t_n)\|_{C^0(\bar\Omega)}>\epsilon_0$,
  we get as $n \to \infty$, $\|u^1-u^*_1\|_{C^0(\bar\Omega)}\geq\epsilon_0$,
which is a contradiction. Hence \eqref{global-stability-2-eq1} holds.

Next, we prove that \eqref{u-v-eq1} has a unique entire positive solution. Suppose that $(u^*_1(x,t),v^*_1(x,t))$ and
$(u^*_2(x,t),v^*_2(x,t))$ are two entire positive solutions of \eqref{u-v-eq1}. We claim that  $(u^*_1(x,t),v^*_1(x,t))\equiv (u^*_2(x,t),v^*_2(x,t))$
for any $t\in\RR$.  Indeed, fix any $t\in\RR$, by the arguments in the proof of \eqref{thm-4-eq3},
\vspace{-0.05in}\begin{align*}
\|u^*_1(\cdot,t)-u^*_2(\cdot,t)\|_{L^2(\Omega)}=\|u(\cdot,t;t_0,u_1^*(\cdot,t_0))-u(\cdot,t;t_0,u_2^*(\cdot,t_0))\|_{L^2(\Omega)}\to 0\quad {\rm as}\quad t_0\to -\infty.
\vspace{-0.05in}\end{align*}
This together with the continuity of $u_i^*(x,t)$ ($i=1,2$) implies that $u^*_1(x,t)\equiv u_2^*(x,t)$ and then $v_1^*(x,t)\equiv v_2^*(x,t)$. Hence \eqref{u-v-eq1} has a unique entire positive solution.

Assume now that $a_i(t,x)\equiv a_i(x)$ ($i=0,1,2$). By Theorem \ref{thm-003}(3) and the uniqueness of entire positive solutions of \eqref{u-v-eq1},
\eqref{u-v-eq1} has a unique positive steady state solution.

Assume that $a_i(t+T,x)=a_i(t,x)$ ($i=0,1,2$). By Theorem \ref{thm-003}(1) and the uniqueness of entire positive solutions of \eqref{u-v-eq1},
\eqref{u-v-eq1} has a unique positive periodic solution with period $T$.

Finally assume  that $a_i(t,x)$ $(i=0,1,2)$ are almost periodic in $t$ uniformly with respect to $x\in\bar\Omega$. Let $(u^*(x,t),v^*(x,t))$ be the unique positive solution of \eqref{u-v-eq1}.
We claim that $(u^*(x,t),v^*(x,t))$ is almost periodic in $t$. Indeed,
for any  sequences $\{\beta_n^{'}\}$, $\{\gamma_n^{'}\}\subset \RR$, by the almost periodicity of $a_i(t,x)$ in $t$, there are  subsequences
$\{\beta_n\}\subset \{\beta_n^{'}\}$
and $\{\gamma_n\}\subset \{\gamma_n^{'}\}$ such that
\vspace{-0.05in}$$
\lim_{m\to\infty}\lim_{n\to\infty} a_i(t+\beta_n+\gamma_m,x)=\lim_{n\to\infty} a_i(t+\beta_n+\gamma_n,x)
\vspace{-0.05in}$$
uniformly in $t\in\RR$ and $x\in\bar\Omega$ for $i=0,1,2$. Let
\vspace{-0.05in}$$
\hat a_i(t,x)=\lim_{n\to\infty} a_i(t+\beta_n,x),\quad \check a_i(t,x)=\lim_{m\to\infty} \hat a_i(t+\gamma_m,x),\quad \tilde a_i(t,x)=\lim_{n\to\infty} a_i(t+\beta_n+\gamma_n,x)
\vspace{-0.05in}$$
for $i=0,1,2$. Observe that $\hat a_i$ ($i=0,1,2$), $\check a_i$ $(i=0,1,2)$, and $\tilde a_i$ ($i=0,1,2)$ also satisfy the hypothesis (H1)
in the introduction, and $\check a_i=\tilde a_i$ for $i=0,1,2$.

Without loss of generality, we may assume that $\lim_{n\to\infty}(u^*(\cdot,t+\beta_n),v^*(\cdot,t+\beta_n))$ exists in $C^0(\bar\Omega)$. Let
$$
(\hat u^*(x,t),\hat v^*(x,t))=\lim_{n\to\infty}(u^*(\cdot,t+\beta_n),v^*(\cdot,t+\beta_n)).
$$ Then $(\hat u^*(x,t),\hat v^*(x,t))$ is an entire positive solution of \eqref{u-v-eq1} with $a_i(t,x)$ being replaced by $\hat a_i(t,x)$ $(i=0,1,2$).

We may also assume that $\lim_{n\to\infty}(\hat u^*(\cdot,t+\beta_n),\hat v^*(\cdot,t+\beta_n))$ exists in $C^0(\bar\Omega)$. Let
\vspace{-0.05in}$$
(\check u^*(x,t),\check v^*(x,t))=\lim_{n\to\infty}(\hat u^*(\cdot,t+\beta_n),\hat v^*(\cdot,t+\beta_n)).
\vspace{-0.05in}$$ Then $(\check u^*(x,t),\check v^*(x,t))$ is an entire positive solution of \eqref{u-v-eq1} with $a_i(t,x)$ being replaced by $\check a_i(t,x)$ $(i=0,1,2$).

 Furthermore, we  may assume that $\lim_{n\to\infty}(\hat u^*(\cdot,t+\beta_n+\gamma_n),\hat v^*(\cdot,t+\beta_n+\gamma_n))$ exists in $C^0(\bar\Omega)$. Let
\vspace{-0.05in}$$
(\tilde u^*(x,t),\tilde v^*(x,t))=\lim_{n\to\infty}(\hat u^*(\cdot,t+\beta_n+\gamma_n),\hat v^*(\cdot,t+\beta_n+\gamma_n)).
\vspace{-0.05in}$$ Then $(\tilde u^*(x,t),\tilde v^*(x,t))$ is an entire positive solution of \eqref{u-v-eq1} with $a_i(t,x)$ being replaced by $\tilde a_i(t,x)$ $(i=0,1,2$).
By the uniqueness of entire positive solutions of \eqref{u-v-eq1} with $a_i(t,x)$ being replaced by $\tilde a_i(t,x)$ $(i=0,1,2$), we have that
\vspace{-0.05in}$$
(\tilde u^*(x,t),\tilde v^*(x,t))=(\check u^*(x,t),\check v^*(x,t))\quad \forall\,\, x\in\bar\Omega,\,\, t\in\RR.
\vspace{-0.05in}$$
It then follows from $\check a_i=\tilde a_i$ for $i=0,1,2$ that
\vspace{-0.05in}$$
\lim_{m\to\infty}\lim_{n\to\infty}(u^*(x,t+\beta_n+\gamma_m),v^*(x,t+\beta_n+\gamma_m))=\lim_{n\to\infty}(u^*(x,t+\beta_n+\gamma_n),v^*(x,\beta_n+\gamma_n))
\vspace{-0.05in}$$
and hence $(u^*(x,t),v^*(x,t))$ is almost periodic in $t$. The theorem is thus proved.
\end{proof}

\begin{remark}
\label{asym-remark-00}
For the proof of Theorem \ref{thm-004} (2) in the general case of $\chi\in \mathbb{R},$ we need to see that
\begin{align*}
&\chi\int_{\Omega}w_{+}^2 (u- A^{-1}u)\nonumber\\
&=(\chi)_+\int_{\Omega}w_{+}^2 (u- A^{-1}u)-(\chi)_-\int_{\Omega}w_{+}^2 (u- A^{-1}u)\nonumber\\
&\leq (\chi)_+\left((r_1+\epsilon)-(r_2-\epsilon)\right)\int_{\Omega}w_{+}^2+(\chi)_-\left(-(r_2-\epsilon)+(r_1+\epsilon)\right)\int_{\Omega}w_{+}^2=\chi\left((r_1+\epsilon)-(r_2-\epsilon)\right)\int_{\Omega}w_{+}^2
\end{align*}
and the proof then follows as in the case $\chi>0.$
\end{remark}

\end{document}